\documentclass[12pt]{amsart}

\usepackage[OT1]{fontenc}
\usepackage{amsmath, mathrsfs, amsfonts, amssymb, geometry, enumerate,
  caption, fullpage, subcaption, multicol}
\usepackage[usenames,dvipsnames]{color}
\usepackage{graphicx}
\usepackage[numbers]{natbib}
\usepackage{pdfsync}

\usepackage{nameref}
\usepackage{zref-xr,zref-user} 
\zxrsetup{toltxlabel}

\usepackage{tikz}
\usetikzlibrary{arrows, calc}
\usetikzlibrary{patterns, shapes, decorations.markings}

\usepackage[colorlinks,citecolor=blue,urlcolor=blue]{hyperref}
\usepackage{pgfplots}
\pgfplotsset{/pgf/number format/use comma,compat=newest}
\usetikzlibrary{external}
\definecolor{labelkey}{rgb}{0,0,1}

\numberwithin{equation}{section}

\numberwithin{equation}{section}
\theoremstyle{definition}
\newtheorem{definition}[equation]{Definition}
\theoremstyle{definition}
\newtheorem{remark}[equation]{Remark}
\theoremstyle{definition}

\theoremstyle{definition}

\theoremstyle{lemma}

\theoremstyle{lemma}
\newtheorem{lemma}[equation]{Lemma}
\theoremstyle{theorem}
\newtheorem*{lemma*}{Lemma}
\theoremstyle{theorem}
\newtheorem{theorem}[equation]{Theorem}
\theoremstyle{proposition}
\newtheorem{proposition}[equation]{Proposition}
\theoremstyle{corollary}
\newtheorem{corollary}[equation]{Corollary}

\theoremstyle{corollary}
\theoremstyle{definition}
\newtheorem{example}[equation]{Example}
\theoremstyle{example}
\theoremstyle{definition}
\newtheorem{assump}[equation]{Assumption}
\theoremstyle{definition}

\theoremstyle{theorem}

\theoremstyle{definition}

\DeclareMathOperator*{\spa}{span}
\DeclareMathOperator*{\supp}{supp}

\newcommand{\ccdot}{\,\cdot\,}

\newcommand{\bk}{\mathbf{k}}

\newcommand{\bj}{\mathbf{j}}

\newcommand{\R}{\mathbf{R}}
\newcommand{\PP}{\mathbf{P}}

\newcommand{\E}{\mathbf{E}}
\renewcommand{\P}{\mathbf{P}}
\newcommand{\C}{\mathbf{C}}
\newcommand{\N}{\mathbf{N}}
\newcommand{\Z}{\mathbf{Z}}

\newcommand{\tth}{\tilde{\theta}}
\newcommand{\tom}{\tilde{\Vort}}

\newcommand{\NS}{\Phi}

\newcommand{\TT}{\mathbf{T}}
\newcommand{\ZZ}{\Z}
\newcommand{\RR}{\mathbf{R}}

\renewcommand{\aa}{\mathbf{a}}

\renewcommand{\aa}{\hat{\mathbf{a}}}
\newcommand{\aaa}{\mathbf{A}}
\newcommand{\Vort}{\xi}
\newcommand{\VortS}{\bar{\xi}}
\newcommand{\bfU}{\mathbf{u}}
\newcommand{\bfw}{\mathbf{w}}
\newcommand{\bfg}{\mathbf{g}}
\newcommand{\bfh}{\mathbf{h}}
\newcommand{\bfe}{\mathbf{e}}
\newcommand{\bfr}{\mathbf{r}}
\newcommand{\bfs}{\mathbf{s}}

\newcommand{\pd}{\partial}
\newcommand{\Th}{\theta}
\newcommand{\Zb}{\mathcal{Z}}

\newcommand{\SGf}{\Phi}
\newcommand{\bff}{\mathbf{f}}

\newcommand{\Bvl}{\bar{\Vort}_\lambda}
\newcommand{\Btl}{\bar{\Th}_\lambda}
\newcommand{\rhV}{\rho^\Vort}
\newcommand{\rhT}{\rho^\Th}
\newcommand{\gmV}{\gamma^\Vort}
\newcommand{\gmT}{\gamma^\Th}

\newcommand{\SRD}{\Phi}

\newcommand{\indFn}[1]{1 \! \! 1_{#1}}

\newcommand{\eps}{\varepsilon}
\newcommand{\bfV}{\mathbf{v}}

\newcommand{\death}{\raisebox{-0.15em}{\includegraphics[width=3.3mm]{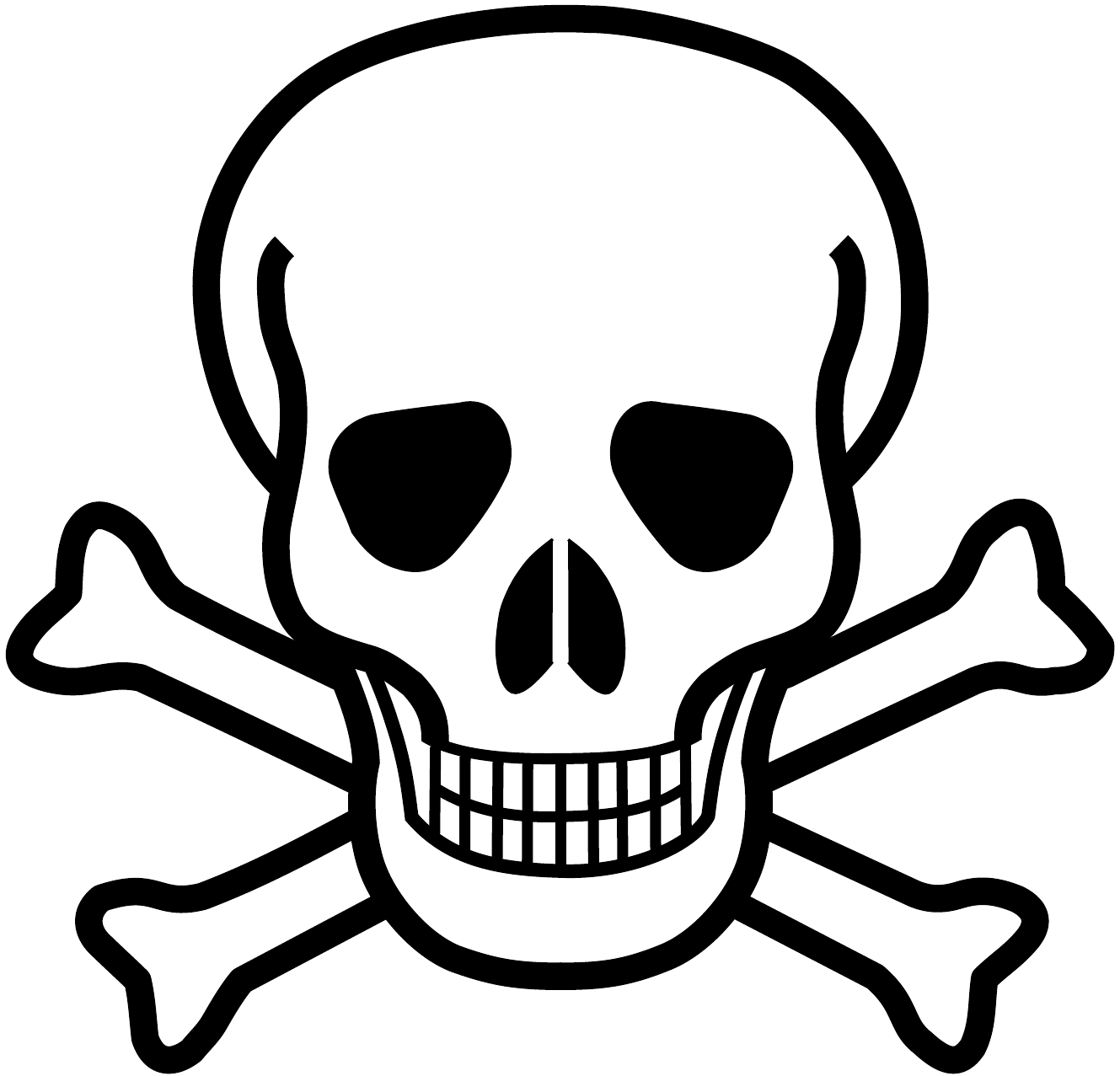}}}
\newcommand{\one}{\mathbf{1}}

\title[Scaling and Saturation in Infinite Dimensions]{Scaling and Saturation in Infinite-Dimensional Control Problems with Applications to Stochastic Partial Differential Equations}
\author[Nathan E. Glatt-Holtz, David P. Herzog, Jonathan C. Mattingly]{Nathan E. Glatt-Holtz, David P. Herzog, Jonathan C. Mattingly\\
  \scriptsize{emails: negh@tulane.edu, dherzog@iastate.edu, jonm@math.duke.edu}}

\begin{document}

\begingroup
\def\uppercasenonmath#1{} 
\let\MakeUppercase\relax 
\maketitle
\endgroup

\begin{abstract}
We establish the dual notions of scaling and saturation from geometric control
theory in an infinite-dimensional setting.  This generalization is applied 
to the low-mode control problem in a number of concrete nonlinear 
partial differential equations.  We also develop applications
concerning associated classes of stochastic partial differential equations (SPDEs).
In particular, we study the support properties of probability laws
corresponding to these SPDEs as well as provide applications concerning
the ergodic and mixing properties of invariant measures for these 
stochastic systems.

\vspace{.05in}

\vspace{.05in}
\end{abstract}

\setcounter{tocdepth}{1}
\tableofcontents

\newpage

\section{Introduction}
\label{sec:intro}

The main goal of this paper is to develop a flexible framework for establishing controllability in important classes of nonlinear partial differential equations.  Our primary motivation for doing this stems from our interest in obtaining support properties of the Markov process solving the associated stochastic partial differential equation (SPDE) that results when the control terms are replaced by independent Brownian motions.  Such support properties constitute a fundamental part in proving unique ergodicity and mixing of the stochastic dynamics.  An additional, but more refined goal of this work is to give practical criteria for the existence and positivity of the density of the law of the projected SPDE solution onto any finite-dimensional subspace.

General support theorems, or equivalently topological irreducibility results, for
SPDEs were initially restricted to settings where the noise is
sufficiently non-degenerate, allowing one to instantaneously counteract any effect of the
drift terms (see \cite{ZabczykDaPrato1992, ZabczykDaPrato1996, flandoli1995} and the references
therein).  While such roughly forced SPDEs arise naturally in the equations
describing the statistics of fluctuations in various scaling limits,
they do not cover many interesting examples.  Here we are largely motivated by applications to statistical hydrodynamics and phase field equations. In these settings, the stochastic forcing is
typically localized at a certain scale and one is interested in how
the dynamics propagates the energy to other scales.  Consequently, we focus on the case when randomness enters the equation externally on a few, select directions in the phase space.

In this degenerate setting, many of the initial approaches to solving the control problem were only sufficient to imply irreducibility and not global controllability.  Furthermore, the methods used were rather ad hoc~\cite{HM06, HM09}.  On the other hand, systematic results close to the setting of this paper were given previously in~\cite{AgrachevSarychev2005, 
AgrachevSarychev2006,  Shirikyan2007, Shirikyan2008, Shirikyan2010, Nersisyan2010,
  Nersesyan2015}.  While related, our results more directly extend the
geometric control theory work of Jurdjevic and Kupka~\cite{JK_81,
  JK_85, Jur_97}.  One advantage is that this approach
interfaces cleanly the local smoothing/contraction estimates coming from the infinite-dimensional version of hypoellipticity~\cite{Hormander1967} developed in~\cite{HM09,HM06}. Both theories
are built on Lie-bracket calculations and a flag of associated
subspaces which capture the ability of the nonlinearity to move
randomness and control action between the degrees of freedom, or rather different scales, in
the SPDE setting. To prove these results, we recast the `method of
saturation' from~\cite{JK_81, JK_85, Jur_97} into a form suitable for the
infinite-dimensional setting. We also make our results applicable to dynamics, which need only exist locally in time, on a very general phase space $X$. In particular, our
formalism is well adapted to infinite-dimensional systems generated by
PDEs which might only have a local existence theory, such as the 3D
Euler equations.  As such, we think that the control theoretic contributions in
this work hold independent interest beyond the immediate
probabilistic applications which motivated us.

The connection between the control problem and support properties of the associated stochastic dynamics is well established~\cite{StroockVaradhan1972a, StroockVaradhan1972, AK87, ZabczykDaPrato1992, ZabczykDaPrato1996}.  More precisely, fixing the initial condition $u_0$, the solution $u(t)$ to an SPDE at time $t$ is obtained
by the solution map $\phi_t: \omega \mapsto u(t)$ where $\omega$ is a
realization of the stochastic forcing, in our setting a finite collection of
independent Brownian motions. We will assume that $u(t) \in X$ for
some function space $X$ and generally take $\omega \in
\Omega_t=C([0,t],\R^m)$ for finite $m$. Approximate
controllability is then simply the statement that for any $u_0, v \in X$ and
$\delta, \, t>0$ there exists an
$\omega \in \Omega_t$ so that $\|\phi_t(u_0, \omega)-v\| \leq \delta$.

As already mentioned, a second but related goal of the paper is to
give conditions guaranteeing the positivity of the density of the 
random variable $\pi u(t)$, where  $\pi$ is a
continuous projection onto a finite dimensional subspace
$\pi(X)$. This generalizes the results and techniques from \cite{MattinglyPardoux06}. The fact that the random variable $\pi u(t) \in \pi(X)$ has a density with respect
to Lebesgue measure on $\pi(X)$ follows from the general principle that the
push forward of a density through smooth map remains a nice density provided the Jacobian at typical points is non-degenerate. In our examples, we push forward the law of the Wiener measure
through the projected SPDE solution map $\pi\phi_t$ to obtain
the random variable $\pi u(t)$.  The theory of Malliavin Calculus
precisely shows that the Jacobian of the map is non-degenerate if the
Malliavin covariance matrix is sufficiently non-degenerate. The
needed control over the Malliavin covariance matrix and its inverse is one
of the main results of the theory of hypoellipticity developed
in~\cite{HM09,HM06}. Positivity of this density then follows from the
the work of Ben-Arous-L\'{e}andre~\cite{BAL91a, BAL91b} (see also
\cite{NualartSF,BouleauHirsch_91}), which makes precise the above ideas concerning the push forward in the finite-dimensional setting.

The existence of a control which drives the solution exactly, after
projection, to a given point requires an extension of the preceding
approximate controllability results to an `exact controllability on
projections' result. More precisely,  we show that for any $v \in X$ and $\delta, t >0$
there exist and $\omega \in \Omega_t$ so that both $\| \phi_t(\omega)
- v\| \leq \delta$ and $\pi\phi_t(\omega)=\pi v$. In  \cite{JK_81,
  JK_85, Jur_97}, this extension relies on the structure of smooth
vector fields on $\R^d$. As the phase space in our setting is infinite-dimensional, we had to develop other methods. Here, we
produce this stronger form of controllability using a more refined
notion of saturation, which we call \emph{uniform saturation},
allowing us to transfer continuity properties of the underlying
semigroups from one approximation to the next.  This transfer of
continuity ultimately facilitates the use of the Brouwer fixed-point
theorem to establish exact controllability on projections.

As mentioned above, in addition to this exact controllability on projections
result, to prove positivity of the projected density of the random variable $\pi u_t$, one must show that the appropriate Jacobian of the projected flow map $\pi\phi_t$ is non-degenerate when evaluated at some control which gives the desired exact control on the projected subspace.  The ideas to prove the existence of this non-degenerate control generalize those in \cite{HerMat15} from the finite-dimensional setting, and \cite{MattinglyPardoux06} from
the specific context of the 2D Navier-Stokes equation. These ideas for proving such non-degeneracy leverage the fact that the Malliavin covariance matrix was proven to be almost
surely non-degenerate in a closely related setting~\cite{HM06,HM09}.

The first major development in the paper, as carried out in
Section~\ref{sec:sat},  extends the ideas of Jurdjevic and Kupka
from control theory~\cite{JK_81, JK_85, Jur_97} to dynamics on a very
general phase space $X$ which need only exist locally in time. 
Section~\ref{sec:app:SPDEs} pairs previously
obtained control information via saturation with the Malliavin
calculus to infer support properties of the solution of the associated
SPDE.  Specifically, in Section~\ref{sec:app:SPDEs} we provide a
self-contained presentation of the Malliavin calculus in the abstract setting of cocycles, giving criteria for the existence
and positivity of the probability density function (with respect to
Lebesgue measure) of the projected stochastic process $\pi u$ living
on a finite-dimensional subspace of $X$.  The topic of unique
ergodicity is also discussed.  It is important to highlight that the
criteria given here for the existence and positivity of $\pi u$ do not
require moment bounds on the Malliavin matrix.  For
example, our hypotheses for existence of the density are comparable to
those given in the language of Dirichlet forms in the
work~\cite{BouleauHirsch_91}.

In order to illustrate 
our framework, the
methods developed in Section~\ref{sec:sat} and
Section~\ref{sec:app:SPDEs} are applied to a number of specific equations in
Section~\ref{sec:examples}.  In particular, using the methods of
Section~\ref{sec:sat} we study the low-mode control problems for a
reaction-diffusion equation, the 2D incompressible Navier-Stokes
equation, the Boussinesq equation and the 3D incompressible Euler
equation.  Support and ergodic properties of the associated stochastic
perturbations are then inferred using the results of
Section~\ref{sec:app:SPDEs}.  We are optimistic 
that our techniques will prove useful in the study of other concrete 
examples in the future.

While the low-mode control problems for
the 2D incompressible Navier-Stokes, the 3D incompressible Euler
equation and variants thereof have been studied
previously~\cite{AgrachevSarychev2005, AgrachevSarychev2006,
  Shirikyan2007, Shirikyan2008, Shirikyan2010, Nersisyan2010,
  Nersesyan2015} (see also~\cite{MattinglyPardoux06,
  AgrachevKuksinSarychevShirikyan2007} for consequences for the
support of the stochastically forced systems), we provide these two
examples for completeness of presentation and to illustrate the
efficacy of our formalism.   On the other hand the results for the
Reaction-Diffusion equations and the Boussinesq equations are to the 
best of our knowledge new and may be seen to compliment other recent works on ergodic properties
of these equations in the presence of a degenerate stochastic forcing
\cite{HM09,FoldesGlattHoltzRichardsThomann2013}. Here it is also
important to highlight that our formalism is used to show that 
equations such as the Navier-Stokes and  Boussinesq equations 
remain uniquely ergodic when, in addition to the stochastic perturbation terms,
we add a more or less arbitrary deterministic source term.

The organization of the remainder of this paper is as follows.  In
Section~\ref{sec:overview}, we provide heuristics which both motivate
and give an overview of the rigorous control methods developed in
Section~\ref{sec:sat}.  Section~\ref{sec:app:SPDEs} concerns
applications of the control results of Section~\ref{sec:sat} to an
associated SPDE whose dynamics generates a \emph{continuous adapted
  cocycle}.  In particular, support properties of the law of the
solution $u$ of the SPDE and unique ergodicity are studied from this
point of view.  In Section~\ref{sec:examples}, the theoretical
frameworks developed in previous sections are applied to specific
equations.  Basic a priori estimates for each of the examples studied
are saved for the appendices.

\vspace{.05in}
\noindent {\bf Acknowledgements.}
This work was initiated when the three authors were research members
at the Mathematical Science Research Institute (MSRI) under the ``New
Challenges in PDE: Deterministic Dynamics and Randomness in High
Infinite Dimensional Systems'' program held in the Fall 2015.  We are also 
grateful for the hospitality and travel support provided by 
the mathematics departments at Iowa State University and Tulane University which hosted a number of research visits that facilitated the completion of this work.  We would like to warmly thank Juraj F\"oldes, Susan
Friedlander and Vlad Vicol for numerous helpful discussions and
encouraging feedback on this work.  Our efforts were supported in part through grants DMS-1313272 (NEGH), 
DMS-1612898 (DPH) and DMS-1613337 (JCM) from the National Science Foundation.

\section{Heuristics, Overview of Methods}
\label{sec:overview}

In order to introduce the main ideas and methods employed below in concrete examples, we consider 
an abstract, controlled evolution equation of the form
\begin{align}
   \frac{du}{dt}  + L u + N(u) = f + \sum_{k \in \mathcal{Z}} \alpha_k(t) \sigma_k, \quad u(0) = u_0 \in X.
   \label{eq:abs:evo}
\end{align}
Our system evolves on a phase space $X$ which, for the purposes of
discussion here, may be thought of as a separable Hilbert space with
norm $\| \ccdot \|$.  We assume that $L$ is a linear (unbounded)
operator and that $N$ is a polynomial nonlinearity of the form
\begin{align}
  N(u) = \sum_{k=2}^M N_k(u)
  \label{eq:m:l:form}
\end{align}
where the highest-order term $N_M(u)= N_M(u,u, \ldots, u)$ is such
that $N_M(u_1, u_2, \ldots, u_M)$ is a symmetric multilinear operator
of degree $M$ and for each $2\leq k < M$, $N_k(u)$ is either $0$ or a
homogeneous operator of degree $k$.  We assume that
$\mathcal{Z}$ is a finite set of indices, for example, $\mathcal{Z}$
might be subset of $\ZZ$, $\ZZ^2$, or any other convenient
alphabet of labels. We further assume that the elements
$f$ and $\sigma_k$, $k \in \mathcal{Z}$, represent fixed directions in
the phase space.  The dynamics \eqref{eq:abs:evo} is influenced by the
(piecewise constant) controls $\alpha_k: [0,\infty) \to \RR$.  Of
course in each example presented below in Section~\ref{sec:examples},
we will make concrete assumptions on $L$, $N$, $f$, $\sigma_k$, etc,
so that \eqref{eq:abs:evo} makes sense and is at least locally
well-posed.   In particular in Section~\ref{sec:examples}, we will treat a reaction-diffusion equation, the 2D Boussinesq equation as well as the 2D/3D the Navier-Stokes and Euler equations, all of which can be posed in the
form \eqref{eq:abs:evo}.

Our goal will be to understand when, for any $u_0, v_0 \in X$, any hitting time $t > 0$ and any tolerance of error $\eps >0$,
we can construct a piecewise constant control $\alpha = (\alpha_k :k
\in \mathcal{Z} ): [0,t] \to \RR^{|\mathcal{Z}|}$ so that 
\begin{align}
   \| u(t, u_0, \alpha \cdot \sigma) - v_0\|  < \eps
   \label{eq:approx:cont}
\end{align}
where $u(t, u_0, \alpha \cdot \sigma)$ denotes the solution of
equation~\eqref{eq:abs:evo} with initial condition $u_0$ and control
$\sum_{k\in \mathcal{Z}} \alpha_k(t) \sigma_k$.  Here $|\mathcal{Z}|$
denotes the cardinality of the set $\mathcal{Z}$.
Furthermore, we would often like to show more strongly that if $\pi: X
\to X$ is a fixed continuous projection onto a finite-dimensional
subspace of $X$, and $u_0, v_0 \in X$, $\eps, t>0$ are given, then
there exists a piecewise constant control $\alpha = (\alpha_k : k \in
\mathcal{Z}): [0,t] \to \RR^{|\mathcal{Z}|}$ satisfying~\eqref{eq:approx:cont} and 
\begin{align}
   \pi (u(t, u_0, \alpha\cdot))  = \pi( v_0);
      \label{eq:exact:cont}
\end{align}
namely, $\alpha$ simultaneously provides approximate control on $X$ and exact control on $\pi(X)$.  Throughout, both of these control problems will be referred to generically as the \emph{low mode control problems}, as one typically assumes that $\sigma_k(x) \sim e^{ikx}$ so that we are trying to drive our system \eqref{eq:abs:evo} through a few select frequencies.

Of course we do not expect to be able to solve the low mode control
problems without making further assumptions on the $\sigma_k$'s. In particular, there needs to be \emph{sufficient} control; that is,
$\mathcal{Z}$ must contain enough elements as is typically dictated by its relative structure
with the nonlinearity $N$.  However, because such assumptions can vary from equation to equation, our goal in this section is to illustrate  
how such assumptions on the controls arise as well as how the methods used to solve the low mode control problem work.  

As we shall see, two notions play a central role in our approach to controlling \eqref{eq:abs:evo}, namely \emph{scaling} and \emph{saturation}.  The former notion, based on the introduction of a large parameter, is used to infer new directions along which the system can move besides the explicit directions (the $\sigma_k$'s) acting on the controlled dynamics.  This approach to deriving the needed controls is illuminating as it provides some explicit understanding of the construction of a desirable control $\alpha$; see \eqref{eq:ray:scal} and \eqref{eq:non:lin:tw}.  On the other hand, we will quickly see that
composing scalings to produce further directions iteratively can become unwieldy due to the multiple time scales present in each scaling limit.   
The notion of \emph{saturation}, originally pioneered by Jurdjevic and Kupka in the finite-dimensional setting of ODEs \cite{JK_81, JK_85, Jur_97}, 
provides a highly effective tool which allows one to directly use the new trajectories obtained from each scaling limit.   
Thus, this saturation machinery allows us to iteratively produce a sequence of seemingly more controlled systems which nevertheless reach the same portions of the phase space as the original control problem.  Below in Section~\ref{sec:sat}, we develop a generalization of the Jurdjevic and Kupka approach which is applicable to the abstract spatial 
setting of a metric space.

To introduce these two ideas on a basic level, we begin by adopting further notation.  We use the semigroup formalism and write the solution of~\eqref{eq:abs:evo} at time $t$ with initial data $u_0 \in X$ and constant control $\alpha \cdot \sigma = \textstyle{\sum}_{k\in \mathcal{Z}} \alpha_k \sigma_k$; that is, $\alpha :[0, \infty)\rightarrow \R^{|\mathcal{Z}|}$ is independent of the time parameter $t$, as
\begin{align}
  \Phi^{\alpha\cdot \sigma}_t u_0 = u(t, u_0, \alpha\cdot\sigma).  
  \label{eq:evol:semi}
\end{align}
Throughout, we also make extensive use of the \emph{ray semigroups}
\begin{align}
  \rho^g_t u_0 := u_0 + t g, \,\,\, t\geq 0,
  \label{eq:ray:sg}
\end{align}
defined for any $u_0, g \in X$.  We let $\mathcal{S}$ denote the collection of \emph{continuous local semigroups} on $X$ (see Definition~\ref{def:cont:sg}).\footnote{The collection $\mathcal{S}$  plays essentially the same role as the family of vector fields on finite-dimensional smooth manifolds do in Jurdevic and Kupka's work \cite{JK_81, JK_85, Jur_97}.}  We use a calligraphic font to distinguish between subcollections of these semigroups, i.e. $\mathcal{F}, \mathcal{G} \subseteq \mathcal{S}$. Specifically,
observe that finite compositions of elements in the set
\begin{align}
  \mathcal{F}_0 := \{  \Phi^{\alpha \cdot \sigma} \, :\,  \alpha \in \RR^{|\mathcal{Z}|} \}
  \label{eq:base:sg}
\end{align}
represent the totality of possible paths that solutions of \eqref{eq:abs:evo} can be made to follow using piecewise constant controls.   In each problem we consider in Section~\ref{sec:examples}, we will see that there is enough structure to guarantee that indeed $\mathcal{F}_0\subseteq \mathcal{S}$.  With this notational convention, we introduce the accessibility sets 
\begin{align}
    A_{\mathcal{F}}(u, t) := \big\{ \Psi_{t_m}^m \cdots \Psi_{t_1}^1 u\, :\,  \Psi^\ell \in \mathcal{F} \, \text{ for all } \ell=1,2, \ldots, m \text{ and }\textstyle{\sum}_{i=1}^m t_i = t\big\},
\label{eq:exact:ass}
\end{align}
so that the condition \eqref{eq:approx:cont} holds precisely when $\overline{A_{\mathcal{F}_0}(u,t)} = X$ for every $u \in X$
and every $t > 0$.

\subsection{Scaling Arguments}
We will primarily use two flavors of scalings to generate new directions in the phase space.  The first scaling increases the magnitude of the control while reducing the time interval on which it acts, thereby allowing us to see that we can approximately reach anything in the set $u_0 + \text{span}\{ \sigma_k \, : \, k \in \mathcal{Z}\}$, $u(0)=u_0$,
in arbitrarily small amounts of time.  This gives us the freedom to flow in the direction of any element in the set $\text{span}\{ \sigma_k \, : \, k \in \mathcal{Z}\}$ in small times.  The second scaling type uses these previously obtained directions and then cycles them through the nonlinearity $N$ via an appropriate composition of flows.  Due to the resonant interaction between the nonlinearity $N$ and scaled rays, we are then able to generate new directions not belonging to the span of the $\sigma_k$'s.  As we will see, this process can then be iterated to produce even more directions.

\begin{remark}
\label{rem:notgenproc}
Although the two specific scalings presented below give one recipe for generating directions in the phase space, the particular path taken in this section may not generate all needed directions to solve a particular control problem.  Indeed, many other scalings are possible and further ingenuity may be needed to demonstrate access to new directions.   In particular, three of the four examples presented in Section~\ref{sec:examples} rely entirely on the two scalings introduced in this section while the Boussinesq equation requires rather different scaling combinations to produce a suitable control.      
\end{remark}

To describe the first scaling type in more detail, fix $\alpha \in \RR^n$ and take $\lambda \gg 1$ to be a 
scaling parameter.  For any $u_0 \in X$, take 
\begin{align}
	u_\lambda(t) = \Phi^{\lambda\alpha \cdot \sigma}_{t/\lambda} u_0  
	\label{eq:ray:scal}
\end{align}
and observe that $u_\lambda$ satisfies:
\begin{align}
	\frac{d u_\lambda}{dt} =  \alpha \cdot \sigma - \frac{1}{\lambda} (L u_\lambda + N(u_\lambda)) + \frac{1}{\lambda} f \approx \alpha \cdot \sigma, \quad u_\lambda(0) = u_0.
	\label{eq:rescal:0}
\end{align}
Thus, one might expect that by employing suitable a priori estimates
\begin{align}
  \lim_{\lambda \to + \infty} \|  \Phi^{\lambda \alpha}_{t/\lambda} u_0  -  \rho^{\alpha\cdot \sigma}_t u_0 \| = 0
  \label{eq:frst:scalling}
\end{align}
for any $u_0 \in X$, $\alpha \in \RR^n$ and $t>0$.    In summary, the scaling introduced in \eqref{eq:ray:scal} allows us to push the dynamics of \eqref{eq:abs:evo}
along rays of the form 
\begin{align*}
\rho_t^{\alpha \cdot \sigma} u_0= u_0 + (\alpha \cdot \sigma) t, \,\,\,\, t\geq 0,
\end{align*}
 in a short burst of time ($t/\lambda$ for $\lambda \gg 1$), thus providing an initial step to generate directions.  Hence, setting 
\begin{align}
  X_0 = \{\alpha \cdot \sigma \, : \, \alpha \in \RR^{|\mathcal{Z}|} \}
  \label{eq:dir:ass}
\end{align}  
we expect to be able to approximately flow along rays in the direction of elements in $X_0$ in small positive times.  

To describe the intuition behind the second scaling that allows us to
generate directions besides those explicitly acting on the dynamics, 
we will use the trajectories $\rho_t^{\alpha \cdot \sigma} u_0$,
$\alpha \in \RR^{|\mathcal{Z}|}$ obtained by the previous scaling,
and  `push' the directions $\alpha \cdot \sigma$, $\alpha \in \RR^{|\mathcal{Z}|}$, through the nonlinear term $N$.   This is carried out as follows.  Consider the scaling 
\begin{align*}
  v_\lambda(t) =  \Phi^0_{\frac{t}{\lambda^M}} \rho^{\lambda^2 \alpha \cdot \sigma}_{\frac{1}{\lambda}} u_0 
\end{align*}
for $\lambda \gg 1$, where we emphasize that the null superscript in
$\Phi^0_{t/\lambda^M}$ means that we completely `turn off' the control
in \eqref{eq:abs:evo}.  
In the above expression for $v_\lambda(t)$, we note that for $\lambda \gg 1$ 
\begin{align*}
\rho_{\lambda^{-1}}^{\lambda^2 \alpha \cdot \sigma} u_0 = u_0 + \lambda \alpha \cdot \sigma \sim \lambda \alpha \cdot \sigma 
\end{align*}
and thus the scale $t/\lambda^M$ is picked to balance the asymptotic
behavior of the leading-order term along $\lambda \alpha \cdot \sigma$
in $N$.  More formally, presuming that $N_M(\alpha \cdot \sigma)\neq 0$ we have  $N_M(\lambda \alpha \cdot \sigma) = \lambda^M N_M(\alpha \cdot \sigma)$, so that for $\lambda \gg 1$
\begin{align*}
   \int_0^{t/\lambda^M}(L(v_\lambda) + N(v_\lambda))ds \approx \int_0^{t/\lambda^M}( L(u_0 + \lambda \alpha \cdot \sigma)+  N(u_0 + \lambda \alpha \cdot \sigma)) \, ds \approx t N_M(\alpha \cdot \sigma).
\end{align*} 
where we have also used that $v_\lambda \approx u_0 + \lambda \alpha
\cdot \sigma$ on the short time interval $[0, t/\lambda^M]$.
Hence we see that
\begin{align}
   v_\lambda(t) &= u_0 + \lambda \alpha \cdot \sigma - \int_0^{t/\lambda^M} [L(v_\lambda) + N(v_\lambda) ]ds + \frac{t}{\lambda^M}f  \notag\\
  &\approx  u_0 + \lambda \alpha \cdot \sigma - t N_M(\alpha \cdot \sigma),
    \label{eq:rescal:m}
\end{align}
which should be valid for all $\lambda \gg 1$. 

Observe that while we have picked up the direction $N_M(\alpha \cdot \sigma)$ the trajectory $v_\lambda(t)$ does not stabilize since $u_0 + \lambda \alpha \cdot \sigma$ blows up as $\lambda \rightarrow \infty$.  To take care of this, we introduce another composition 
 \begin{align}
  w_\lambda(t) = \rho_{\lambda^{-1}}^{-\lambda^2 \alpha \cdot \sigma} v_\lambda(t)
   = \rho_{\lambda^{-1}}^{-\lambda^2 \alpha \cdot \sigma} \Phi^0_{t/\lambda^M} \rho^{\lambda^2 \alpha \cdot \sigma}_{\lambda^{-1}} u_0 
  \label{eq:non:lin:tw}
\end{align}
which, starting from $v_\lambda(t)$, simply flows along the direction $-\lambda^2 \alpha \cdot \sigma$ in $\lambda^{-1}$ time units.  Using~\eqref{eq:rescal:m} we see that for $\lambda \gg 1$
\begin{align*}
w_\lambda(t) \approx u_0 -t N_M(\alpha \cdot \sigma).  
\end{align*} 
One can also see the approximation above by noting that $w_\lambda$ satisfies
\begin{align}
	\frac{d w_\lambda}{dt} =  - \frac{1}{\lambda^M} (L (w_\lambda+\lambda \alpha \cdot \sigma) + N(w_\lambda + \lambda \alpha \cdot \sigma)) + \frac{1}{\lambda^M} f 
  \approx - N_M(\alpha \cdot \sigma),
  \label{eq:rs:nLin:twit}
\end{align}
starting from  $w_\lambda(0)=u_0$.  In summary, given suitable
PDE-dependent a priori estimates, we therefore expect
\begin{align}
\lim_{\lambda \to + \infty} 
	\|   \rho^{- \lambda^2\alpha \cdot \sigma}_{\lambda^{-1}} \Phi^0_{t/\lambda^M} \rho^{\lambda^2 \alpha \cdot \sigma}_{\lambda^{-1}} u_0   - \rho^{-N_M(\alpha \cdot \sigma) }_t u_0 \| = 0,
	 \label{eq:m:scl}
\end{align}
for any fixed $u_0$ and any $\alpha \cdot \sigma$, $t\geq 0$.  

Of course it is not immediately clear that the trajectory given by
\eqref{eq:non:lin:tw} can be obtained from (\ref{eq:abs:evo}) by
composing elements solely from the set \eqref{eq:base:sg}.   
On the other hand by the first scaling argument, we can expect to get arbitrarily close to $\rho^{\lambda^2 \alpha \cdot \sigma}_{\lambda^{-1}} u_0 $ for $\lambda \gg 1$ by considering
\begin{align*}
  \Phi_{(\lambda \mu)^{-1}}^{(\lambda\mu)^2 \alpha \cdot \sigma}\,\,\,\, \text{ for } \,\,\,\, \mu=\mu(\lambda) \gg 1.
\end{align*}  
Thus for each $\lambda \gg 1$ by picking $\mu=\mu(\lambda) \gg 1$ and considering 
\begin{align}
  \Phi_{(\lambda \mu)^{-1}}^{-(\lambda\mu)^2 \alpha \cdot \sigma}   
  \Phi^0_{\frac{t}{\lambda^M}} \Phi^{(\lambda\mu)^2 \alpha \cdot \sigma}_{(\lambda \mu)^{-1}} u_0,
  \label{eq:sc:it:1}
\end{align}
we may expect to be able to approximate $w_\lambda$ using the trajectories in $\mathcal{F}_0$.  

In view of the above discussions we would now like to iterate the use of
the two scalings (\ref{eq:ray:scal}) and (\ref{eq:non:lin:tw}) to
approximate a much richer collection of ray semigroups.    To this end we
define
\begin{align*}
X_1 = \text{span}  \Big\{X_0 \cup \{N_M(g) \, : \,  g \in X_0\} \Big\}
\end{align*}
with $X_0$ as in \eqref{eq:dir:ass} and for $k\geq 2$
we define $X_{k}$ inductively by 
\begin{align}
\label{eq:brak:gen}
X_k = \text{span}\Big\{X_{k-1} \cup \{ N_M(g) \, :
  \,  g \in X_{k -1} \}\Big\}.
\end{align}
Thus we might expect to approximately reach
points of the form
\begin{align}
  u_0 + v \quad \text{ for any } v \in X_\infty :=\bigcup_{k \geq 1}
  X_k
  \label{eq:n:reach:pts}
\end{align}
by composing the relevant scalings as we did for points in $X_1$ in
\eqref{eq:sc:it:1} above.  Hence we might expect the density condition
\begin{align}
   \overline{X_\infty} = X
  \label{eq:den:brak:prelim}
\end{align}
to be sufficient to achieve approximate controllability as in \eqref{eq:approx:cont}.  

Before turning to the notion of \emph{saturation} which we will use to
facilitate the iterative process of generating elements in
$X_k$ for $k \geq 1$ through multiple scalings, we now make some crucial remarks.  
\begin{remark}
\label{rmk:braks}
\mbox{}
\begin{itemize}
\item[(i)]  If $M$ is even, it is not true in general that we can generate directions in $X_k$ using the scaling analysis above.  This is because one need needs to be able to flow both forwards and backwards along the directions $N_M(g)$, $g\in X_{k-1}$, using the ray semigroup.  In particular in the case when $M$ is even, it is certainly not obvious nor true in general that both $N_M(g), -N_M(g) \in X_{k}$ given that $g\in X_{k-1}$, so the second scaling used above does not work when replacing $\alpha \cdot \sigma$ with $N_M(g)$.  We will see that in some special cases, specifically in some models from fluid mechanics where $M=2$, that this scaling analysis can still in fact be used to realize the sets $X_k$ via the scalings above.  See Remark~\ref{rem:even} for a further discussion of this point.              
\item[(ii)] Given the definition of the $X_k$'s, it is not clear that the set $X_\infty$ is rich enough to be dense in $X$.  However, we will see that Condition~\eqref{eq:den:brak:prelim} is equivalent to an infinite-dimensional analogue of H\"{o}rmander's condition.  See Section~\ref{sec:wHor:cond}.      
\item[(iii)] In each of the examples considered below in
Section~\ref{sec:examples}, it will not be the case that $N_M(g)\in X$
for generic $g\in X$.  This will not pose any significant problem,
however, as we are mainly interested in the low mode
control problem where the control subspace $X_0$ will
consist of smooth ($C^\infty$) elements in $X$.  This, in particular,
will allow us to conclude that $N_M(g) \in X$ and is smooth for each
$g \in X_0$.  Moreover, when the scaling above is iterated,
we will also be able to conclude $N_M(g) \in X$ for each $g\in
X_k$ and each $k\geq 0$.  We will therefore operate
under the assumption that $N_M(g) \in X$ for each $g\in X_k$
and each $k\geq 0$ for the remainder of this section.  
\end{itemize}
\end{remark}

\begin{remark}
\label{rem:ASapproach}
At this stage it is important to note some differences between the control theoretic approach adopted here and the Agrachev-Sarychev approach~\cite{AgrachevSarychev2005, AgrachevSarychev2006,
  Shirikyan2007, Shirikyan2008, Shirikyan2010, Nersisyan2010,
  Nersesyan2015}.  The latter approach relies on establishing three key properties for the control problem~\eqref{eq:abs:evo}:
\begin{itemize}
\item[(I)]  The \emph{Extension Principle}. This states that the system~\eqref{eq:abs:evo} is approximately controllable on $X$ if and only if the following control problem 
\begin{align}
\label{eqn:morecontrolsol}
\partial_t u + L(u + h_1) + N(u + h_1) = f + h_0,
\end{align}    
where $h_0, h_1: [0, \infty)\rightarrow X_0$ belong to an appropriate class of controls, is \emph{approximately controllable} on $X$.  Here, approximate controllability of~\eqref{eqn:morecontrolsol} means that for every $\epsilon, t >0, u_0, v \in X$ there exist controls $h_0, h_1:[0, \infty) \rightarrow X_0$ such that the solution $u(\cdot, u_0)$ of~\eqref{eqn:morecontrolsol} with $u(0,u_0)=u_0$ has $\| u(t,u_0) - v\| < \epsilon$.  Note that by simply setting $h_1\equiv 0$, we retain the original control problem~\eqref{eq:abs:evo}.  Thus this principle is important because~\eqref{eqn:morecontrolsol} has more degrees of freedom (in terms of control) even though both control problems, \eqref{eq:abs:evo} and \eqref{eqn:morecontrolsol}, are equivalent in this sense.        
\item[(II)]  The \emph{Convexification Principle}.  Let $\tilde{X}_1\subseteq X$ denote the finite-dimensional subspace generated by elements of the form
\begin{align*}
h_0 + N_M(h_1, h_2,\ldots, h_M).
\end{align*}      
The Convexification Principle states that~\eqref{eqn:morecontrolsol} is approximately controllable by the $X_0$-valued controls $h_0, h_1$ if and only if the original control problem~\eqref{eq:abs:evo} is approximately controllable by an $\tilde{X}_1$-valued control.  Note here that $\tilde{X}_1 \supseteq X_0$.  Thus provided $\tilde{X}_1 \supsetneq X_0$, we have gained more control directions over the original system via equivalence.  
\item[(III)]  The \emph{Saturating Property}.  This simply states that equivalence in (I) and (II) is transitive so that the process can be iterated, producing an increasing family of subspaces 
\begin{align*}
\tilde{X}_1 \subseteq \tilde{X}_2 \subseteq \cdots \subseteq \tilde{X}_n \subseteq \cdots
\end{align*}       
such that if $\cup_n \tilde{X}_n$ is dense in $X$, then the original control problem~\eqref{eq:abs:evo} is approximately controllable. 
\end{itemize}
The approach adopted here is different in the sense that explicit scalings are used to generate directions along which the dynamics can move in short bursts of time.  For the two scalings used above, we will see in Section~\ref{sec:relhormander} that the subspaces $X_n, \, n=1,2,\ldots$, generated in~\eqref{eq:brak:gen} precisely coincide with the subspaces produced iteratively in step (III).  Furthermore, the explicit nature of the method not only allows us to shed light on how the equivalence in (I) and the directions in (II) arise but it also allows us to bypass showing the Extension Principle and Convexification in each equation altogether.  Later in Remark~\ref{rem:diffexact}, we will highlight another important difference between the two approaches which allows us to induce simultaneous approximate control on $X$ and exact control on $\pi(X)$, $\pi:X\rightarrow X$ denoting a continuous projection onto a finite-dimensional subspace of $X$.  This is done through the combined use of the explicit scalings and a new, refined notion of saturation called \emph{uniform saturation} defined in Section~\ref{sec:exactcontrol}.      
\end{remark}

\subsection{Saturation}
It is clear that there is a tantalizing connection between the scaling arguments \eqref{eq:frst:scalling} and \eqref{eq:m:scl}
and the task of generating controls which approximate points of the form appearing in \eqref{eq:n:reach:pts}.  
However, there are a few issues which prevent us from directly concluding \eqref{eq:approx:cont} from \eqref{eq:frst:scalling} and \eqref{eq:m:scl} under a density condition like \eqref{eq:den:brak:prelim}.  First, notice that in order to iterate our strategy we quickly end up with a horrific tangle of multiple time scales.  Another problem is that the rescaling strategy leading to \eqref{eq:frst:scalling} and to  \eqref{eq:m:scl} are more conducive to studying the 
time relaxed sets defined by 
\begin{align}
  A_{\mathcal{F}}(u, \leq t) := \bigcup_{s \leq t}  A_{\mathcal{F}}(u, s)  = \{  \Psi_{t_m}^m \cdots \Psi_{t_1}^1 u: \Psi^\ell \in \mathcal{F}, \ell=1,2, \ldots, m, \text{ and } \textstyle{\sum}_{i=1}^m t_i  \leq t \}.
  \label{eq:rlx:acs}
\end{align}
It is thus clear that further arguments are needed to mediate between points lying in $A_{\mathcal{F}}(u, \leq t)$ to those lying in the more restricted sets
$A_{\mathcal{F}}(u, t)$.

The notion of \emph{saturation} addresses these dual considerations and more.  Let us begin with the observation that the relaxed accessibility sets \eqref{eq:rlx:acs} provide us with a way to place a partial ordering on $\mathcal{S}$.  Given $\mathcal{F}, \mathcal{G} \subseteq \mathcal{S}$ we will say that $\mathcal{F}$ \emph{subsumes} $\mathcal{G}$, denoted by
$\mathcal{G}\preccurlyeq \mathcal{F}$, if $$\overline{A_{\mathcal{G}}(u, \leq t)} \subseteq \overline{A_{\mathcal{F}}(u, \leq t)}$$ for every $u \in X$ and $t > 0$.
On the other hand we say that two collections of semigroups  $\mathcal{F}, \mathcal{G} \subseteq \mathcal{S}$  are \emph{equivalent}, denoted by $\mathcal{F} \sim \mathcal{G}$, 
if both $\mathcal{G}\preccurlyeq \mathcal{F}$ and
$\mathcal{F}\preccurlyeq \mathcal{G}$.    

As we will see, it is not hard to show that $\mathcal{G} \preccurlyeq \mathcal{F}$ if and only if given any $\Psi \in \mathcal{G}$, $u \in X$ and any $\eps, t>0$
there exists  $\Phi^1, \ldots,   \Phi^n \in \mathcal{F}$ and $t_1 + \cdots + t_n \leq t$ so that
\begin{align}
    \| \Phi_{t_n}^n  \cdots  \Phi_{t_1}^1 u - \Psi_t u \|  < \eps.
    \label{eq:app:char}
\end{align}
This characterization \eqref{eq:app:char} allows us to consider `one scaling at a time' as follows.  
Let 
\begin{align*}
  \mathcal{G}_0 :=  \{ \rho^\xi: \xi \in X_0\} \cup \mathcal{F}_0, 
\end{align*}
and for $k \geq 1$ take
\begin{align*}
  \mathcal{G}_{k} :=\{  \rho^{\xi} : \xi \in X_{k} \} \cup \mathcal{F}_0
\end{align*}
where $X_k$ is defined as above in \eqref{eq:brak:gen}.
Observe that  \eqref{eq:frst:scalling} combined with \eqref{eq:app:char} shows that 
\begin{align}
  \mathcal{G}_0 \preccurlyeq \mathcal{F}_0,
  \label{eq:0th:gen:sat}
\end{align}
where $\mathcal{F}_0$ represent the solutions of our original system \eqref{eq:abs:evo} under constant controls as defined above in \eqref{eq:base:sg}.  Similarly combining \eqref{eq:m:scl} and \eqref{eq:app:char} implies that
\begin{align}
\mathcal{G}_k \preccurlyeq \mathcal{G}_{k-1},
  \label{eq:Mth:gen:sat}
\end{align}
for every $k \geq 1$. 

Let us now see how combining the observations in \eqref{eq:0th:gen:sat} and  \eqref{eq:Mth:gen:sat} now allows us to conclude that
\begin{align}
  \overline{A_{\mathcal{F}}(u, \leq t)} = X
  \label{eq:sat:span:con}
\end{align}
for any $u \in X$ and $t > 0$ under the density condition \eqref{eq:den:brak:prelim}.
Indeed, the characterization of $\mathcal{G} \preccurlyeq \mathcal{F}$ above in \eqref{eq:app:char} allows us
to conclude that if 
\begin{align*}
  \mathcal{H}_i \preccurlyeq \mathcal{F}
\end{align*}
for some collection $\mathcal{H}_i$ of subsets of $\mathcal{S}$, then
\begin{align*}
  \mathcal{F} \sim \mathcal{F}  \cup \bigcup_i \mathcal{H}_i.
\end{align*}
In particular it is clear that the \emph{saturate} of $\mathcal{F}$ defined as
\begin{align*}
  \text{Sat}(\mathcal{F}) := \bigcup_{ \mathcal{H} \preccurlyeq \mathcal{F}} \mathcal{H}
\end{align*} 
satisfies $\text{Sat}(\mathcal{F}) \sim \mathcal{F}$.  Thus in particular
we find that \eqref{eq:0th:gen:sat} and \eqref{eq:Mth:gen:sat} imply
\begin{align*}
   \{ \rho^{\xi}: \xi \in X_\infty \}  \preccurlyeq  \text{Sat}(\mathcal{F}_0) \sim \mathcal{F}_0 , 
\end{align*}
and hence \eqref{eq:sat:span:con} follows from \eqref{eq:den:brak:prelim}. 

Of course \eqref{eq:sat:span:con} does not immediately imply the exact time 
approximate controllability condition \eqref{eq:approx:cont} is satisfied.
This is due to the fact that we lack
precise control over the time at which we get close to the target
$v_0$.  Nevertheless, it turns out that we can show that
\eqref{eq:sat:span:con} implies \eqref{eq:approx:cont} in a very
general setting.  The argument which establishes this time conversion is roughly the following.  Given $t>0$, starting from $u_0\in X$ we can get arbitrarily close to a desired target $v\in X$ at some time $0<s<t$.  We can then bounce back and forth between this neighborhood and other values in $X$ to make up the remaining time $t-s$. For further details, see Lemma~\ref{lem:ex:tm:conv}.

\subsection{Exact Control on Projections: Uniform Saturation, Fixed
  Point Arguments}
 
The arguments sketched so far provide a broadly applicable approach to obtaining
approximate controllability at a fixed time $t > 0$.  However, in order to simultaneously provide an approximate control on $X$, condition \eqref{eq:approx:cont}, and exact control on $\pi(X)$  as in~\eqref{eq:exact:cont}, further refinements of the saturation formalism are needed.  As already noted we are mainly interested in the situation where $\pi:X\rightarrow X$ is a continuous projection onto a finite-dimensional subspace of $X$.

Below in Section~\ref{sec:exactcontrol} we introduce the notion of \emph{uniform saturation}
which essentially guarantees continuity in our control with respect to changes
in the initial condition and target point.   This continuity is used conjunction with the Brouwer fixed
point theorem to infer \eqref{eq:exact:cont} when $\pi$ has finite-dimensional range.  
See Theorem~\ref{thm:exact:cont:abs} for our precise formulation. Note that we use the term
`uniform saturation' since we require scaling approximations to hold
uniformly over compact subsets of initial data and compact subsets of
the control parameter space.  This uniformity allows us to transfer
continuity from one approximation to the next.

While requiring uniformity does complicate the presentation of the saturation formalism, 
the needed estimates at the level of the PDE do not change much.  Specifically we will see that it is sufficient
to replace \eqref{eq:frst:scalling} with
\begin{align}
  \lim_{\lambda \to + \infty} \sup_{u_0 \in K, \alpha \in \tilde{K}} \|  \Phi^{\lambda  \alpha\cdot\sigma }_{t/\lambda} u_0  -  \rho^{\alpha\cdot \sigma}_t u_0 \| = 0,
  \label{eq:uni:scale:b}
\end{align}
for any compact $K \subseteq X, \tilde{K} \subseteq \RR^n$. Similarly, \eqref{eq:m:scl} needs to be extended to
\begin{align}
\lim_{\lambda \to + \infty} \sup_{u_0 \in K, g \in \tilde{K}} 
	\|   \rho^{- \lambda^2 g}_{1/\lambda} \Phi^0_{t/\lambda^M} \rho^{\lambda^2 g}_{1/\lambda}  u_0 - \rho^{-N(g) }_t u_0 \| = 0,
	\label{eq:uni:scale:it}
\end{align}
over any compact sets $K, \tilde{K} \subseteq X$.  We may expect such bounds to follow from \eqref{eq:rescal:0} or  \eqref{eq:rescal:m}, by similar estimates for any reasonably well-behaved equations of the form \eqref{eq:abs:evo}.  

\begin{remark}
\label{rem:diffexact}
The previous paragraph highlights another difference between the control theoretic approach developed here and the Agrachev-Sarychev approach.  In particular, using the approach developed in our paper, one does not need to prove approximate controllability and then prove simultaneous approximate control on $X$ and exact control on $\pi(X)$ for a given finite-dimensional projection $\pi$.  Rather, the stronger form of controllability follows immediately by the strength of the explicit scaling estimates.  In other words, one bypasses this step when estimates such as \eqref{eq:uni:scale:b} and \eqref{eq:uni:scale:it} are satisfied, so long as a dense set of directions can be generated by iterating the scaling estimates.        
\end{remark}

\subsection{Further Remarks on Spanning Conditions}
\label{sec:wHor:cond}

We finally return to the discussion of the sequence of approximating spaces $X_k$ defined above in \eqref{eq:brak:gen}.  
As already mentioned, the scope of algebraic conditions covered by this set up is wider than it at first appears.
First we will show that condition~\eqref{eq:den:brak:prelim} is equivalent to an infinite-dimensional analogue of 
H\"{o}rmander's Condition introduced and employed in~\cite{HM09}.  We then conclude this subsection with some remarks
showing that in certain cases when the degree of the polynomial $N$ in \eqref{eq:abs:evo} is even, 
the subspaces $X_k$ can still be produced using the scaling and saturation arguments above.  
Both of these observations will be crucially used for the examples in Section~\ref{sec:examples} below.

\subsubsection{Relationship to H\"{o}rmander's Condition}
\label{sec:relhormander}
To introduce this infinite-dimensional version of H\"{o}rmander's Condition, starting from 
\begin{align*}
   \tilde{X}_0 := \mbox{span} \{  \sigma_k : k \in
  \mathcal{Z} \}
\end{align*}
for $n \geq 1$ let 
\begin{align*}
   \tilde{X}_n := \mbox{span} \left\{ \tilde{X}_{n-1} 
                                                   \cup 
                                      \{ N_M(g_1, \ldots, g_M) : g_j \in
                                            \tilde{X}_{n-1}\}\right\} 
\end{align*}
where we are assuming that $N_M(g_1, \ldots, g_M) \in X$ whenever $g_1,g_2, \ldots, g_M \in X_n$ for some $n$.  
\begin{definition}\label{def:hor:cond}
We say that $(N_M, \sigma)$ satisfies \emph{H\"{o}rmander's Condition on} $X$ if
\begin{align}
   \overline{\bigcup_{n \geq 0}   \tilde{X}_n}  = X.
  \label{eq:Hors:Cond}
\end{align}  
\end{definition}

We now state and prove the following proposition giving equivalence of condition~\eqref{eq:den:brak:prelim} and H\"{o}rmander's condition.  

\begin{proposition}
\label{prop:horeq}
We have that $X_k= \tilde{X}_k$ for all $k\geq 0$.  Consequently, condition~\eqref{eq:den:brak:prelim} is satisfied if and only if $(N_M, \sigma)$ satisfies H\"{o}rmander's condition on $X$.  
\end{proposition}

\begin{proof}
Clearly, $X_0= \tilde{X}_0$.  Also, for $k\geq 1$, $X_k \subseteq \tilde{X}_k$.  To see the opposite inclusion for $k\geq 1$, we adapt
the argument in Lemma~6 of \cite{JK_85}.  Fix $g,h \in X_{k-1}$ and consider $$X_{k-1}(g,h):=\text{span}\, \{ N_M(g + \alpha h) \, : \, \alpha \in \R\}.$$ 
Observe that since $N_M$ is multilinear of degree $M$, $X_{k-1}(g, h)$ is a finite-dimensional subspace of $X_k$, hence is closed.  In particular, since the sequence      
\begin{align*}
\bigg\{\frac{1}{\lambda}(N_M(g+ \lambda h) - N_M(g))\bigg\}_{\lambda \in (0, 1]} \subseteq X_{k-1}(g,h)
\end{align*}   
converges as $\lambda \rightarrow 0$ to $M \cdot N_M(g,g,\ldots, g, h)$, 
we conclude that $N_M(g,g,\ldots, g, h) \in X_{k}$ for all $g, h \in X_{k-1}$.  Recall here that the multilinear operator $N_M$ has been symmetrized.  This argument can then be iterated to see that $N_M(h_1, h_2, \ldots, h_M) \in X_{k}$  for all $h_i \in X_{k-1}$, allowing us to conclude~$\tilde{X}_k \subseteq X_k$.        
\end{proof}

\subsubsection{Even Degree Nonlinearities}
\label{rem:even}
Let us next make some remarks concerning even degree polynomial nonlinearity $N$ in \eqref{eq:abs:evo}.   
Specifically we introduce conditions on $N$ applicable to the 2D Navier-Stokes equations, the 3D Euler Equations 
and the 2D Boussinesq equations (in each case in the absence of boundaries)
considered below in Section~\ref{sec:examples}.

Suppose that the leading-order nonlinearity $N_2$ is a bilinear form and assume we 
have countable set of elements $\{ e_j \}_{j\in \N}\subseteq X$ satisfying the following conditions:
\begin{itemize}
\item[(1)] $\text{span}\{e_j \,: \, j=1,2, \ldots, n\} \subseteq \text{span}\{ \sigma_k  \, : \, k \in \mathcal{Z}\}$ for some $n\geq 1$;    
\item[(2)] We have the cancellation property
\begin{align}
  N_2(e_j, e_j)=0  \qquad \text{ for every } j \in \N;
  \label{eq:bi:form:can}
\end{align}
\item[(3)] For all $j,k \in \N$ there exists a natural number $N(j, k)$ such that  
\begin{align}
  N_2(e_j, e_k) \subseteq \text{span}\{ e_\ell \, : \, \ell \leq N(j, k) \}.  
  \label{eq:bi:form:span}
\end{align}
\end{itemize}
To see how these conditions may be satisfied see for example \eqref{eq:rel:deg:NSE} below.

In this case defining $X_0 = \text{span}\{ e_j \, : \, j =1,2, \ldots, n \}$ and $X_k$, $k\geq 1$, as 
\begin{align*}
X_k = \text{span} \Big\{ X_{k-1} \cup \{ N_2(g) \, : \, g \in X_{k-1} \}\Big\},
\end{align*}
we now see that each set $X_k$ can be realized using the two scaling arguments above.
In this regard, the key observation is that for any $\alpha \in \R$ and $j, k\in \N$ the first part of condition (2) implies
\begin{align*}
N_2(\alpha e_j + e_k)= 2 \alpha N_2(e_j, e_k) .
\end{align*}
Hence by condition (1) and the second part of condition (2),
inductively the $X_k$ can be obtained using the scaling and saturation
arguments above by choosing the $\alpha$ to have the correct sign
(either positive or negative).  In other, more imprecise words, the
nonlinearity $N_2$ is `behaving like' an odd degree polynomial.  In
the finite-dimensional setting, this behavior is captured in the
notion of \emph{relative degree} introduced and studied
in~\cite{HerMat15}.  See also~\cite{Rom_04}.


\section{Saturation in Infinite Dimensions}
\label{sec:sat}

We turn now to provide a rigorous treatment of saturation in the sprit of the framework developed by Jurdjevic and Kupka~\cite{JK_81, JK_85, Jur_97}.  
Much of the formalism developed here requires little underlying structure of the phase space, and we therefore present many of the results in the section in the general setting 
of a metric space.  After introducing the rigorous setup in Section~\ref{eq:gen:sat:not}, we turn to proving some
results about saturation that are crucial elements for establishing \eqref{eq:approx:cont} in the forthcoming examples
in Section~\ref{sec:examples}.  This subsection concludes with a `conversion lemma' (Lemma~\ref{lem:ex:tm:conv})
which allows us to translate controllability on relaxed time sets a la (\ref{eq:sat:span:con}) to exact time controllability (\ref{eq:approx:cont}), (\ref{eq:exact:ass}).   
The final subsection (Section~\ref{sec:exactcontrol})
introduces a more refined version of saturation, called \emph{uniform saturation}, which also tracks the continuity of approximations with respect to parameters.  This notion is crucial for the main result of this section, Theorem~\ref{thm:exact:cont:abs}, which is used in conjunction with Lemma~\ref{lem:sat} and Lemma~\ref{lem:ex:tm:conv}
to establish establish exact controllability for finite-dimensional projections via \eqref{eq:exact:cont} in the examples treated below in Section~\ref{sec:examples}.

\subsection{General Notions For Controllability}\label{eq:gen:sat:not}
Let $(X, d)$ be a metric space.  We fix an additional point $\death$, called the \emph{explosive state}, not belonging to $X$.  This is the `death state' where locally-defined semigroups will live at times after they fail to exist in $X$.       
 
\begin{definition}\label{def:cont:sg}
We call a mapping $(t,u) \mapsto \Phi_t u : [0,\infty) \times X \rightarrow X \cup \{\death\} $ a \emph{continuous local semigroup} on $(X,d)$ if, 
for every $u\in X$, there exists $T_u \in (0, \infty]$ such the following conditions are satisfied: 
\begin{itemize}
\item[(i)] For $t \in [0, T_u)$, $\Phi_t u \in X$ and for $t \geq T_u$, $\Phi_t u= \death$.   
\item[(ii)] $\Phi_0 u =u$ and for all $t, s\in [0, T_u)$ with $t+s \in [0, T_u)$, we have that $t\in[0, T_{\Phi_s u})$ and $\Phi_{t+s}u= \Phi_t \Phi_s u$.
\item[(iii)]  For all $t\in [0, T_u)$ and all $\eps >0$, there exists $\delta >0$ such that whenever $(t', u') \in [0, \infty) \times X$ satisfies 
\begin{align*}
|t-t'| + d(u, u') < \delta
\end{align*}
we have that $t' \in [0, T_{u'})$ and 
\begin{align*}
d(\Phi_t u, \Phi_{t'}u') < \eps.  
\end{align*}
\end{itemize}
\end{definition}

For notational convenience, we will use $\Phi$ to denote a continuous local semigroup $(t,u) \mapsto \Phi_t u:[0, \infty)\times X \rightarrow X \cup \{\death\}$.   
We will say that $\Phi$ is \emph{global}  if $T_u = \infty$ for every $u \in X$.  Throughout, $\mathcal{S}$ will denote the set of all such continuous local semigroups on $X$.  

\begin{remark}
Some of the semigroups we will work with are only known to be defined locally in time, e.g. the 3D Euler equation 
considered in Section~\ref{sec:e:nse:ex}.  Thus, when we write the composition 
\begin{align*}
\Phi^n_{t_n} \Phi^{n-1}_{t_{n-1}} \cdots \Phi_{t_1}^1 u
\end{align*}
below, it is implicitly assumed that $\Phi^j_{t_j} \Phi^{j-1}_{t_{j-1}} \cdots \Phi_{t_1}^1 u \in X$ for all $j=1,2, \ldots, n$.  
\end{remark}

Given $\mathcal{F} \subseteq \mathcal{S}$ arbitrary, we now introduce the \emph{accessibility sets} corresponding to $\mathcal{F}$, which are simply the points in 
$X$ that can be reached by iteratively composing elements in $\mathcal{F}$.  
\begin{definition}\label{def:acc:cont}
Consider $\mathcal{F} \subseteq \mathcal{S}$.   
\begin{itemize}
\item[(i)] For $u \in X$ and $t>0$, define
\begin{align}
  A_{\mathcal{F}}(u, t) = \big\{\Phi_{t_n}^n  \Phi_{t_{n-1}}^{n-1} \cdots  \Phi_{t_1}^1 u \,: \, \Phi^j \in \mathcal{F}, \,\textstyle{\sum} t_j = t \big\}
  \label{eq:ex:ass:sets}
\end{align}
and take
\begin{align*}
A_{\mathcal{F}}(u, \leq t) = \bigcup_{0<s\leq t} A_{\mathcal{F}}(u, s).  
\end{align*}
These are the \emph{accessibility sets} of $\mathcal{F}$.
\item[(ii)] We say $\mathcal{F}$ is \emph{approximately controllable on $X$} if for any $t > 0$ and any $u \in X$
\begin{align*}
\overline{A_{\mathcal{F}}(u, t)} = X  
\end{align*}
where for $A\subseteq X$, $\overline{A}$ is the closure of $A$.    Equivalently, $\mathcal{F}$ is approximately controllable on $X$ if for any $u, v \in X$ and any $t, \eps > 0$, there exist positive times $t_1, \ldots, t_n$ and elements $\Phi^1, \ldots, \Phi^n \in \mathcal{F}$ such that  $t_1 + \cdots + t_n = t$ and
\begin{align*}
   d( \Phi_{t_n}^n  \Phi_{t_{n-1}}^{n-1} \cdots  \Phi_{t_1}^1 u , v) < \eps.
\end{align*}
\item[(iii)] Suppose that $\pi: X \to Y$ is continuous where $Y$ is another metric space (which we will usually take to be a subset of $X$).   We say that $\mathcal{F}$ is  
	\emph{approximately controllable on $X$ and exactly controllable on $\pi(X)$} if for any $u, v \in X$ and any $\eps, t > 0$,
	there exist positive times $t_1, \ldots, t_n$, $\Phi^1, \ldots, \Phi^n \in \mathcal{F}$ such that $t_1 + \cdots + t_n = t$,  
	\begin{align*}
	  \pi( \Phi_{t_n}^n  \Phi_{t_{n-1}}^{n-1} \cdots  \Phi_{t_1}^1 u ) = \pi(v) \quad \text{ and }\quad   d( \Phi_{t_n}^n  \Phi_{t_{n-1}}^{n-1} \cdots  \Phi_{t_1}^1 u , v) < \eps.
	\end{align*}
\end{itemize}
\end{definition}

\begin{remark}
When $X$ is a Fr\'echet space and $\pi$ is a continuous linear projection onto a finite-dimensional subspace, the notion introduced in Definition~\ref{def:acc:cont} (iii) reduces 
to exact controllability on finite-dimensional projections.  This is the setting in which we provide criteria for establishing (iii) below in Section~\ref{sec:exact:cont} which is based on establishing approximate controllability with continuous dependence on the target point.  Note that this notion of controllability in (iii) above is a slight generalization of the usual notion of simultaneous approximate controllability on $X$ and exact control on a given finite-dimensional projection on $X$ as in~\cite{AgrachevSarychev2005, AgrachevSarychev2006,
  Shirikyan2007, Shirikyan2008, Shirikyan2010, Nersisyan2010,
  Nersesyan2015} since $\pi$ here can be a given continuous mapping and not just a finite-dimensional projection.  
\end{remark}

\subsection{Saturation} \label{sec:sat:lems}
The scaling arguments introduced above in Section~\ref{sec:overview}, in particular in \eqref{eq:rescal:0} and \eqref{eq:rescal:m}, do not immediately yield approximate controllability due to the lack of control over the time parameter.  Thus, in order to identify points in the sets  
\begin{align}
  \overline{A_{\mathcal{F}}(u, t)}, \,\, u\in X\text{ and }t>0, 
  \label{eq:exact:set:rig}
\end{align}
we first determine points which belong to the time-relaxed sets 
\begin{align*}
\overline{A_{\mathcal{F}}(u, \leq t)},\,\, u \in X\text{ and }t>0,
\end{align*}
with the aid of saturation.  Under suitable circumstances, for example when $$\overline{A_{\mathcal{F}}(u, \leq t)} = X$$ for all $u\in X$ and $t>0$,  we can then employ general arguments to obtain information about the exact time sets. This is captured in Lemma~\ref{lem:ex:tm:conv} of this section. 

To see how this works in a general context, we need to introduce some further definitions.   
\begin{definition}
\label{def:sat}
Let $\mathcal{F}, \mathcal{G} \subseteq \mathcal{S}$. 
 \begin{itemize}
 \item[(i)] We say that $\mathcal{G}$ \emph{subsumes} $\mathcal{F}$, written as $\mathcal{F} \preccurlyeq \mathcal{G}$, if 
 \begin{align}
  \overline{A_{\mathcal{F}}(u, \leq t)} \subseteq   \overline{A_{\mathcal{G}}(u, \leq t)},
  \label{eq:sg:inclusion}
\end{align}
for all $u\in X$ and $t>0$, where we recall that $\overline{A}$ denotes the closure of $A\subseteq X$.  
\item[(ii)]
We say that $\mathcal{F}$ is \emph{equivalent} to $\mathcal{G}$, denoted by $\mathcal{F}\sim \mathcal{G}$, if both
$\mathcal{G} \preccurlyeq \mathcal{F}$ and $\mathcal{F} \preccurlyeq \mathcal{G}$.
\item[(iii)] The \emph{saturate of $\mathcal{F}$}, denoted by $\mbox{Sat}(\mathcal{F})$, is defined by
\begin{align*}
   \mbox{Sat}(\mathcal{F}) = \bigcup_{ \mathcal{G} \preccurlyeq \mathcal{F}} \mathcal{G}.
\end{align*} 
\end{itemize}
\end{definition}

The next Lemma gives a characterization of equivalence for collections of semigroups.  This will provide a basic formulation which we will use in applications 
below.  Moreover this formulation is the basis for a generalization to the uniform setting introduced below in Subsection~\ref{sec:exactcontrol}.
\begin{lemma}[Saturation Lemma]\label{lem:sat}  For any collections $\mathcal{F}, \mathcal{G} \subseteq \mathcal{S}$,
$\mathcal{F} \preccurlyeq \mathcal{G}$ if and only if
for every $\Psi \in \mathcal{F}$, $u \in X$ and $\eps, t>0$ with $\Psi_t u\in X$, there exists  $\Phi^1, \ldots,   \Phi^n \in \mathcal{G}$ and positive times $t_1, \ldots, t_n$
such that $\sum t_j  \leq t$ and
\begin{align}
     d( \Phi_{t_n}^n  \Phi_{t_{n-1}}^{n-1} \cdots  \Phi_{t_1}^1 u , \Psi_t u )  < \eps.
     \label{eq:equiv:approx}
\end{align}
Moreover, given any collection
$ \mathcal{H}^i \subseteq \mathcal{S}$ such that $\mathcal{H}^i \preccurlyeq  \mathcal{G}$ for every $i$, then  $\mathcal{G} \sim \bigcup_i \mathcal{H}^i \cup \mathcal{G}$.  
In particular, $\mathcal{F} \sim   \mbox{\emph{Sat}}(\mathcal{F}).$
\end{lemma}

In the following proofs, we will frequently encounter expressions of the form
\begin{align*} 
\Phi_{t_n}^{n}  \Phi_{t_{n-1}}^{n-1} \cdots \Phi_{t_1}^{1} u
\end{align*}   
where $\Phi^j \in \mathcal{F}$.  As such, we will write the product above as $\prod_{i=1}^n \Phi_{t_i}^i u.$  Also, we introduce the notation
\begin{align}
\mathbf{\Phi}_s u := \sum_{j =1}^n
\Phi_{s- s_{j-1}}^j \Phi_{t_{j-1}}^{j-1} \cdots \Phi_{t_1}^1u \indFn{[s_{j-1}, s_{j})}(s)
\label{eq:comp:not}
\end{align}
where $s_0 =0$ and $s_j = \sum_{k = 1}^j t_j$ so that
\begin{align*}
\mathbf{\Phi}_s u   = \Phi^j_{s - s_{j-1}} \Phi_{t_{j-1}}^{j-1}  \cdots \Phi_{t_1}^1 u \quad \text{ when } \quad s_{j-1} \leq  s < s_{j}.
\end{align*}
We offer the abuse of notation $\mathbf{\Phi} \in \mathcal{F}$ when the context is clear.  

\begin{proof}[Proof of Lemma~\ref{lem:sat}]
It is clear that if $\mathcal{F} \preccurlyeq \mathcal{G}$, then the property 
in \eqref{eq:equiv:approx} holds for all $\Phi \in \mathcal{F}$.  Suppose now that the characterization leading to \eqref{eq:equiv:approx} is assumed.  To infer $\mathcal{F} \preccurlyeq \mathcal{G}$ we will prove that $A_{\mathcal{F}}(u, \leq t) \subseteq   \overline{A_{\mathcal{G}}(u, \leq t)}$ for any $u \in X$ and $t >0$.  Fix $\eps, t>0$ and $v\in A_{\mathcal{F}}(u, \leq t)$.  By hypothesis, there exists $\Phi^1, \Phi^2, \ldots, \Phi^n \in \mathcal{F}$ and times $t_1, t_2, \ldots, t_n >0$ with $\sum t_j \leq t$ and  
\begin{align*}
v = \prod_{i=1}^n \Phi_{t_i}^i u .  
\end{align*} 
By induction on $n\geq 1$, we will prove that there exists $\mathbf{\Psi} \in \mathcal{G}$ such that  $d( v, \mathbf{\Psi}_s u ) < \eps$ for some $s\leq \sum t_i$.  If $n=1$ in the product above, then by the hypothesis there is nothing to prove.  Supposing that $n\geq 2$ we may write the product as 
\begin{align*}
v = \prod_{i=1}^n \Phi_{t_i}^i u = \Phi^n_{t_n}  \prod_{i=1}^{n-1} \Phi^i_{t_i} u.  
\end{align*}
First, invoking the continuity of $\Phi^n_{t_n}$, we may pick $\delta >0$ such that for all $w\in X$: 
\begin{align}
d\bigg( w, \prod_{i=1}^{n-1} \Phi^{i}_{t_i} u \bigg)< \delta \,\,\,\,\text{ implies }\,\, \,\, d ( \Phi_{t_n}^n w, v)< \frac{\eps}{2}.  
  \label{eq:sat:lm:1}
\end{align}
By the inductive hypothesis, we may pick $\mathbf{\Psi}^1 \in \mathcal{G}$ such that  
\begin{align*}
  d \bigg(\mathbf{\Psi}_{s_1}^1 u ,\prod_{i=1}^{n-1} \Phi^{i}_{t_i} u \bigg)< \delta 
\end{align*}
for some $s_1\leq \sum_{i\leq n-1} t_i$.  Also by hypothesis and \eqref{eq:sat:lm:1}, we may pick $\mathbf{\Psi}^2 \in \mathcal{G}$ such that 
\begin{align*}
d \bigg(\mathbf{\Psi}^2_{s_2}\mathbf{\Psi}_{s_1}^1 u ,\Phi_{t_n}^n \mathbf{\Psi}_{s_1}^1 u\bigg)< \frac{\eps}{2}
\end{align*}
for some $s_2 \leq t_n$.  
The triangle inequality then implies 
\begin{align*}
d \bigg(\prod_{i=1}^n \Phi^i_{t_i} u, \mathbf{\Psi}^2_{s_2}\mathbf{\Psi}_{s_1}^1 u \bigg) \leq d \bigg(\prod_{i=1}^n \Phi^i_{t_i} u , \Phi^n_{t_n} \mathbf{\Psi}^1_{s_1} u\bigg) + d \bigg(\Phi^n_{t_n} \mathbf{\Psi}^1_{s_1} u , \mathbf{\Psi}^2_{s_2}\mathbf{\Psi}_{s_1}^1 u \bigg) < \frac{\eps}{2}+ \frac{\eps}{2}=\eps.  
\end{align*}
This finishes this part of the proof as $s_1+s_2\leq \sum t_i$.

To address the second property, it is obvious that $\mathcal{G} \preccurlyeq  \bigcup_i \mathcal{H}^i \cup \mathcal{G}$.  On the other hand
$ \bigcup_i \mathcal{H}^i \cup \mathcal{G} \preccurlyeq \mathcal{G}$ follows immediately from the characterization of containment given
by  \eqref{eq:equiv:approx}.   This completes the proof.

\end{proof}

We now state and prove the `Conversion Lemma' which allows us to convert between the relaxed sets and the exact time sets.  Its statement and proof are fundamentally different than in the works of Jurdjevic and Kupka~\cite{JK_81, JK_85, Jur_97} because we cannot rely on topological properties of the underlying space $X$.    

\begin{lemma}[Conversion Lemma]\label{lem:ex:tm:conv}
Suppose that $\mathcal{F} \subseteq \mathcal{S}$  and that $V\subseteq X$ is open with the property that 
\begin{align*}
V\subseteq \overline{A_{\mathcal{F}}(u, \leq t)},  
\end{align*}
for all $u\in V$, $t>0$.  Then, 
\begin{align*}
V\subseteq \overline{A_{\mathcal{F}}(u, t)},
\end{align*}  
for all $t>0$ and every $u\in V$.
\end{lemma}
We have the following  immediate, but important corollary.
\begin{corollary}
Suppose that $\mathcal{F} \subseteq \mathcal{S}$ is such that $\overline{A_{\text{\emph{Sat}}(\mathcal{F})}(u, \leq t)} = X$ for any $t > 0$ and any $u \in X$.  Then
 $\mathcal{F}$ is approximately controllable on $X$ in the sense of Definition~\ref{eq:equiv:approx}.
\end{corollary}

\begin{proof}[Proof of Lemma~\ref{lem:ex:tm:conv}]
Fix any $u, v \in V$ and any $\eps, t >0$.  We will establish the desired result by showing that there is a corresponding $\mathbf{\Phi} \in \mathcal{F}$
such that $d( \mathbf{\Phi}_t u, v) < \eps$, where we are maintaining the notational convention introduced above in \eqref{eq:comp:not}.
Observe that, without loss of generality, we may suppose that $\eps >0$ is such that $B(v, \eps) \subseteq V$.

As a first step pick any $\psi^* \in \mathcal{F}$.  Invoking continuity we may choose $\eps' \in(0, \eps)$ such that
\begin{align*}
  \sigma := \inf_{\tilde{v} \in B(v, \eps')} \Big\{\inf\{ s> 0 \, : \, d( \psi^*_s \tilde{v}, v) > \eps \}\Big\}>0.  
\end{align*}
By assumption, we may pick $\mathbf{\Psi}^0 \in \mathcal{F}$ so that 
\begin{align*}
  v_0 := \mathbf{\Psi}_{\tau^0}^0 u \in B(v, \eps')
\end{align*}
for some $\tau^0 \leq t$ (See Figure~\ref{fig:con:lem}).  
\begin{figure}[tb]
 \centering
 \includegraphics{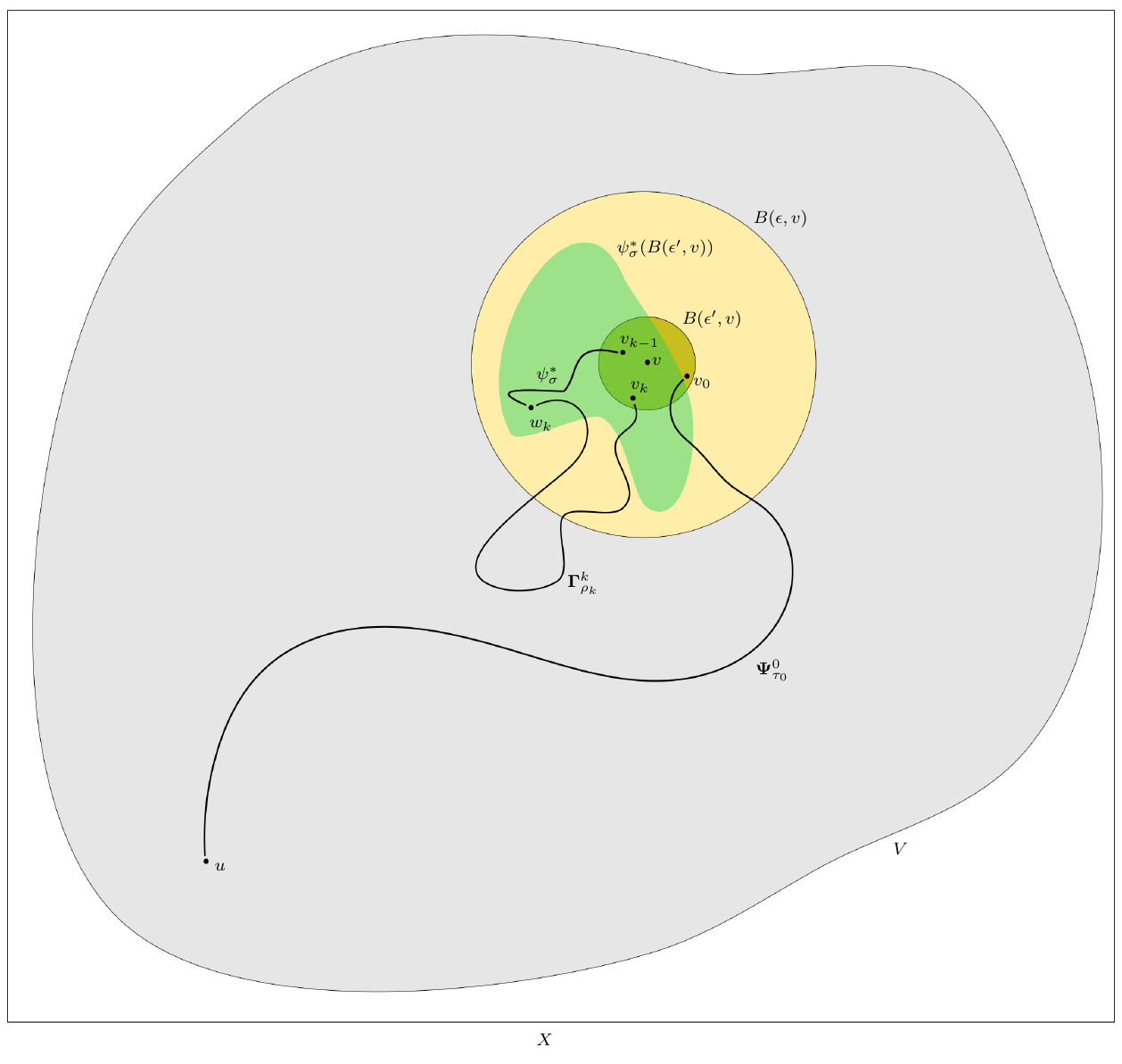}
 \vspace{-.2in}
 {\scriptsize
 \caption{A visual representation of the `pinball strategy' used to prove Lemma~\ref{lem:ex:tm:conv}. \label{fig:con:lem}}}
 \end{figure}
If it happens that $\tau^0 + \sigma \geq t$ then we simply take
\begin{align*}
  \mathbf{\Phi}_s u := \mathbf{\Psi}_{s}^0 u\indFn{[0, \tau^0)}(s) + \psi^*_{s}\mathbf{\Psi}_{\tau^0}^0 \indFn{[\tau^0, \tau^0 + \sigma)}(s)
  \end{align*}
  and observe that $\mathbf{\Phi}_t u \in B(v, \eps)$.  Otherwise, we may find an integer $n \geq 1$ so that
\begin{align*}
	\tau^0 + n \sigma < t \leq  \tau^0 + (n+1) \sigma.
\end{align*}
In this later case, define elements $v_k\in B(v, \eps'), w_k \in B(v, \eps)$, $\mathbf{\Psi}^k\in \mathcal{F}$ and $\tau^k$, $k=1,2,\ldots, n$,  inductively as follows.
Let 
\begin{align*}
	\delta := t - (\tau^0 + n \sigma) < \sigma.
\end{align*}
Having found $v_{k-1}$ define
\begin{align*}
	w_k = \psi^*_\sigma v_{k-1}.
\end{align*}
Now, under the given assumptions, we find a $0 < \rho_k \leq \delta/n$ and $\mathbf{\Gamma}^k \in \mathcal{F}$ such that 
\begin{align*}
v_k :=\mathbf{\Gamma}^k_{\rho_k} w_k \in B(v, \eps').
\end{align*}
Defining $\tau^k = \sigma + \rho_k$ and the map $\mathbf{\Psi}_s^k$ on the interval $s\in [0, \tau^k]$ by 
\begin{align*}
\mathbf{\Psi}^k_s := \psi^*_{s} \indFn{[0, \sigma)}(s) + \mathbf{\Gamma}^k_{s}\psi^*_\sigma  \indFn{[\sigma, \tau^k]}(s)
\end{align*}
we find that $\mathbf{\Phi} \in \mathcal{F}$ given by 
\begin{align*}
\mathbf{\Phi}_s u&= \psi^*_{s- \sum_{l=0}^n \tau^l} \mathbf{\Psi}^n_{\tau^n} \cdots \mathbf{\Psi}^0_{\tau^0}u \indFn{[\sum_{l=0}^{n}\tau^l, \tau^0 + (n+1)\sigma ]}(s) \\
&\qquad +\sum_{j=0}^{n} \mathbf{\Psi}^j_{s-\sum_{l=0}^{j-1}\tau^l} \mathbf{\Psi}^{j-1}_{\tau^{j-1}} \cdots \mathbf{\Psi}^0_{\tau^0} u  \indFn{[\sum_{l=0}^{j-1}\tau^l, \sum_{l=0}^{j}\tau^l)}(s)
\end{align*}
has $\mathbf{\Phi}_t \in B(v, \eps)$.  The proof is now complete.
\end{proof}

\subsection{Uniform Saturation}
\label{sec:exactcontrol}
In order to obtain results concerning exact controllability on certain projections, the above framework is not sufficient.    As we are specifically interested in finite-dimensional projections, we introduce an extension of the above formalism which provides a way of tracking the continuity of controls with respect to parameters.  This continuity is then used in conjunction with the Brouwer fixed point theorem below in Section~\ref{sec:exact:cont} to establish an abstract result suitable for our applications below.  

We begin by extending our notion of continuous semigroups, Definition~\ref{def:cont:sg}, to include
parameter dependence.   For what follows we consider an auxiliary metric space $(Y, d_Y)$.
\begin{definition}
\label{def:oneparamlocal}
Let $Z \subseteq Y$ be non-empty.  We call a function 
\begin{align*}
(t,u,p) \mapsto \Phi_t^p u: [0, \infty) \times X\times Z \rightarrow X\cup \{ \death\} 
\end{align*}
a \emph{one-parameter family of continuous local semigroups on $X$ parametrized by $Z$} if, for every $u\in X$ and $p \in Z$, there exists a $T_{u,p}>0$, called \emph{the time of existence},
for which the following conditions are met:
\begin{itemize}
\item[(i)]  For $t\in [0, T_{u,p})$ we have $\Phi_t^p u  \in X$ and for $t\geq T_{u,p}$ we have $\Phi_t^p u=\death$.   
\item[(ii)] $\Phi_0^p u= u$ and for all $s,t\geq 0$ with $s+t \in [0, T_{u,p})$ one has
  $t\in[0, T_{\Phi_s^p u,p})$ and $\Phi_{t+s}^p u= \Phi_t^p \Phi_s^p u$.
\item[(iii)] For all $t\in [0, T_{u,p})$ and all $\eps >0$, there exists $\delta >0$ such that 
whenever $(t', u',p') \in [0, \infty)\times X\times Z$ satisfies 
\begin{align*}
 |t - t'| + d(u, u') + d_Y(p, p') < \delta
\end{align*}
we have $t'\in [0, T_{u',p'})$ and 
\begin{align*}
  d(\Phi_t^p u, \Phi_{t'}^{p'} u') < \eps.  
\end{align*}  
\end{itemize}
Analogously to Definition~\ref{def:cont:sg} above, we abbreviate $\Phi$ for 
this mapping or write $(\Phi, Z)$ when we need to emphasize the associated parameter set $Z \subseteq Y$.  For $p \in Z$ 
we write $\Phi^p$ for the element in $\mathcal{S}$ defined by 
$(t, u) \mapsto \Phi^p_t u: [0,\infty) \times X \to X \cup \{\death\}$.
\end{definition}   

Before proceeding further, we introduce some useful notations. 
 For $Z\subseteq Y$, we let $\mathfrak{S}(Z)$ denote the collection of one-parameter 
 continuous local semigroups on $X$ parametrized by $Z$, and define
\begin{align*}
\mathfrak{S} = \bigcup_{\substack{Z \subseteq Y}} \mathfrak{S}(Z).  
\end{align*}    
A generic element of $\mathfrak{S}$ will be denoted by $\Phi$ and $\mathcal{P}(\Phi)$ will denote the parameter set of $\Phi$; that is, $\mathcal{P}(\Phi) = Z$ 
means that $\Phi \in \mathfrak{S}(Z)$.  We will use $\mathfrak{F}$ to denote an arbitrary subset of $\mathfrak{S}$ 
using this typographic choice to distinguish between subsets $\mathcal{F}$ of $\mathcal{S}$ introduced in the previous 
section in Definition~\ref{def:cont:sg}.   
Given  $\mathfrak{F} \subseteq \mathfrak{S}$ we associate a subset of $\mathcal{S}$ according to
\begin{align}
  \mathbb{D}(\mathfrak{F}) := \{ \Phi^p : \Phi \in \mathfrak{F}, p \in \mathcal{P}(\Phi) \}.
  \label{eq:unwrap:opp}
\end{align}

We now define the analogue of the saturate, which we call the \emph{uniform saturate}, in this setting.   
This builds on the characterization of equivalent collections of semigroups revealed by Lemma~\ref{lem:sat}.

\begin{definition}\label{def:uni:sat}
Suppose that $\mathfrak{F}, \mathfrak{G} \subseteq \mathfrak{S}$. 
\begin{itemize}
 \item[(i)]  We say that $\mathfrak{G}$ \emph{uniformly subsumes} $\mathfrak{F}$, denoted $\mathfrak{F} \preccurlyeq_u \mathfrak{G}$  if, 
for any $\Psi \in \mathfrak{F}$, $t, \eps > 0$, and any compact subsets $K_1\subseteq X$, $K_2 \subseteq
\mathcal{P}(\Psi)$, there exists  $\Phi^1, \ldots,   \Phi^n \in \mathfrak{G}$, times $t_1, \ldots, t_n > 0$ 
and continuous functions $f_k:  K_2 \to\mathcal{P}(\Phi^k)$,
$k = 1, \ldots, n$, such that $\sum t_j  \leq t$ and
\begin{align}
   \sup_{u \in K_1, p \in K_2}  d\Big(  \Psi_t^p u , \prod_{k=1}^n  \Phi_{t_{k}}^{k, f_k(p)} u \Big)  < \eps.
   \label{eq:uni:stat:def}
\end{align}
We say that $\mathfrak{F}$ and $\mathfrak{G}$ are \emph{uniformly equivalent}, denoted $\mathfrak{F} \sim_u \mathfrak{G}$, if both $\mathfrak{F} \preccurlyeq_u \mathfrak{G}$
and $\mathfrak{G} \preccurlyeq_u \mathfrak{F}$.
\item[(ii)] The \emph{uniform saturate} of $\mathfrak{F}$, denoted  $\text{Sat}_u(\mathfrak{F})$ is taken to be 
\begin{align*}
  \text{Sat}_u(\mathfrak{F}) := \bigcup_{ \mathfrak{G} \preccurlyeq_u \mathfrak{F}} \mathfrak{G}.
\end{align*}
\end{itemize}
\end{definition}     

\begin{remark}\label{rmk:u:imply:non:u}
It is worth emphasizing that uniform subsumption and uniform saturation imply regular subsumption and saturation.  More precisely 
if $\mathfrak{F} \preccurlyeq_u \mathfrak{G}$ then $\mathbb{D}(\mathfrak{F}) \preccurlyeq \mathbb{D}(\mathfrak{G})$. 
With the fact that $\mathfrak{F} \sim_u \text{Sat}_u(\mathfrak{F})$ this implies
\begin{align}
  \mathbb{D}( \text{Sat}_u(\mathfrak{F}))  \sim \text{Sat}(\mathbb{D}(\mathfrak{F})).
  \label{eq:D:Sat:com:ob}
\end{align}
We use this observation below in Corollary~\ref{cor:exact:con}.
\end{remark}

We next show that an analog of Lemma~\ref{lem:sat} holds in the setting of Definition~\ref{def:uni:sat}.  Here, however, we have to be careful
to show that $\preccurlyeq_u$ is in fact a transitive relation on $\mathfrak{S}$.

\begin{lemma}
\label{lem:uni:stat:m}
If $\mathfrak{F} \preccurlyeq_u \mathfrak{G}$ and $\mathfrak{G} \preccurlyeq_u \mathfrak{H}$ then $\mathfrak{F} \preccurlyeq_u \mathfrak{H}$.
Moreover 
if $\mathfrak{F}^i \preccurlyeq_u \mathfrak{G}$ then 
\begin{align}
\bigcup_i \mathfrak{F}^i \cup \mathfrak{G} \sim_u \mathfrak{G},
\label{eq:sat:uni}
\end{align}
so that, in particular, $\mathfrak{F} \sim_u \text{\emph{Sat}}_u(\mathfrak{F}).$

\end{lemma}
The approach here mimics the proof of Lemma~\ref{lem:sat} but requires
a little more bookkeeping.  In particular, in order to use a bound analogous to \eqref{eq:equiv:approx} in the
proof of Lemma~\ref{lem:sat} we make the following
elementary observation concerning compactness and continuity:
\begin{lemma}
\label{lem:gen:comp}
Consider two continuous mappings $f, g: Y \times X \to X$.  Then, given any compact
sets $K \subseteq Y, \tilde{K} \subseteq X$ and any $\eps > 0$ there
exists $\delta >0$ such for any function $h: Y \times X \to X$ with
\begin{align*}
  \sup_{p \in K, x \in \tilde{K}} d(g(p, x), h(p,x)) < \delta 
\end{align*}
we have that
\begin{align*}
    \sup_{p \in K, x \in \tilde{K}} d(f(p, g(p,x)), f(p, h(p,x))) < \eps. 
\end{align*}
\end{lemma}

We this in hand we turn to the proof of Lemma~\ref{lem:uni:stat:m}.
\begin{proof}[Proof of Lemma~\ref{lem:uni:stat:m}]
The main step is to establish the desired transitivity in the uniform
subsumption relation.  To this end let $\Psi \in \mathfrak{F}$, $\eps, t > 0$ and $K_1 \subseteq X$, $K_2 \subseteq \mathcal{P}(\Psi)$,
both compact sets, be given. Suppose for $n
\geq 1$ and $m_k \geq 1$, $k = 1, \ldots, n$, we have $\Phi^k \in \mathfrak{G}$, $\Gamma^{k, l} \in \mathfrak{H}$, 
along with continuous functions $f_k: K_2 \to \mathcal{P}(\Phi^k)$, $g_{k,l}: f_k(K_2) \to \mathcal{P}(\Gamma^{k,l})$
and times $t_k > 0$, $s_{k,l} > 0$ with
\begin{align*}
  t_1 + \cdots + t_n \leq t, \quad s_k := s_{k,1} + \cdots + s_{k, m_k} \leq t_k, \text{ for } k = 1, \ldots, n.
\end{align*}
Analogously to \eqref{eq:comp:not} above we adopt the abbreviated notation
\begin{align*}
  \mathbf{\Gamma}^{k; h_k(p)}_{s_k} := \prod_{l =1}^{m_k} \Gamma^{k,l; h_{k,l}(p)}_{s_{k,l}}
  \quad \,\text{ where }\, \quad h_{k,l}(p) := g_{k,l}(f_k(p)) \quad p \in K_2.
\end{align*}
Observe that, by the triangle inequality,
\begin{align}
   \sup_{u \in K_1, p \in K_2}& d\biggl( \Psi_t^p u, \prod_{k =1}^n \mathbf{\Gamma}^{k; h_k(p)}_{s_k}u\biggr)
   	\leq \sup_{u \in K_1, p \in K_2} d\biggl( \Psi_t^p u, \prod_{k =1}^n \Phi^{k; f_k(p)}_{t_k}u\biggr)
	\notag\\
	 +& \sum_{l = 1}^n \sup_{u \in K_1, p \in K_2}d \biggl( 
		\prod_{k = l}^n \Phi^{k; f_k(p)}_{t_k} \prod_{k' =1}^{l-1} \mathbf{\Gamma}^{k'; h_{k'}(p)}_{s_{k'}}u,  
		\prod_{k =l+1}^n \Phi^{k; f_k(p)}_{t_k} \prod_{k' =1}^{l} \mathbf{\Gamma}^{k'; h_{k'}(p)}_{s_{k'}}u\biggr)
		\notag\\
	\leq& \sup_{u \in K_1, p \in K_2} d\biggl( \Psi_t^p u, \prod_{k =1}^n \Phi^{k; f_k(p)}_{t_k}u\biggr)
	\notag\\
	 +& \sum_{l = 1}^n \sup_{v \in K_1^l , p \in K_2}d \biggl( 
		\prod_{k = l+1 }^n \Phi^{k; f_k(p)}_{t_k} \Phi^{l; f_l(p)}_{t_l}v, 
		\prod_{k =l+1}^n \Phi^{k; f_k(p)}_{t_l}  \mathbf{\Gamma}^{l; h_{l}(p)}_{s_{l}}v\biggr)
		\label{eq:pr:st:ineq}
\end{align}
where 
\begin{align*}
  K_1^l := \biggl\{ \prod_{k =1}^{l-1} \mathbf{\Gamma}^{k; h_{k}(p)}_{s_{k}}u: u \in K_1, p \in K_2 \biggr\}
\end{align*}
so that, under the standing continuity assumptions, $K_1^l$ is a compact subset of $X$ for $l =1, \ldots, n$.
Note that as above in Lemma~\ref{lem:sat} we are maintaining the convention that $\prod_{k = n+1}^n \Phi^{k; f_k(p)}_{t_k} = \text{Id} =  
\prod_{k' =1}^{-1} \mathbf{\Gamma}^{k'; h_{k'}(p)}_{s_{k'}}$.

Under our assumption that $\mathfrak{F} \preccurlyeq_u
\mathfrak{G}$, we may choose $n \geq 1$, and elements $\Phi^{k}$'s and 
$f_k$'s such that
\begin{align}
	\sup_{u \in K_1, p \in K_2} d\biggl( \Psi_t^p u, \prod_{k =1}^n \Phi^{k; f_k(p)}_{t_k}u\biggr) < \frac{\eps}{2}.
	\label{eq:sb:bd:f:g}
\end{align}
Next, according to Lemma~\ref{lem:gen:comp} we choose $\delta_l > 0$, 
so that, for any $h: Y \times X \to X$, the bound
\begin{align}
\sup_{v \in K_1^l, p \in K_2}
		d \biggl( \prod_{k = l+1 }^n \Phi^{k; f_k(p)}_{t_k} \Phi^{l; f_l(p)}_{t_l}v,  
		\prod_{k =l+1}^n \Phi^{k; f_k(p)}_{t_l} h(p, v)\biggr) < \frac{\eps}{2n}
		\label{eq:ug:tri:bnd:1}
\end{align}		
holds whenever
\begin{align*}
  \sup_{v \in K_1^l, p \in K_2}
		d ( \Phi^{l; f_l(p)}_{t_l}v, h(p, v)) < \delta_l.
\end{align*}
On the other hand, invoking that $\mathfrak{G} \preccurlyeq_u
\mathfrak{H}$ and referring back to \eqref{eq:uni:stat:def}  we may
choose $\mathbf{\Gamma}^{l}$, $g_l$ and $s_l$ such that 
\begin{align}
     \sup_{v \in K_1^l, p \in K_2}d \biggl(  \Phi^{l; f_l(p)}_{t_l}v, \mathbf{\Gamma}^{l; h_{l}(p)}_{s_{l}}v\biggr) 
     = \sup_{v \in K_1^l, q \in f(K_2)}d \biggl(  \Phi^{l; q}_{t_l}v, \mathbf{\Gamma}^{l; g_{l}(q)}_{s_{l}}v\biggr) 
< \delta_l.
   \label{eq:ug:tri:bnd:2}
\end{align}
Thus \eqref{eq:ug:tri:bnd:1}, \eqref{eq:ug:tri:bnd:2} yield
\begin{align}
    \sup_{v \in K_1^l, p \in K_2}&d \biggl( 
		\prod_{k = l+1 }^n \Phi^{k; f_k(p)}_{t_k} \Phi^{l; f_l(p)}_{t_l}v,  
		\prod_{k =l+1}^n \Phi^{k; f_k(p)}_{t_l}  \mathbf{\Gamma}^{l; h_{l}(p)}_{s_{l}}v\biggr)
	< \frac{\eps}{2n}.
	\label{eq:ug:tri:bnd:3}
\end{align}
By combining \eqref{eq:pr:st:ineq} with the bounds \eqref{eq:sb:bd:f:g}
and \eqref{eq:ug:tri:bnd:3} and recalling that $\eps > 0$ was
arbitrary we now conclude that $\mathfrak{F}\preccurlyeq_u
\mathfrak{H}$, as desired.  

As above in Lemma~\ref{lem:sat}, \eqref{eq:sat:uni} and uniform saturation follows immediately from \eqref{eq:uni:stat:def}.  The proof is now complete.

\end{proof}
        
\subsection{Exact Controllability for Projections}
\label{sec:exact:cont}

We are now prepared to state the main final abstract result of this section.  Here we specialize and 
assume that $X$ is a Fr\'echet space so that, in particular, an
addition operation, $+$, is defined and the metric $d$ is shift
invariant with respect to $+$; namely, we have the property that 
\begin{align*}
d(u, v) = d(u + w, v+w), \,\,\, \text{ for any} \,\,\, u,v, w \in X.
\end{align*}

\begin{theorem}
\label{thm:exact:cont:abs}
Consider $\mathfrak{F} \subseteq \mathfrak{S}$ and let $\pi: X \to X$
be a linear, continuous projection operator mapping onto a finite-dimensional subspace $W = \pi(X)$ of $X$.   For $v \in X$ define 
\begin{align*}
   B^{\pi, v}(\eps) := \{ v+  w : d(0, w) \leq  \eps, w \in W\}.
\end{align*}
Suppose the following: 
\begin{itemize}
\item[(i)] $\mathbb{D}(\mathfrak{F})$ is approximately controllable
on $X$; that is, for any $u \in X, \, t > 0$ we have that
$\overline{A_{\mathbb{D}(\mathfrak{F})}(u, t)} = X$ where recall that
$A_{\mathbb{D}(\mathfrak{F})}(u, t)$ is defined as in \eqref{eq:exact:set:rig}.
\item[(ii)] For any $v \in X$ and $\eps,\eps', t > 0$, there exists an initial condition 
	$\tilde{u} \in X$, elements 
	$\Phi^1, \ldots,  \Phi^k \in \mathfrak{F}$,
	times $s_1, \ldots, s_k >0$,
	and continuous functions $g_j: B^{\pi,v}(\eps) \to \mathcal{P}(\Phi^k)$, $j =1, \ldots, k$,
	such that $\sum_j s_j \leq t$ and
	\begin{align}
	   \sup_{\tilde{v} \in B^{\pi,v}(\eps)} d\biggl( \prod_{j=1}^k \Phi^{j; g_j(\tilde{v})}_{s_j}\tilde{u} , \tilde{v}\biggr) < \eps'.
	   \label{eq:main:stat:cond}
	\end{align}
\end{itemize}	
	Then, for all $u, v \in X$, $t > 0$ and $\eps > 0$ there exists $\Psi^1, \ldots, \Psi^n \in \mathbb{D}(\mathfrak{F})$, $t_1, t_2, \ldots, t_n>0$
	such that $t_1 + t_2 + \cdots + t_n =t$ and
	\begin{align}
	   \pi( \Psi^n_{t_n} \cdots \Psi^1_{t_1}u) = \pi(v), \quad d(\Psi^n_{t_n} \cdots \Psi^1_{t_1}u, v) < \eps.
	   \label{eq:cont:con}
	\end{align}
	In other words, $\mathbb{D}(\mathfrak{F})$ is approximately controllable on $X$ and exactly controllable
	on $W = \pi(X)$ in the sense of Definition~\ref{def:acc:cont}.
\end{theorem}
\begin{proof}
 Fix any $u, v \in X$ and any $t, \eps > 0$.  Observe that for  any 
 $\eps' > 0$, we may invoke assumption (ii) and 
 choose $\tilde{u} \in X$, 
 $\Phi^1, \ldots, \Phi^k \in \mathfrak{F}$, $s_1, \ldots, s_k >0 $ and 
continuous functions $g_j: B^{\pi, v}(\eps/2) \to \mathcal{P}(\Phi^k)$, $j=1,2,\ldots, k$, such that
 $s:=\sum_{j=1}^k s_j \leq t/2$ and 
 	\begin{align*}
	   \sup_{\tilde{v} \in B^{\pi,v}(\eps/2)} d\biggl( \prod_{j=1}^k \Phi^{j, g_j(\tilde{v})}_{s_j}\tilde{u} , \tilde{v}\biggr) < \frac{\eps'}{2}.
	\end{align*}
Using continuity, compactness of the closed ball $B^{\pi,v}(\eps/2)$ and assumption (i), we can 
pick $\Psi^1 \in \mathbb{D}(\mathfrak{F})$ to ensure that
\begin{align*}
  \sup_{\tilde{v} \in B^{\pi,v}(\eps/2)} d\biggl( \prod_{j=1}^k \Phi^{j, g_j(\tilde{v})}_{s_j}\Psi^1_{t-s}u ,  
  \prod_{j=1}^k \Phi^{j, g_j(\tilde{v})}_{s_j}\tilde{u}\biggr) < \frac{\eps'}{2}
\end{align*}
and hence infer that
\begin{align}
 	   \sup_{\tilde{v} \in B^{\pi,v}(\eps/2)} 
	   d\biggl( \prod_{j=1}^k \Phi^{j, g_j(\tilde{v})}_{s_j}\Psi^1_{t-s}u - \tilde{v}, 0\biggr) < \eps',
	   \label{eq:glob:cont}
\end{align}
where we have also used the shift invariance of $d$.  

With this bound \eqref{eq:glob:cont} in hand, we now invoke the continuity of $\pi$ and pick $\eps' >0$ 
such that
\begin{align}
  \eps' \leq \frac{\eps}{2} \text{ and whenever } d(y, 0) < \eps' 
   \text{ then } \|\pi(y) \|_W < \frac{\eps}{2}.
   \label{eq:uni:pi}
\end{align}
For this value of $\eps'$ and the corresponding values of $\Phi^j$,
$g$, $\Psi$ etc. leading to \eqref{eq:glob:cont} we next define 
\begin{align*}
    G(w) = \pi\left( v + w - \prod_{j=1}^k \Phi^{j, g_j(v+ w)}_{s_j}\Psi^0_{t -s} u \right)
\end{align*}
for $w\in W$ satisfying $d(0,w)< \frac{\eps}{2}$.  
According \eqref{eq:uni:pi}, \eqref{eq:glob:cont} $G$ defines a continuous map on
\begin{align*}
\{w\in W \, : \, d(0, w) < \eps/2 \}.
\end{align*}
into itself.  Recalling that $W$ is a finite-dimensional subspace, 
we infer from the Brouwer fixed point theorem that there exists $w^*
\in W$ such that 
\begin{align*}
  v+ w^* \in B^{\pi, v}(\eps/2) \,\, \text{  and } \, \,  \pi( v)  = \pi\bigg( \prod_{j=0}^{k} \Psi^{j}_{t_j}u \bigg), 
\end{align*}
where $t_0 = t -s$ and $\Psi^{j}_{t_j} = \Phi^{j, g_{j}(v +
  w^*)}_{s_{j}} \in \mathbb{D}(\mathfrak{F})$, $j = 1, \ldots, k$.
This is the first condition in \eqref{eq:cont:con}.
In view of \eqref{eq:glob:cont} and our choice that $\eps' \leq \eps/2$ it is clear that this control
also satisfies the global approximation condition in \eqref{eq:cont:con}.  The proof is therefore
complete.
\end{proof}

We will make use of the following corollary of Theorem~\ref{thm:exact:cont:abs} 
in the examples considered below in
Section~\ref{sec:examples}.  To state this result, let us first recall and
extend the `ray semigroup' notation introduced above in \eqref{eq:ray:sg}.
Following the notational convention introduced in
Definition~\ref{def:oneparamlocal}, given any $Y \subseteq X$
we take $(\rho,Y)$ to be the \emph{ray semigroup
  parameterized by $Y$}; namely,
 \begin{align}
  (t,u,p) \mapsto \rho^p_tu: [0,\infty) \times X \times
  Y \to X \quad \text{ where } \quad \rho^p_tu := u + t p.
   \label{eq:para:ray:semi}
 \end{align}

\begin{corollary}\label{cor:exact:con}
Suppose that $\mathfrak{F} \subseteq \mathfrak{S}$ and suppose that 
$X_n$ is an increasing sequence of subspaces of $X$
such that $(\rho, X_n) \in \text{\emph{Sat}}_u(\mathfrak{F})$.  If $\cup_n X_n$ is dense in $X$ then $\mathbb{D}(\mathfrak{F})$
is approximately controllable on $X$ and exactly controllable for 
any continuous projection mapping into a finite dimensional subspace
as in Definition~\ref{def:cont:sg}.
\end{corollary}

\begin{proof}
We will establish this result by showing that the conditions in Theorem~\ref{thm:exact:cont:abs} hold.
Regarding the first condition (i), fix any $u, v \in X$ and $t > 0$.  By our density assumption on $\cup_n X_n$ we 
may consider a sequence of 
elements $w_n \in X_n$ such that $w_n \to w = \frac{v -u}{t}$.  This in turn implies that
$\rho^{w_n}_t u \to v$ as $n \to \infty$. On the other hand we have by assumption that
$\rho^{w_n} \in \mathbb{D}(\text{Sat}_u(\mathfrak{F}))$.  Referring back to \eqref{eq:D:Sat:com:ob} this means that $\rho^{w_n} \in \text{Sat}(\mathbb{D}(\mathfrak{F}))$ for all $n$.
We thus conclude that
\begin{align*}
v \in \overline{A_{\text{Sat}(\mathbb{D}(\mathfrak{F}))}(u,t)} 
\subseteq \overline{A_{\text{Sat}(\mathbb{D}(\mathfrak{F}))}(u,\leq t)} 
= \overline{A_{\mathbb{D}(\mathfrak{F})}(u,\leq t)},
\end{align*}
where we have use Lemma~\ref{lem:sat} for the last equality.  Since $u, v \in X$, $t >0$ were arbitrary here
this shows that $\overline{A_{\text{Sat}(\mathbb{D}(\mathfrak{F}))}(u,t)} = X$ for any $u \in X$, $t > 0$. 
Thus, by Lemma~\ref{lem:ex:tm:conv}, we have that $\overline{A_{\mathbb{D}(\mathfrak{F})}(u, t)} = X$.  In particular, we have established condition (i) of Theorem~\ref{thm:exact:cont:abs}.

Turning to the second condition in Theorem~\ref{thm:exact:cont:abs}, again fix any $u,v \in X$,
$t > 0$ and a finite dimensional projection $\pi$.  For any given $\eps, \eps' > 0$
we may show that the condition \eqref{eq:main:stat:cond} is satisfied by taking $\tilde{u} = v$.   Indeed,  since
$\pi(X)$ is finite dimensional we can approximate the basis elements $u^{(1)}, \ldots, u^{(N)}$ 
up to any precision $\delta > 0$ by elements in $X_n$ for some $n = n(\pi, \delta)$.   In particular 
this implies that we may choose $\delta$ and $\tilde{u}^{(1)}, \ldots, \tilde{u}^{(N)} \in X_n$ such that
\begin{align*}
 \sup_{|\alpha| \leq  \eps} d(  \rho_t^{\alpha_1 \tilde{u}^{(1)} + \ldots + \alpha_N \tilde{u}^{(N)}} v,\rho_t^{\alpha_1 u^{(1)} + \cdots + \alpha_N u^{(N)}}v) < 
 \frac{\eps'}{2}.
\end{align*}
Combining this observation with the fact that $(\rho,X_n) \in \text{Sat}_u(\mathfrak{F})$ we
thus conclude \eqref{eq:main:stat:cond}, completing the proof.
\end{proof}

\begin{remark}\label{rmk:stupid:spans}
The following observation concerning uniform subsumption
and ray semigroups is used several times below in order 
to establish the conditions for Corollary~\ref{cor:exact:con}.
Maintaining our assumption that the phase space $X$ is a Fr\'echet space
consider a collection $\mathfrak{F} \subset \mathfrak{S}$.  Suppose that 
$Y_1, Y_2$ are linear subspaces of $X$ such that 
\begin{align*}
   (\rho, Y_1), (\rho, Y_2)  \in  \text{Sat}_u(\mathfrak{F}).
\end{align*}
Then an argument similar to the one given in Lemma~\ref{lem:uni:stat:m}
yields that
\begin{align*}
  (\rho, \spa \{ Y_1 \cup Y_2 \}) \in  \text{Sat}_u(\mathfrak{F}).
\end{align*}
\end{remark}

\section{Applications to Stochastic Partial Differential Equation}
\label{sec:app:SPDEs}

We now turn our attention to applying the previous control results to stochastic partial differential equations (SPDEs) of the form 
\begin{align}\label{SPDE}
  \partial_t u + L u + N(u) = f + \sum_{k\in \mathcal{Z}} \sigma_k \partial_t W_k  
\end{align}
where $L$ is a linear operator, $N$ is a nonlinear operator, 
$\sigma_k$ and $f$ are fixed spatial functions, and $\{W_k\}_{k\in \mathcal{Z}}$ is
an independent collection of standard real-valued Brownian motions.
We assume that the set $\mathcal{Z}$ is 
finite and that the phase space $X$ is a Hilbert space with norm $\| \ccdot \|$ and inner product $\langle \ccdot, \ccdot \rangle_X$.  

Observe that by replacing the Brownian motions $W_k(\,\cdot\,)$ with actuators $\int_0^\cdot \alpha_k ds$ 
in \eqref{SPDE} we obtain the control system (\ref{eq:abs:evo}). Our goal in this section is
to illustrate some implications of the controllability of the
system (\ref{eq:abs:evo}) for the SPDE \eqref{SPDE}.  Specifically, we present results concerning
topological irreducibility, unique ergodicity as well as density properties 
of finite-dimensional projections of \eqref{SPDE} for which the
control properties (\ref{eq:approx:cont}) and (\ref{eq:exact:cont}) play a
crucial role.

To avoid the technicalities of defining solutions of~\eqref{SPDE} in a general abstract setting,  we will instead simply posit the existence of a suitable \emph{cocycle} 
$\phi$.  See Definition~\ref{def:co-cyc} below and, for example, \cite{Arnold2013} for a general discussion of this formalism.  
The analysis in this section is carried out from this starting point.  
Below in Section~\ref{sec:examples}, we provide details of a concrete functional setting in each example, 
hence inferring the existence of such a cocycle $\phi$ corresponding to an equation of the form \eqref{SPDE} on a case-by-case basis.

\subsection{Cocycle Setting}
\label{sec:cocyclesetting}
Let us now recall the precise setting of the cocycle formalism.
In the process, we will introduce some assumptions used throughout this section and  notational conventions used throughout the rest of the paper.  

In what follows it will be convenient to take the \emph{Wiener space} as our underlying probability space. 
For this purpose we take the sample space $\Omega$ to be
\begin{align}
\label{eq:omega}
  \Omega := \{ V\colon (-\infty, \infty) \rightarrow \R^{|\mathcal{Z}|}\, \text{ continuous with }\, V(0) = 0 \}
\end{align}
and endow $\Omega$ with the usual topology induced by the semi-norms
\begin{align}
   \| V\|_{\infty, s, t} = \sup_{s\in [s, t]}|V(s)| \quad \text{ for any } \quad V\in \Omega  
    \label{eq:sup:norm}
\end{align}
defined for $-\infty < s < t < \infty$.
Similarly for $t>0$ we take
\begin{align*}
  \Omega_t := \{ V\colon [0, t] \rightarrow \R^{|\mathcal{Z}|} \, \text{ continuous with } \, V(0) = 0 \}.
\end{align*}
We use $\| \ccdot \|_{\infty, t} := \| \ccdot \|_{\infty, 0, t}$ to
denote the sup norm on $\Omega_t$ as in \eqref{eq:sup:norm}.  We will
also make use of the \emph{Cameron-Martin} subspace $\mathcal{H}_t
\subseteq \Omega_t$ defined as
\begin{align}
   \mathcal{H}_t := \{ H \in H^1([0, t];\R^{|\mathcal{Z}|}) : H(0) = 0\}
   \label{eq:Cam:Mar:Sp}
\end{align}
and endowed with the inner product
\begin{align*}
   \langle H, G \rangle_{\mathcal{H}_t} = \int_0^t \dot{H} \dot{G} ds ,
\end{align*}
for any $H, G \in \mathcal{H}_t$.

We let $\P$ denote the Wiener measure on the space $\Omega$, which is
the unique measure so that the process 
induced by the evaluation map on $\Omega$ is a two-sided Brownian motion on $\RR^{|\mathcal{Z}|}$.  The associated expectation will be denoted by $\E$.  Here the $\sigma$-algebra is provided by the Borel subsets of 
$\Omega$.  See, for example, \cite{RevuzYor2013} for detailed constructions.

We define the \emph{shift map} $\theta_s :\Omega \rightarrow \Omega$ for $s \in \RR$ by
\begin{align}
  \theta_s V(t)= V(t+s)-V(s) \quad \text{ for any } \; V \in \Omega,\, t \in \RR.
  \label{eq:shf:mp:act}
\end{align}
Recall that $\{\theta_s\}_{s \in \RR}$ is a group of measure preserving transformations; 
namely,  $\theta_s\theta_r = \theta_{s+r}$ for any $s,r \in \RR$ and
$\P(\Gamma) = \P(\theta_s(\Gamma))$ for any $s \in \RR$ and $\Gamma \in \mathcal{F}$.

We recall the definition of a (continuous, adapted) cocycle as follows:
\begin{definition}\label{def:co-cyc}
We say that a mapping $\phi\colon  [0, \infty) \times X \times \Omega \rightarrow X$ is a \emph{continuous adapted cocycle} if
\begin{itemize} 
\item[(1)]  $\phi$ is continuous;
\item[(2)]  for every $u \in X$ and $V\in \Omega$, $\phi_0(u, V)=u$;
\item[(3)]  for every $u\in X$, $V\in \Omega$ and $t, s >0$, 
\begin{align}
  \phi_{t+s}(u, V) = \phi_t(\phi_s(u, V), \theta_s V)
  \label{eq:cocyc:prop}
\end{align} 
where we recall that $\theta_s$ is the shift map;
\item[(4)] For any $t > 0$, $u \in X$, $V, \tilde{V} \in \Omega$, 
\begin{align}
  \text{if } V(s) = \tilde{V}(s) \text{ for all } s \in [0,t] \text{ then } \phi_t(u, V) = \phi_t(u, \tilde{V}).
  \label{eq:no:future}
\end{align}
\end{itemize}
\end{definition}

Throughout this section, we let $\phi$ denote an arbitrary fixed
cocycle satisfying (1)-(4).  Note that the level of generality of a continuous
cocycle will be sufficent to establish the irreducibility and ergodicity results 
Sections~\ref{sec:top:irritability}, \ref{sec:er:god}.  In order to prove results on finite-dimensional
projections below in Section~\ref{sec:den:pos} some further, more refined
conditions on $\phi$ will be imposed (see Assumption~\ref{ass:Mal:der}).

\begin{remark} \label{rmk:res:ccy:ex}
Given $t > 0$ and any measurable map $\mathcal{E} : \Omega_t \to \Omega$
such that for any $V \in \Omega_t$
\begin{align*}
  \mathcal{E}(V)(s) = V(s) \quad \text{ for every } s \in [0,t].
\end{align*}
We can define a map $\phi_t^{\mathcal{E}}:  X \times \Omega_t \to X$
according to
\begin{align}
   \phi^{\mathcal{E}}_t(u, V) :=  \phi_t(u, \mathcal{E}(V)),
\end{align}
for any continuous adapted cocycle $\phi$.
In view of assumption (4), it is clear that $\phi^\mathcal{E}_t$ is independent of
$\mathcal{E}$ and continuous on $X$.  In what follows we will abuse notation and 
consider $\phi_t$ as also defining a continuous map from $X \times \Omega_t$
into $X$.  

It is worth emphasizing that assumption (4) implies that, for every fixed $u \in X$
the random process
\begin{align}
  t \mapsto \phi_t(u, W) \quad \text{ is adapted to the filtration generated by } W.
  \label{eq:adptd:coc}
\end{align}
This property will allow us to associate a Markovian framework with $\phi$ in what follows.\footnote{In the language 
of \cite{Arnold2013}, condition (4) in Definition~\ref{def:co-cyc} in implies 
that $\phi$ defines a \emph{Markov random dynamical system}.  Thus, in particular, $\phi$ is in the wider class of cocycles 
for which a corresponding Markovian framework can be defined.}
\end{remark}

\begin{remark}
\label{rem:pdecocycle}
Below in the examples considered in Section~\ref{sec:examples}, we 
have that each  concrete formulation of \eqref{SPDE} 
can be written as a continuous functional of the sample path of the Brownian motion. 
In fact this is a property typically enjoyed by systems with additive noise since we can
write solutions of \eqref{SPDE} as $u(t, u_0, W) = v(t,u_0, W) + \sigma W$ where
$v$ obeys
\begin{align*}
   \partial_t v + L (v + \sigma W)  + N(v + \sigma W) = f, \quad v(0) = u_0,
\end{align*}
We can then in turn \emph{define} the cocycle $\phi$ by $\phi_t(u_0,V)= v(t, u_0, V) + \sigma V$
which make sense for \emph{every} $V=(V_k)_{k\in \mathcal{Z}}\in \Omega$.    
\end{remark}
\begin{remark}
  Of course the notation of a cocycle introduced above can be extended to cover systems defined locally in time.  Just as with the formalism in Section~\ref{sec:sat} we expect that many of the results in this section can be extended to such a local setting.  For the sake of clarity and 
  simplicity, we will refrain from addressing this situation here leaving this case for future work. For applications of 
  degenerate control problems to locally defined finite dimensional stochastic systems, see \cite{HerMat15} in a general context and \cite{Her11, BirHerWehr12} for further specific applications.
\end{remark}

We associate a Markovian framework with $\phi$ as follows.
In view of \eqref{eq:adptd:coc} and the cocycle property (\ref{eq:cocyc:prop})
\begin{align}
  P_t g(u_0) = \E g(\phi_t(u_0,W))\quad \text{ for any } \; g \in \mathcal{M}_b(X),
  \label{eq:mar:sg:cyc}
\end{align}
defines a Feller Markov semigroup.
Here $\mathcal{M}_b(X)$ is the collection of real valued bounded, measurable functions from $X$.   
The associated transition kernel is given by 
\begin{align}
  P_t(u_0, A) = (P_t \one_A)(u_0) \quad \text{for any} \;  
                           t \geq 0, \, u_0 \in X, \,  A\in  \mathcal{B}(X),
  \label{eq:trn:fn:cyc}
\end{align}
where $\mathcal{B}(X)$ denotes the Borel $\sigma$-algebra of subsets of $X$.
Recall that $P_t$ acts dually on probability measures $\mu$ on $X$ via 
\begin{align*}
\mu P_t(A)= \int_X P_t(u, A) \mu(du) = \int_X \E [\one_A (\phi(u, W)) ]\mu(du)
\end{align*}   
for any $A\in \mathcal{B}(X)$. We call a probability measure $\mu $ on $\mathcal{B}(X)$ \emph{invariant} if 
\begin{align*}
\mu P_t = \mu
\end{align*} 
for all $t>0$.\footnote{Recall that the collection of such measures is a convex set with the 
extremal points being the \emph{ergodic} invariant measures, i.e. those measures $\mu$ such that 
if $P_t \one_A = \one_A$ $\mu$-almost everywhere, then $\mu(A) \in \{0, 1\}$.  
Note that any two ergodic invariant measures either coincide or are mutually singular.  
See, e.g., \cite{RB_06} for further details.}

Finally it remains to connect the cocycle formalism with the control
theoretic setting described above in Sections~\ref{sec:overview}~\ref{sec:sat}.  
Observe that to the cocycle $\phi$ we may associate the following
collection of 
continuous (global) semigroups $\mathcal{F}_0$ given by   
\begin{align}
\mathcal{F}_0 =  \{ (t, u) \mapsto \phi_t(u, V_\alpha) : V_\alpha(s) = s \alpha \text{
  for } \alpha \in \R^{|\mathcal{Z}|}, s \geq 0\}.  
\label{eq:in:sgs:cyc}
\end{align}
Observe that $\phi_t(u_0, V_\alpha)$ is a continuous
semigroup in the sense of Definition~\ref{def:cont:sg}
for any fixed $\alpha \in \RR^{|\mathcal{Z}|}$.
\begin{remark}
Following the notation introduced in Section~\ref{sec:overview}, we
recall that $\Phi_t^{\alpha \cdot \sigma}u_0$ formally denotes the solution
of~\eqref{eq:abs:evo} with initial condition $u_0\in X$ and control
$\alpha\in L^2([0, \infty); \R^{|\mathcal{Z}|})$.  In particular, given a Cameron-Martin direction $V \in \mathcal{H}_t$ (cf. \eqref{eq:Cam:Mar:Sp}) we have that 
\begin{align*}
u(t, u_0, V)=\phi_t(u_0, V)= \Phi_t^{\alpha\cdot \sigma}u_0  
\end{align*}  
where $\alpha = \dot V$.   
Notationally the use
of $\phi$ is natural 
in this section as we are now considering sample paths 
from a Brownian forcing which does not have a
traditional time derivative. 
\end{remark}

The notions of controllability given in Definition~\ref{def:acc:cont}
are equivalently formulated in the cocycle formalism as follows:
\begin{definition}\label{def:control:cyc}
Let $\phi$ be a continuous adapted cocycle.  We say that 
\begin{itemize}
\item[(i)] $\phi$ is \emph{approximately controllable} if the
associated collection of continuous semigroups $\mathcal{F}_0$ given by \eqref{eq:in:sgs:cyc} are
approximately controllable on $X$. In other words $\phi$ is approximately 
controllability if, for any $u, v \in X$ and any $\delta, t>0$, there exists a 
piecewise linear function $V \in \Omega_t$ so that 
\begin{align}
   \| \phi_t(u,V)- v\| < \delta.
  \label{eq:ap:cyc:con}
\end{align}
\item[(ii)] Let $\pi: X \to X$ be a projection onto a finite-dimensional subspace of $X$.
If, for any $u,v \in X$ and $t, \delta >0$, there is a piecewise linear
function $V \in \Omega_t$
such that \eqref{eq:ap:cyc:con} holds and additionally
\begin{align}
\label{eqn:exact:proj}
  \pi(\phi_t(u, V)) = \pi(v)
\end{align}
then we say that  $\phi$ is \emph{approximately controllable and exactly controllable on $\pi(X)$}.  If there exists $V\in \Omega_t$ such that \eqref{eqn:exact:proj} holds, then we say that $\phi$ is \emph{exactly controllable} on $\pi(X)$.   
\end{itemize}
\end{definition}

\subsection{Topological Irreducibility} 
\label{sec:top:irritability}
We now show that approximate controllability implies a form 
of topological irreducibility that all points on the phase space
are approximately reached with positive probability.

\begin{lemma}\label{l:control}  Let $\phi$ be a continuous adapted cocycle 
  and $P_t(u_0, A)$ be its associated Markov transition function.
  If $\phi$ is approximately controllable, then 
  $P_t(u,B_\delta(v)) >0$ for all $u,v \in X$ and 
  $\delta>0$ or, in other words, $\text{\emph{supp}}(P_t(u,\ccdot)) = X$. 
  Furthermore, for any compact set $K \subseteq X$ and any $\delta, t>0$, $v \in X$, there 
  exists an $\epsilon_0 =\epsilon_0(K, \delta, t, v) >0$ such that 
  \begin{align}
    \inf_{u \in K} P_t(u,B_\delta(v)) \geq \epsilon_0 > 0. 
    \label{eq:irred:comp}
  \end{align}
\end{lemma}
\begin{remark}
  Generally the first consequence in Lemma~\ref{l:control} is a sufficient 
  form of topological irreducibility to establish unique ergodicity.
  See Corollary~\ref{cor:full:sup} and Corollary~\ref{cor:uniq:criteria} below and also
  \cite{HM06}.   The uniform lower bound over compact sets, \eqref{eq:irred:comp}, is useful in establishing 
  rates of convergence to the stationary distribution. 
\end{remark}

\begin{proof}[Proof of Lemma~\ref{l:control}]
  Let $t, \delta >0$ and $u\in X$.  First observe that approximate
  controllability of $\phi$ implies the existence of a piecewise
  linear $V\in \Omega_t$ such that $\| \phi_t(u, V) - v\| <
  \frac{\delta}{2}.$  By assumption, the mapping
 $\phi_t(u,\ccdot)\colon \Omega_t \rightarrow X$ is continuous.  Hence,
 there exists an $\eps>0$ so that if $\| \tilde{V}-V\|_{\infty, t} <
 \eps$ with $\tilde{V}\in \Omega_t$ then 
  $\|\phi_t(u,\tilde{V})- \phi_t(u,V)\| < \delta/2$. Combining everything we 
  have $\| \tilde V -V\|_{\infty, t} < \eps$ implies 
  \begin{align*}
    \|\phi_t(u,\tilde V)- v\|  \leq \|\phi_t(u,\tilde V)- \phi_t(u,V)\|  +
   \|\phi_t(u,V)-v\|< \delta.   
  \end{align*}
Since, for any $\eps>0$, $\P(  \| W-V\|_{\infty, t} < \eps) >0$, the proof of the first statement now follows.

Turning to the proof of the second statement, fix $K\subseteq X$ compact and let $\delta, t >0$ and $v\in X$.  First observe that $P_t(\,\ccdot\,, B_\delta(v))\colon X\rightarrow [0,1]$ is continuous by dominated convergence and the fact that $\phi$ is a continuous cocycle.  Indeed, note that for any $u\in X$ we have
\begin{align*}
\lim_{w\rightarrow u} P_t(w, B_\delta(v)) &= \lim_{w\rightarrow u} \int \one_{\{\| \phi_t(w, W) - v\| < \delta\}} \, d\P = \int \lim_{w\rightarrow u} \one_{\{\|\phi_t(w, W)-v\| < \delta \}}\, d\P = P_t(u, B_\delta(v)).  
\end{align*} 
Consequently, for every $u\in K$ define $\eps_u>0$ such that the following holds 
\begin{align*}
  \| w - u \| < \eps_u \,\, \text{ implies }\,\, P_t(w,
  B_\delta(v)) \geq \tfrac12  P_t(u,
  B_\delta(v))\,. 
\end{align*}
Since 
$\{ B_{\eps_u}(u) : u \in K\}$ is an open cover of the compact set 
$K$, there exists a finite subcover $\{ B_{\eps_{u_k}}(u_k) : k=1,\dots,m\}$ for some collection $\{u_1, \ldots, u_m\} \subseteq K$. By the first statement proven, $P_t(u_k,B_\delta(v))>0$ for each $k$ and hence 
\begin{align*}
  \inf_{u\in K} P_t(u, B_\delta(v)) \geq \epsilon_0 := \tfrac12 \min_k P_t(u_k,B_\delta(v)) >0,
\end{align*}
which is the desired result.
\end{proof}

We have the following simple but important consequence:
\begin{corollary}
\label{cor:full:sup}
 If $\phi$ is approximately controllable, then $\mu(B)>0$ for any
 invariant measure $\mu$ and any open set $B\subseteq X$.  In other words,
 $\text{\emph{supp}}(\mu) = X$ for every probability measure $\mu$ which is invariant under $P_t$.
\end{corollary}
\begin{proof}
For $k\in \N$, define the subsets 
\begin{align*}
  A_k = \{ u\in X\,: \, P_t(u, B) \geq k^{-1}\} 
\end{align*}
and note that since $\mu$ is invariant  
 \begin{align*}
   \mu(B)= \int_X P_t(u,B) \mu(du) \geq \tfrac 1k \mu(A_k)
 \end{align*}
for every $k\in \N$.  Since, according to Lemma~\ref{l:control},  $P_t(u, B) >0$ for every $u\in X$ we have that $\mu(A_k) \uparrow \mu(X)=1$ as $k\rightarrow \infty$.  In particular, $\tfrac 1k \mu(A_k)>0$ for some $k\in \N$, thus finishing the proof.    
\end{proof}

\subsection{Unique ergodicity}
\label{sec:er:god}

We turn next to examine some consequences of approximate controlability for unique ergodicity in systems
like (\ref{SPDE}).

In \cite{HM06}, the concept of an \textit{asymptotically strong 
  Feller} Markov process was introduced. The asymptotic strong Feller property is a generalization 
of the well-known strong Feller property. 
It furnishes the semigroup $P_t$ with just enough 
smoothing to be able to conclude unique ergodicity when all its invariant measures have a point of 
common support (see Theorem~\ref{thm:disjoinSuport} below).  This is useful especially for classes 
of stochastic partial differential equations where 
the strong Feller property appears to be untenable to prove or may not hold.  

Following \cite{HM06,HM09}, we will work mainly in the context of the following lemma which establishes 
the asymptotically strong Feller property by means of an estimate controlling the derivative of the Markov 
semigroup with respect to the initial  condition.  For further reference, see Section~1.1 from~\cite{HM06}. 

  \begin{proposition}[Proposition 3.12 from \cite{HM06}, Proposition 
    1.1 from~\cite{HM09}] 
    A Markov semigroup $\{P_t\}_{t \geq 0}$ on $X$ is asymptotically strong Feller at a point 
    $u \in X$ if there exist an open neighborhood $U$ of $u$ and positive sequences $\{t_n\}_{n \geq 1}$ and $\{\delta_n\}_{n\geq1}$ with $\{t_n\}_{n \geq 1}$ non-decreasing and $\{\delta_n\}$ converging to zero such that 
    \begin{align}
      \label{eq:ASF}
      \sup_{\| \xi\| = 1}|D P_{t_n} f(v) \xi | \leq C (|f |_\infty + \delta_n \|D f \|_\infty)
    \end{align}
    for all $n\in \N$, $v\in U$ and all test functions $f \in C^1(X)$.  Here 
    \begin{align*}
      | f |_\infty = \sup_{x \in X} |f(x)|, \quad \| D f\|_\infty  = \sup_{x \in X, \| \xi\| = 1} | Df(x)\xi|
    \end{align*}
    and $C > 0$ is a fixed constant.
  \end{proposition}

In the current context, our interest in the asymptotic strong Feller property is 
the following result.  See also Theorem 2.1 and Corollary 2.2 from \cite{HM09}. 

\begin{theorem}[Theorem 3.16 from \cite{HM06}]\label{thm:disjoinSuport}
  Suppose $P_t$ is asymptotically strong Feller at a point $u \in X$. If $P_t$ 
  admits two distinct ergodic invariant measures $\mu$ and 
  $\nu$, then $u \not \in \mathrm{supp}(\mu) \cap \mathrm{supp}(\nu)$. 
\end{theorem}

Combining this result with Lemma~\ref{l:control} gives the following. 
\begin{corollary}
\label{cor:uniq:criteria}
If  $\phi$ is approximately controllable and the associated Markov semigroup $P_t$
  asymptotically strong Feller at some point $u\in X$, then  
$P_t$ has at most one invariant measure. 
\end{corollary}
\begin{proof}Lemma~\ref{l:control} implies that the support of any 
  invariant measure is the whole space. Because the semigroup is 
  asymptotically strong Feller, Theorem~\ref{thm:disjoinSuport}
  implies that any distinct ergodic measure must have disjoint 
  supports. Since the support is the whole space, there can be at most one 
  ergodic invariant measure. Since any invariant measure can be 
  decomposed into ergodic invariant measures, there can be at most one 
  invariant probability measure. 
\end{proof}

Establishing the asymptotically strong Feller property, in particular the 
estimate \eqref{eq:ASF}, is a story in and of itself.  In fact, one of the 
central assumptions often employed to assure this property is a formal `H\"ormanader like' bracket 
condition very reminiscent of, if not exactly the same as, the condition used to establish approximate controllability.
A further discussion on the formal relationship between H\"ormander's condition and sufficient conditions for approximate 
controllability is given at the end of Section~\ref{sec:overview}.  
In regards to establishing the asymptotically strong Feller property from H\"ormander's condition,
we refer the reader to \cite{HM09} for a general framework and assumptions. In particular, see Meta-Theorem 1.5 and 
Theorem 8.1 in this article \cite{HM09}.  The methods introduced in \cite{HM06, HM09} have also played 
a central role in a number of other recent works~\cite{FoldesGlattHoltzRichardsThomann2013, ConstantinGlattHoltzVicol2013, 
FoldesFriedlanderGlattHoltzRichards2016, FriedlanderGlattHoltzVicol2014}.

\subsection{Strict positivity of the density on finite-dimensional
  projections}
\label{sec:den:pos}

 We next show how the stronger form of controllability outlined in
 equation~\eqref{eq:exact:cont} can be used to show that,
for any fixed $u\in X$ and $t>0$, the random variable $\pi (\phi_t(u, W))$ 
has a strictly positive density with respect to Lebesgue measure on $\pi(X)$.
See Theorem~\ref{prop:exod} and Theorem~\ref{thm:pos} below.
Here $\pi$ is any projection onto a finite-dimensional subspace of
$X$. Throughout this subsection, we take $m$ to be the dimension of
$\pi(X)$ so that $\pi(X) \cong \RR^m$ and let $\{e_j\}_{j =1}^m$
 to be an orthonormal basis for $\pi(X)$.

For the results in this section we impose some further properties
on the continuous adapted cocycle $\phi$.  
 These assumptions essentially allow us to work in the setting of Malliavin calculus.  
While the proofs in this section make significant use
of methods from the Malliavin calculus, cf. \cite{BouleauHirsch_91,
  NualartSF, Nualart2ndEd}, our presentation is essentially self-contained. 

Our additional standing assumptions on $\phi$ are as follows:
\begin{assump}\label{ass:Mal:der}
\mbox{}
\begin{itemize}
\item[(i)] For every $u \in X$ and $t>0$, the map
  $\phi_t(u,\,\ccdot\,)\colon \Omega_t\rightarrow X$ is Frech\'et
  differentiable in the Cameron-Martin subspace $\mathcal{H}_t$ of
  $\Omega_t$ where $\mathcal{H}_t$ is defined as in \eqref{eq:Cam:Mar:Sp}.
  The derivative with respect to the `noise variable' will be denoted by
  $D_w$, respectively.\footnote{The random variable
    $D_w\phi_t(u,W)H$ coincides with the Malliavian derivative of
    $\phi_t(u, W)$ in the direction $H$.   See Section~1.2.1 of \cite{Nualart2ndEd}.}
\item[(ii)] For any fixed $V\in \Omega_t$, the map $\phi_t(\ccdot,V):X\rightarrow X$ is Frech\'et
  differentiable in $X$.  This derivative with respect to the `initial condition'
  will be denoted by $D_u$ and we suppose
  \begin{align}
    D_u\phi: [0,\infty) \times X \times \Omega \times X \to X  \text{ is continuous}.
    \label{eq:IC:cont}
  \end{align}
\item[(iii)] For all $t>0$, $v\in X$ and $V\in \Omega_t$, the
  linear map $D_u \phi_t(v, V)$ is \emph{non-degenerate}; i.e., 
  \begin{align*}
    D_u \phi_t(v, V) \xi \neq 0 \,\, \text{ whenever } \,\, \xi\in X\setminus \{ 0\}.  
  \end{align*}
\item[(iv)] For every $u\in X$ and $V\in \Omega_t$ the following integral representation
\begin{align}
   D_w\phi_t(u,V)H=\sum_{k\in \mathcal{Z}}\int_0^t J_{s,t}(u,V)\sigma_k \,
  \dot H_k(s) \, ds 
  \label{eq:du:ham:form}
\end{align}
holds for any $H=(H_k)_{k\in \mathcal{Z}} \in \mathcal{H}_t$ where $J_{s,t}$ is the \emph{Jacobi flow} and is defined by 
\begin{align}
\label{eqn:JacobiFlow}
  J_{s,t}(u,V) \xi = (D_u\phi_{t-s})(\phi_s(u,V), \theta_{s}V)\xi \quad \text{ for }\quad  \xi \in X.
\end{align} 
\item[(v)] For any $u \in X$, $s\leq t$ and $V\in \Omega_t$ the adjoint $J_{s,t}^*(u, V)$ of the linear map $J_{s,t}(u, V)$ is also non-degenerate. 
\end{itemize}
\end{assump}
\begin{remark}
We emphasize that our assumptions are not particularly restrictive for the type of 
SPDEs in which we are interested.  Below in Section~\ref{sec:examples} we show how (i) -- (v) follow from routine a priori estimates. 
See also Remark~\ref{rmk:du:ham} concerning the structural 
assumptions \eqref{eq:du:ham:form} and \eqref{eqn:JacobiFlow}.
\end{remark}
\begin{remark}
One can show that the Jacobi flow $J_{s,t}$ 
satisfies
\begin{align}
\label{eqn:Jacobico}
  J_{s,t}(u, V) \xi = J_{r,t}(u, V) J_{s, r}(u, V)\xi 
\end{align}  
for $s\leq r\leq t$, $u, \xi \in X$, $V\in \Omega_t$.  
We will use this group property extensively below.   Note also that
\eqref{eqn:JacobiFlow}, \eqref{eq:IC:cont} with \eqref{eq:du:ham:form}
implies that, for any $t > 0$
\begin{align}
  D_w\phi:  [0,t] \times X \times \Omega_t \times \mathcal{H}_t \rightarrow X  \text{ is continuous}.
  \label{eq:noise:cont}
\end{align}
\end{remark}
\begin{remark}
\label{rmk:du:ham}
Returning to the formal setting of the abstract stochastic evolution equation (\ref{SPDE}),  
it is not hard to see that for any $\xi \in X$, $s <t$ and $u_0
\in X$, $\rho := J_{s,t}(u_0, W)\xi$ would be
expected to satisfy the linear system
\begin{align*}
  \partial_t \rho  + L \rho + D N(u) \rho = 0, \quad \rho(s) = \xi,
\end{align*}
on the interval $[s,t]$ where $u = u(\ccdot, u_0, W)$ is the solution of (\ref{SPDE}) corresponding to $u_0$ and the Brownian path $W$.  
Here again $D N$ is the Frechet Derivative of $N$ so that, 
recalling polynomial structure of $N$ given in
(\ref{eq:m:l:form}), we have
\begin{align*}
  D N ( u) \rho = \sum_{k = 2}^M k N_k(\rho, u, \ldots, u).
\end{align*}
On the other hand the Malliavin derivative 
$\bar{\rho} = D_w \phi_t(u_0, W) H$ in the noise direction $H \in \mathcal{H}_t$ would be expected to satisfy
\begin{align*}
  \partial_t \bar{\rho}  + L \bar{\rho} + D N(u) \bar{\rho} 
   = \sum_{k \in \mathcal{Z}} \sigma_k \dot{H}_k  , \quad \rho(0) = 0.
\end{align*}
Thus, our assumption \eqref{eq:du:ham:form} is simply a reflection of the Duhamel formula.   
See, e.g., \cite{HM06} and below in Section~\ref{sec:examples} for further details in a concrete setting.
\end{remark}

A basic object in the Malliavin calculus is the Malliavin covariance
matrix (see, e.g., \cite{Nualart2ndEd}).  The spectral and invertibility properties of this operator can be used to
derive important consequences for the law of the associated random 
variable.  For example, in finite dimensions, such properties are often
used to derive the existence and regularity of 
the probability density function; cf. \cite{BouleauHirsch_91, Nualart2ndEd,
  NualartSF, MattinglyPardoux06}.   The Malliavin matrix and its spectral properties are also central to 
proofs of unique ergodicity in  \cite{HM06, HM09, FoldesGlattHoltzRichardsThomann2013, 
ConstantinGlattHoltzVicol2013, FoldesFriedlanderGlattHoltzRichards2016, FriedlanderGlattHoltzVicol2014}.

In our current cocycle setting we define this operator as follows:
\begin{definition}
\label{def:Mal:Mtrx}
For any $u \in X$ and $ V \in \Omega_t $, let
\begin{align}
   M_t(u,V) := D_w\phi_t(u,V) (D_w\phi_t(u,V))^*
  \label{eq:MM:def}
\end{align}
where $(D_w\phi_t(u,V))^*$ is the adjoint of $D_w\phi_t(u,V)$.  
Equivalently $M_t(u,V)$ is defined by 
\begin{align}
\langle M_t(u,V) \xi,\rho \rangle_{X}= \sum_{k\in \mathcal{Z}} \int_0^t\langle 
  J_{s,t}(u,V)\sigma_k\, ,\xi\rangle_X \langle J_{s,t}(u,V)\sigma_k\, ,\rho \rangle_X ds 
  \label{eq:MM:alt:def}
\end{align}
for any $\rho, \xi \in X$.  
\begin{itemize}
\item[(i)] For any fixed $V\in \Omega_t$, the we refer to the operator
  $M_t(u,V)$ as the \emph{Gramian}
 following the terminology of control theory.  For $t>0$ $u \in X$ and $V\in \Omega_t$, we say that the Gramian $M_t(u,V)$ 
is \emph{non-degenerate for the control $V$} if 
\begin{align}
    \langle M_t(u,V) \xi,\xi\rangle_{X} >0 \quad\text{for all} \quad \xi \in X\setminus\{ 0\}.   
  \label{eq:pt:wise:pd:MM}
\end{align}
\item[(ii)] When $V\in \Omega_t$ is replaced by the random variable $W$, we call  $M_t(u,W)$ the 
\emph{Malliavin Covariance Matrix} of the random variable $\phi_t(u,W)$.  
 We say that the Malliavin Covariance Matrix $M_t(u,W)$ is
 \emph{non-degenerate} if it is non-degenerate for almost every Brownian path, i.e.,
\begin{align}
  \P \Big( \langle M_t(u,W)\xi , \xi \rangle_X >0 \quad \text{for all}
  \quad \xi \in X\setminus\{ 0\} \Big) =1. 
  \label{eq:Mal:as:pos:cond}
\end{align}
\end{itemize}
\end{definition}

\begin{remark}
\label{rem:malliavin:spec}
  The condition \eqref{eq:Mal:as:pos:cond} is a nontrivial property of
  stochastic systems like \eqref{SPDE}.    We may expect such a condition
  to hold when \eqref{SPDE} satisfies some form of the H\"ormander 
  bracket condition.   See, e.g. \cite{Nualart2ndEd}, for the finite-dimensional
  setting.    The infinite dimensional case has been addressed in  
  \cite{MattinglyPardoux06, HM09, FoldesGlattHoltzRichardsThomann2013}.   
  In these works on SPDEs, it is established that the associated Malliavin
  matrix satisfies bounds like
  \begin{align}
    \P \left(  \sup_{\xi \in \mathcal{S}_{\alpha, N}} \frac{\langle M_t(u, W) \xi, \xi \rangle_X}{\|\xi\|_X^2} > \epsilon  \right) \geq 1 - r_{\alpha, N}(\epsilon)
    \label{eq:M:pos:cond}
  \end{align}
  where $r_{\alpha, N}(\epsilon) \to 0$ as $\epsilon \to 0$ for any fixed $\alpha \in (0,1)$ and $N \geq 1$.  Here 
  $\mathcal{S}_{\alpha, M} := \{ \xi \in X : \| P_N \xi\| \geq \alpha \| \xi \|  \}$ where $P_N$ is the projection onto the first $N$
 elements of an orthonormal basis for $X$.  
 Note that, by a simple limiting argument one may infer \eqref{eq:Mal:as:pos:cond} from \eqref{eq:M:pos:cond}.
\end{remark}

\begin{remark}
  \label{rmk:abv:not}
  In what follows $t >0$, $u \in X$ and even $V \in \Omega$ will sometimes be fixed quantities.  We therefore frequently
  adopt the abbreviated notations
  \begin{align}
    M = M_t(u,V), \quad \phi = \phi_t(u, V) \quad D_w \phi = D_w \phi_t(u,V),
    \label{eq:abv:Not}
  \end{align}
  exhibiting the dependence on $V$ etc. only when it is warranted.  We will also take 
  \begin{align}
    M^\pi = M^\pi_t(u,V) = \pi M_t(u,V)\pi
    \label{eq:abv:Not:MPM}
  \end{align}
  for the `projected Grammian (or Malliavin) matrix'.
\end{remark}

With these preliminaries now in place we turn to the first result in this section.
It gives a criterion under which $\pi (\phi_t(u, W))$ possesses a probability density. 
This `absolute continuity' result is an adaptation of Theorem~5.2.2 in 
Bouleau and Hirsch~\cite{BouleauHirsch_91} to the setting of this paper.  
\begin{theorem}\label{prop:exod}
  Let $\phi$ be a continuous adapted cocycle satisfying 
  the conditions imposed in Assumption~\ref{ass:Mal:der}.  Suppose
  $\pi$ is a projection onto  a finite dimensional subspace of $X$ and 
  assume that for some fixed $u\in X$ and $t>0$, the Malliavin matrix $M_t(u, W)$ is non-degenerate,
  in the sense of Definition~\ref{def:Mal:Mtrx}, (ii).
  Then the law of the
  random variable $\pi(\phi_t(u,W))$ is absolutely continuous with respect to
  Lebesgue measure on $\pi(X) \cong \R^m$.
\end{theorem}

\begin{remark}
The proof of Theorem~\ref{prop:exod} relies on Federer's coarea formula to establish the desired absolute 
continuity.   Recall that for $n \geq m$, any Lipshitz continuous $\eta: \RR^n \to \RR^m$ and any
measurable, non-negative $g: \RR^n \to \RR$
\begin{align}
  \int_{\RR^n} g(y) \mathfrak{J}_n(\eta)(y) dy =  \int_{\RR^m} \int_{\eta^{-1}(x)} g(y) d \mathcal{H}^{n-m}(y) dx
  \label{eq:fed:up:w:Federer}
\end{align}
where $\mathcal{H}^{n-m}$ is the $n-m$ dimensional Hausdorff measure on $\RR^n$ and 
\begin{align}
  \mathfrak{J}_n(\eta) :=  \sqrt{\det (\nabla \eta) (\nabla \eta)^*}.
  \label{eq:wiedo:jacob}
\end{align}
Here, $\nabla \eta$ is the Jacobian of $\eta$ and $(\nabla \eta)^*$ is the adjoint of this matrix.\footnote{Recall that, 
by Rademacher's theorem, every Lipshitz continuous function is differentiable almost everywhere.}  See e.g.
\cite{Federer2014} for further details.
\end{remark}

\begin{proof}[Proof of Theorem~\ref{prop:exod}]
Since $u \in X$, $t > 0$ are fixed throughout we adopted the abbreviated notation as in Remark~\ref{rmk:abv:not}
exhibiting the dependence on the Brownian path as needed.
We proceed by showing that 
\begin{align}
  \E \biggl( \psi(\pi(\phi(W))) \sqrt{\det( M^\pi(W) )} \biggr) = 0,
  \label{eq:abs:con:sf}
\end{align}
for any $\psi\colon \R^m \rightarrow [0, \infty)$ with $\psi(x)=0$ 
almost surely on $\R^m$.  Since we have assumed that $M$ is non-degenerate 
it follows that 
\begin{align}
  \sqrt{\det(M^\pi(W))} > 0
  \label{eq:M:pos:con}
\end{align}
up to a set of measure zero.  
Hence \eqref{eq:abs:con:sf} and \eqref{eq:M:pos:con} imply that 
$\psi(\pi(\phi(W))) =0$ almost surely for any such $\psi$.
By now selecting $\psi= \one_B$ where $B\subseteq \R^m$ is 
any Borel set with  Lebesgue measure zero we infer the desired absolute continuity
of $\pi(\phi_t(u,W))$ with respect to Lebesgue measure on $\R^m$, hence the desired result.

In order to apply the coarea formula \eqref{eq:fed:up:w:Federer} to prove \eqref{eq:abs:con:sf} 
we make use of the Girsanov theorem with a suitable truncation of the projected
Malliavin matrix $M^\pi$.   Fix
\begin{align*}
  \{h_\ell\}_{\ell \geq 1} \text{ to be an orthonormal basis of } L^2([0, t]; \R^{|\mathcal{Z}|}) 
\end{align*}
so that  the elements $H_\ell(\, \cdot \,) = \int_0^\cdot
h_\ell(s) \, ds$ form an orthonormal basis in $\mathcal{H}_t$.  
Using the assumed continuity of $D_w \phi$ we have that
\begin{align}
   \langle M e_i, e_j \rangle_X 
         &= \sum_{\ell = 1}^\infty 
           \bigl\langle D_w \phi  H_\ell \langle (D_w\phi)^* e_i , H_\ell \rangle_{\mathcal{H}_t} , 
           e_j \bigr\rangle 
         = \sum_{\ell = 1}^\infty 
             \langle D_w \phi  H_\ell , e_j \rangle 
             \langle D_w \phi  H_\ell , e_i \rangle.
             \label{eq:CM:exp}
\end{align}
Truncating in this expansion we define the random matrices $M_n^\pi$ according to
\begin{align}
   ( M_n^\pi)_{ij} = \langle M_n^\pi e_i, e_j \rangle := \sum_{\ell = 1}^n 
             \langle D_w \phi H_\ell , e_j \rangle 
             \langle D_w \phi H_\ell , e_i \rangle,
  \label{eq:tran:mal:mtx}
\end{align}
for $n \geq 1$.  
For any  $y\in \R^n$ we denote
\begin{align*}
T_y W(r) = W(r) + \sum_{\ell=1}^n y_\ell H_{\ell}(r). 
\end{align*}  
for $r \geq 0$.  Since $H_\ell$ are fixed elements in $\mathcal{H}_t$, we observe that
the translation $T_y W$ is an admissible Girsanov shift for any given values of 
$n \geq 1$, $y \in \RR^n$.

Fix any $\rho\in C^\infty(\R^n; (0, \infty))$ satisfing $\int_{\R^n}\rho(y)
\, dy =1$.  Then, for any $\psi \colon \R^m \rightarrow [0, \infty)$
bounded and  measurable, the Girsanov theorem implies
\begin{align}
  \E \psi (\pi(\phi(W))) \sqrt{\det (M_n^\pi(W))} 
  &= \int_{\R^n}\E   \psi (\pi(\phi(W))) \sqrt{\det (M_n^\pi(W))} \rho(y) \, dy \notag\\
  &= \E \int_{\R^n}  \psi (\pi(\phi(T_yW)))  G_H(y,W) \rho(y)  \sqrt{\det (M_n^\pi(T_yW))} \, dy   
    \label{eq:Gir:app:int}
\end{align} 
where $G_H(y, W) > 0$ denotes the Girsanov density associated to the shift
$T_y W$.   

For any $n \geq m$ define $\eta: \RR^n \to \RR^m$ according to
$\eta(y) = \pi(\phi(T_y W))$.  By Assumption~\ref{ass:Mal:der}, it is clear that $\eta$ 
is $\P$-almost surely Lipschitz and we have that 
\begin{align*}
   \mathfrak{J}_n(\eta)(y) = \sqrt{\det (M_n^\pi(T_yW))} 
\end{align*}
$\P$-almost surely, 
where $\mathfrak{J}_n$ is defined according to \eqref{eq:wiedo:jacob}.
Thus, the coarea formula (\ref{eq:fed:up:w:Federer}) implies that, for any $n \geq m$,
\begin{align}
\int_{\R^n}  \psi (\pi(\phi(T_yW))) & G_H(y,W) \rho(y)  \sqrt{\det (M_n^\pi(T_yW))} \, dy   \notag\\
  &= \int_{\R^m} \bigg[ \int_{\eta^{-1}(x)} \psi( \pi( \phi(T_y W))) 
    G_H(y, W) \rho(y) \, d\mathcal{H}^{n-m}(y) \bigg] \, dx \notag\\
  &= \int_{\R^m}\psi(x) \bigg[ \int_{\eta^{-1}(x)} G_H(y, W) \rho(y) \, d\mathcal{H}^{n-m}(y) \bigg] \, dx 
    =0,
   \label{eq:co:area:app}
\end{align}  
almost surely.  

By combining \eqref{eq:Gir:app:int} and \eqref{eq:co:area:app} we obtain that
\begin{align*}
  \E \psi (\pi(\phi(W))) \sqrt{\det (M_n^\pi(W))} = 0.
\end{align*}
In view of (\ref{eq:CM:exp}) and (\ref{eq:tran:mal:mtx}), $M_n^\pi(W) \to M^\pi(W)$ almost surely.  We therefore infer 
\eqref{eq:abs:con:sf} from Fatou's lemma, completing the proof.
\end{proof}

\begin{remark}
Perusing this proof, it is notable that Theorem~\ref{prop:exod} still holds under the weaker 
condition that only $M^\pi = \pi M_t(u,W) \pi$ is non-degenerate.
\end{remark}

We next state and prove the final result of this section which gives a sufficient
condition under which, for fixed $u \in X$ and $t>0$, the density of
$\pi(\phi_t(u, W))$ is strictly positive.    We refer to Nualart~\cite[Theorem 4.2.2]{NualartSF} and 
 \cite[Theorem 8.1]{MattinglyPardoux06} for previous related results.

\begin{theorem}\label{thm:pos} 
Let $\phi$ be a continuous adapted cocycle satisfying the hypotheses of Assumption~\ref{ass:Mal:der}. 
In addition we suppose that 
  \begin{enumerate}
    \item $M_t(u,V)$ is non-degenerate from some fixed $u\in X$, $t>0$ and a fixed sample path $V\in \Omega_t$;
      cf. Definition~\ref{def:Mal:Mtrx}, (i).
    \item $\phi$ is exactly controllable on $\pi(X)$ as in Definition~\ref{def:control:cyc}.    
   \end{enumerate}
   If, for some $s> 0$, the law of $\pi (\phi_{t+s}(u, W))$ has a density $p_{t+s}(\ccdot)$ 
   with respect to Lebesgue measure on $\R^m$, 
   then $p_{t+s}(x) >0$ for Lebesgue almost every $x\in \R^m$. 
\end{theorem}

As in the proof Theorem~\ref{prop:exod}, Theorem~\ref{thm:pos} is established 
with aide of a carefully chosen change of variables
using the Girsanov Theorem.  Two additional ingredients are needed for the proof: 
one concerns invertibility properties of $M$ while the second is a quantitative 
version of the inverse function theorem.  We provide some further
intuition, beyond the given proof, as to why the invertibility of $M$ and
the exact controlability of $\phi$ on $\pi(X)$ implies that the density
is strictly positive in Remark~\ref{rmk:submersion} following the proof.

Regarding the Grammian Matrix $M_t(u,V)$ we make the following observation.  
It shows that once this matrix is non-degenerate at a time $t$ it remains non-degenerate for all later times $t + s$.
\begin{lemma}\label{l:BetterAndBetter}
  Let $u\in X$ and $t>0$.  Suppose that $V\in \Omega_t$ is such that $M_t(u,V)$ 
  is non-degenerate, cf. (\ref{eq:pt:wise:pd:MM}).  
  Then, for any $s>0$ and any $V_{e}\in \Omega_{t+s}$ with $ V_{e}(r)= V(r)$ for every
  $r\in [0,t]$, the operator $M_{t+s}(u,V_{e})$ is also non-degenerate.  
\end{lemma}
\begin{proof}[Proof of Lemma~\ref{l:BetterAndBetter}]
  First let $u\in X$, $t>0$ and $V\in \Omega_t$ be as in the statement
  of the lemma, and suppose that 
  $V_{e}\in \Omega_{t+s}$ satisfies $ V_{e}(r)= V(r)$ for $r\in [0,t]$.  We first observe that, cf. \eqref{eq:MM:alt:def},
  \begin{align*}
   \langle M_{t+s}(u,V_{e}) \xi,\xi\rangle_{X} &= \sum_{k\in \mathcal{Z}} \int_0^t\langle 
  J_{r,t+s}(u,V_{e})\sigma_k\, ,\xi\rangle_X^2  dr + \sum_{k\in \mathcal{Z}} \int_t^{t+s}\langle 
  J_{r,t+s}(u,V_{e})\sigma_k\, ,\xi\rangle_X^2  dr \,. \\
    &\geq \sum_{k\in \mathcal{Z}} \int_0^t\langle 
  J_{r,t+s}(u,V_{e})\sigma_k\, ,\xi\rangle_X^2  dr,
  \end{align*}
which holds for any $\xi \in X$.  Let $J_{r,t+s}^*$ be the adjoint of $J_{r,t+s}(u, V_{e})$ in $X$. 
Using the group property (\ref{eqn:Jacobico}) and \eqref{eq:no:future} we have
\begin{align*}
  \sum_{k\in \mathcal{Z}} \int_0^t\langle 
  J_{r,t}(u,V)\sigma_k\, ,J_{t,t+s}^*\xi\rangle_X^2  dr =  \langle M_t(u,V) J_{t,t+s}^*\xi,J_{t,t+s}^*\xi\rangle_{X} .   
\end{align*}
From Assumption~\ref{ass:Mal:der} we have that $J_{t,t+s}^*\xi \not= 0$ whenever $\xi$.  
Thus, combining these observations, we infer the desired non-degeneracy of $M_{t+s}(u,V_{e})$, 
completing the proof.
\end{proof}

The following quantitative invertibility criteria for $C^1$ functions
may be established in a very similar fashion to the standard 
Inverse Function Theorem.  Also note that our statement here is a variation of
Lemma~4.2.1 of~\cite{NualartSF}.
\begin{lemma}
  \label{lem:quant:IFT}
   Suppose that $\mathcal{G} \subseteq C^{1}(\RR^m)$ is a collection such that
   \begin{itemize}
     \item[(i)] For every $g \in \mathcal{G}$, $g(0) = 0$.
     \item[(ii)] There is an $0< A < \infty$ such that 
       \begin{align}
         \|\nabla g(0)^{-1}\| \leq A \quad 
         \text{ for every } g \in \mathcal{G}.
         \label{eq:inv:bnd:IFT}
         \end{align}
       \item[(iii)] For some $\bar{\gamma} > 0$
         \begin{align}
             \sup_{x \in B_{\bar{\gamma}}(0)} A \| \nabla g(x)  - \nabla g(0) \| < \frac{1}{2}
              \quad \text{ for every } g \in \mathcal{G}.
           \label{eq:growth:bnd:IFT}
           \end{align}
   \end{itemize}
   Then, for any $\gamma \in (0, \bar{\gamma}]$, there is a $\kappa = \kappa(\gamma) > 0$
   such that, for every $g$, $U_g = g(B_\gamma(0))$ is open set with $g$ diffeomorphic between $B_\gamma(0)$
   and $U_g$ with
    \begin{align}
      B_\kappa(0)  \subseteq g(B_\gamma(0)) = U_g, \quad \text{ for every } g \in \mathcal{G}.
      \label{eq:uni:ngh:R}
    \end{align}
\end{lemma}

\begin{proof}[Proof of Theorem~\ref{thm:pos}]  Let $u\in X$,  $t>0$, $V\in \Omega_t$ be as in the 
  statement of the result and let $s> 0$ be such that $\pi(\phi_{t+s}(u, W))$ has density $p_{t+s}(\ccdot)$ 
  with respect to Lebesgue measure on $\R^m$.  We proceed by showing that, for any $x\in \R^m$, 
  that there exists a continuous function $h^x: \RR^m \to \RR$ such that 
  \begin{align}
    h^x(x) > 0 \quad  \text{ and } \quad  \int_{B_\epsilon(x)} p_{t +s}(y) dy \geq \int_{B_\epsilon(x)} h^x(y) dy \quad \text{ for every } \epsilon > 0.
    \label{eq:suff:cond:lbf}
  \end{align}
  With such an $h^x$ the desired result, that $p_{t+s}(x)>0$ for almost every $x$, then follows from the  Lebesgue differentiation theorem. 

  To establish \eqref{eq:suff:cond:lbf} for a suitable $h^x$ we begin by building $V_x \in \Omega_{t+s}$ with
  \begin{align}
    \pi(\phi_{t+s}(u, V_x))=x \,\,\text{ and }\,\,
    M_{t+s}(u, V_x) \text{ is non-degenerate},
    \label{eq:gd:cont}
  \end{align}
  for any $x \in \RR^m$.
  By assumption
  $M_{t}(u,V)$ is non-degenerate.     On the other hand, since the $\phi$ is 
  exactly controllable on $\pi(X)$, there exists a piecewise linear $\hat{V} \in
  \Omega_s$  so that 
  \begin{align}
    \pi(\phi_s( \phi_t(u,V),\hat V))=x. 
    \label{eq:ex:cont:comp}
  \end{align}
  We now define $V_x \in \Omega_{s + t}$ according to
  \begin{align}
    V_x (r)=
    \begin{cases}
      V(r)  &\text{ for } r \in [0,t]   \\
      \hat V(r) + V(t)  &\text{ for } r \in [t,t + s]   
    \end{cases}
     \label{eq:g:path:ext}
  \end{align}
  With \eqref{eq:ex:cont:comp} and the cocycle property
  \eqref{eq:cocyc:prop} we infer the first condition in \eqref{eq:gd:cont}.
  By Lemma~\ref{l:BetterAndBetter}, since $M_{t}(u,V)$ is
  non-degenerate, we conclude that $M_{t+s}(u,V_x)$ is itself
  non-degenerate, yielding the second condition in \eqref{eq:gd:cont}.

  With this $V_x$ in hand we construct $h^x$ in terms of a suitable Girsanov density 
  and a small perturbations around $(\pi M_{t+s}(u, V_x) \pi)^{-1}$.
  For any $y \in \R^{m}$ we take 
  \begin{align} 
    T_y^x W(\tau) 
    = W(\tau) + 
       \sum_{\ell=1}^m y_\ell \int_0^{\tau} \langle J_{r,t+s}(u, V_x) \sigma_k,e_\ell\rangle_X \, dr \quad
    \text{ for any } \tau \in [0, t+s], 
    \label{eq:spec:sft}
  \end{align}
  where we recall that the elements $e_i$ are an orthonormal basis for $\pi(X)$.  In particular we have that
  \begin{align*}
    \int_0^{\ccdot} \langle J_{r,t+s}(u, V_x) \sigma_k,e_\ell\rangle_X \, dr  \in \mathcal{H}_{t +s},
  \end{align*}
  so that $T_y^x  W$ is an admissible Girsanov shift for any $y \in \RR^m$.
  
  Fix a $\rho \in C(\R^m; (0, \infty))$ 
  satisfing $\int_{\R^m }\rho(y) \, dy =1$.  From the Girsanov theorem we infer that, 
  for any measurable $\psi: \RR^m \to [0,\infty)$,
  \begin{align}
     \int \psi(y) p_{t+s}(y) dy  &= \int_{\R^m} \E \psi(\pi(\phi_{t+s}(W))) \rho(y) \, dy 
                  \notag\\
                  &= \E \int_{\R^m} \psi(\pi(\phi_{t+s}(T_y^x W))) G^x(y, W) \rho(y) \, dy \notag\\
                  &= \E \int_{\R^m} \psi(g^x(y, W) + \pi(\phi_{t+s}(W))) G^x(y, W) \rho(y) \, dy 
                    \label{eq:g:shift:V:DIR}
  \end{align}    
  where $G^x(y, W)$ is the Girsanov density associated to the shift $T_y W$ and 
  \begin{align}
    g^x(y, \tilde{V}) = \pi(\phi_{t+s}(T_y\tilde{V})) -\pi(\phi_{t+s}(\tilde{V}))
    \label{eq:fd:cg:of:var}
  \end{align}
  for any $\tilde{V} \in \Omega_{t +s}$.   We therefore
  obtain an expression for an $h^x$ from \eqref{eq:g:shift:V:DIR}
  with the desired proprieties in (\ref{eq:suff:cond:lbf}) 
  by showing that $g^x$ is invertible and changing variables.

  In view of Assumption~\ref{ass:Mal:der},
  $g^x(y,\tilde{V})$ is differentiable in $y$ for any $\tilde{V} \in \Omega_{t+s}$ and 
  \begin{align*}
    \nabla g^x(y,\tilde{V}) =  \mathfrak{T}(T_y^x\tilde{V}) 
  \end{align*}
  where $\mathfrak{T}: \Omega_{t +s} \to \RR^{m \times m}$ is the continuous map given by
  \begin{align*}
    (\mathfrak{T}(\tilde{V}))_{ij} := \sum_{k \in \mathcal{Z}} 
    \int_0^{t +s} \langle J_{r, t+s}(u,\tilde{V}) \sigma_k, e_i\rangle
                \langle J_{r, t+s}(u,V_x) \sigma_k, e_j\rangle dr.
  \end{align*}
  In particular we have that $\nabla g^x(0, V_x) := M^\pi_{t+s}(u, V_x)$.  Thus, the non-degeneracy of $M_{t+s}(u,V_x)$, (\ref{eq:gd:cont}), 
  and the continity of $\mathfrak{T}$ implies that,
  there exists a $\delta > 0$, $\alpha > 0$, such that,
  \begin{align}
    |\det \nabla g^x(y, \tilde{V})| \geq \alpha \quad 
    \|\nabla g^{x}(0, \tilde{V})^{-1}\| \leq \frac{1}{\alpha} \quad 
     \| \nabla g^{x}(0, \tilde{V}) - \nabla g^{x}(y, \tilde{V})\| < \frac{\alpha}{2},
    \label{eq:non:deg:cond}
  \end{align}
  whenever $| y | + \| \tilde{V} - V_x\|_{\infty, t + s}  < 2\delta$.

  Observe that, choosing $\delta > 0$ so that \eqref{eq:non:deg:cond}, we have
  \begin{align*}
    \mathcal{G} := \{ g^x( \ccdot, \tilde{V}) : \| \tilde{V} - V_x\|_{\infty, t +s} < \delta \} \subseteq C^{1}(\RR^m)
  \end{align*}
  satisfies the conditions of Lemma~\ref{lem:quant:IFT}.  Picking $\gamma = \min\{\delta, \bar{\gamma}\}$ we invoke the lemma and obtain the  $\kappa > 0$ so that (\ref{eq:uni:ngh:R}) holds.   For $\tilde{V}$
  we let $f^x( \ccdot , \tilde{V})$ be the corresponding inverse of $g^x(\ccdot, \tilde{V})$ 
  mapping $U_{\tilde{V}} := g^x(B_\gamma(0), \tilde{V})$ to $B_\gamma(0)$.
  Continuing  from \eqref{eq:g:shift:V:DIR} and denoting $\mathscr{H}^x_{\delta}(W) = \one_{\{ \|W-V^x\|_{\infty, t+s} < \delta\}}$
  we have
    \begin{align*}
      &\int \psi(y) p_{t+s}(y) dy\\
       &\geq \E \mathscr{H}^x_{\delta}(W) \int_{B_\gamma(0)} \psi(g^x(y,W) + \pi(\phi_{t+s}(W))) G^x(y, W) \rho(y) \, dy\\
       &=  \E \mathscr{H}^x_{\delta}(W) \int_{U_W}  \! \!  \! \! 
                        \psi(z+ \pi(\phi_{t+s}(W))) G^x(f^x(z, W), W) \rho(f^x(z,W)) | \det \nabla f^x(z,W)|\, dz\\
       &\geq  \E \mathscr{H}^x_{\delta}(W) \int_{B_\kappa(0)}
                        \frac{\psi(z+ \pi(\phi_{t+s}(W))) G^x(f^x(z, W), W) \rho(f^x(z,W))}{| \det \nabla g^x(f(z,W),W)|}\, dz\\
       &= \int_{\R^m} \psi(z)  h^x(z)dz
  \end{align*} 
  where 
  \begin{align*}
    h^x(z)\! :=  \E \left[\frac{\one_{\{ \|W-V^x\|_{\infty, t+s} < \delta \}} \one_{\{ |X(z, W)| < \kappa \}}G^x(f^x(X(z),W), W) \rho(f^x(X(z),W))}
                               {| \det \nabla g^x(f^x(X(z),W),W)|}\right]
  \end{align*}
  and $X(z,W)=z -\pi(\phi_{t+s}(W))$.   In view of \eqref{eq:non:deg:cond}, standard properties of Brownian motion
  and noting that $\rho$ and $G^x$ are both strictly positive we therefore conclude that $h^x(x)$ is strictly positive. Hence
  $h^x$ satisfies (\ref{eq:suff:cond:lbf}), completing the proof of the result.
 \end{proof}

\begin{remark}
  Note that, in contrast to Theorem~\ref{prop:exod} which requires $M$ to be
  non-degenerate for almost every Brownian path, Theorem~\ref{thm:pos}
  simply requires that the $M$ be non-degenerate for a single $V \in
  \Omega_t$.  In practice, however, we will prove an estimate like
  \eqref{eq:M:pos:cond} and use the implication
  \eqref{eq:Mal:as:pos:cond} to select one path from a set of full
  $\P$ measure on $\Omega_t$ to satisfy the conditions in
  Theorem~\ref{thm:pos}.  
\end{remark}

\begin{remark}
\label{rmk:submersion}
 While the exact controllability of $\pi(\phi)$ produces, for each $x \in \RR^m$, at least one noise path $V_x \in \Omega_{t +s}$
 such that $x = \pi(\phi_{t + s}(u, V_x))$ the invertibility of $\pi M_{t+s}(u, V_x) \pi$ shows that the
 tangent space around this point $x$ produced by Cameron-Martin perturbations $\mathcal{H}_{t+s}$ is of full rank.    
 Indeed, to show that $D_w\pi(\phi_{t + s}(u,V_x))$ is of full rank in $\mathcal{H}_{t+s}$ we would like to 
 show that, for any unit length $\xi \in \RR^m$ there is a
 corresponding perturbation $H_\xi \in \mathcal{H}_{t +s}$ of $V_x$,
 such that
\begin{align}
  \pi(\phi_{t+s}(u,V_x + \epsilon H_\xi)) \approx x + \epsilon D_w \pi (\phi(u,V_x))H_\xi = x + \epsilon \xi,
  \label{eq:submersion}
\end{align}
 for $0< \epsilon \ll 1$.
 We may produce such an $H_\xi$ by solving the following least squares
 problem.   Assuming that $H_\xi$ has the form  $H_\xi = (D_w \phi_{t+s}(u, V_x))^* \pi \eta$
 then, cf.  (\ref{eq:MM:def}), we have that $\xi = D_w \pi( \phi_{t+s}) H_\xi$ 
 when $\eta = M_{t-s}^\pi(t, V_x)^{-1} \xi$. Hence 
 \begin{align*}
 H_\xi = (D_w \phi_{t+s}(u, V_x))^* \pi  M_{t-s}^\pi(t, V_x)^{-1} \xi.
 \end{align*}
 yields \eqref{eq:submersion}.
\end{remark}


\section{Examples}
\label{sec:examples}

The goal of this section is to see how the results of Sections~\ref{sec:sat} and \ref{sec:app:SPDEs} can be applied 
to study specific degenerately forced problems.  
In particular we will consider the following examples illustrating different aspects of the theory introduced previously:
\begin{itemize}
\item[(i)] The Reaction-Diffusion equation.
\item[(ii)] The 2D Navier-Stokes equations.
\item[(iii)] The 2D Boussinesq equations. 
\item[(iv)] The 3D Euler equations.
\end{itemize}
In each example we will also see how the control results in Section~\ref{sec:sat} can be used in conjunction with 
the formalism introduced in Section~\ref{sec:app:SPDEs} 
to infer properties of the support of the law of the random variable 
solving the associated stochastic partial differential equation.   In the examples
(ii) and (iii), we also deduce unique ergodicity of
invariant measures in the presence of inhomogenous forcing terms.  To the best of the authors' knowledge, the results concerning (i) and (iii) are new while the results on (ii) and (iii) in similar functional settings have been obtained and discussed previously in, respectively,~\cite{AgrachevSarychev2005,AgrachevSarychev2006, HM06, HM09,MattinglyPardoux06, AgrachevKuksinSarychevShirikyan2007} and~\cite{Shirikyan2008, Nersisyan2010, Nersesyan2015}.  Nevertheless, the examples (ii) and (iii) illustrate the applicability of the methods of Section~\ref{sec:sat} and Section~\ref{sec:app:SPDEs}.   

\subsection{Reaction-Diffusion equation}
\label{eq:rd:ex}

For a first example, we consider the following reaction-diffusion equation  
\begin{align}
  \partial_t u- \kappa \partial_{xx} u = f(u) + \sigma \cdot \partial_t V 
  \label{eq:RD:cont}
\end{align}
where $\kappa >0$ is the diffusivity constant.  The (scalar) equation~\eqref{eq:RD:cont} is posed on the interval $[0,\pi]$ and is supplemented 
with the Dirichlet boundary conditions 
\begin{align}
   u(t, 0) = 0 = u(t, \pi) \,\, \text{ for all } \,\, t \geq 0.  
   \label{eq:RD:bcs}
\end{align}
Here, the nonlinearity $f$ is assumed to be an odd polynomial of the form
\begin{align}
  f(v) = \sum_{k = 0}^{2n -1} b_k v^k, \,\, n\geq 2.
  \label{eq:RD:nonlin}
\end{align}
We suppose that the leading-order coefficient $b_{2n-1}$ in the nonlinearity is such that
\begin{align*}
b_{2n-1} \leq - \nu
\end{align*}
for some $\nu > 0$.   In particular this implies that, for some constant $K$ depending only on $f$
\begin{align}
   v f(v) \leq K - \frac{\nu}{2} v^{2n} \quad \text{ and } \quad \sup_{v \in \RR} f'(v) \leq K
   \label{eq:RD:nl:prop}
\end{align}
for all $v\in \R$.  The term $\sigma \cdot V$ takes the form
\begin{align}
  \sigma \cdot \partial_t V = \sum_{k \in \mathcal{Z}} \sigma_k \partial_t V_k
\end{align}
The controlled directions $\mathcal{Z} \subseteq \Z_{\geq 1} = \{1,2,3, \ldots\}$ are a finite subset with
$\sigma_k(x) =\sin(kx)$ and $V= (V_k)_{k\in \mathcal{Z}}$ is a fixed element in 
\begin{align*}
\Omega = \{V:(-\infty, \infty) \rightarrow \R^{\mathcal{|Z|}} \text{ continuous with }V(0)=0 \},
\end{align*}
following the notation introduced in Section~\ref{sec:app:SPDEs}.  

We proceed with our analysis of \eqref{eq:RD:cont}--\eqref{eq:RD:bcs} by recalling the cocycle setting in Proposition~\ref{prop:contco:rd} followed by the main control results in Theorem~\eqref{thm:RD:main}.  The main scaling estimates are encapsulated in Lemmas~\ref{lem:sc:lem:1:RD} and \ref{lem:sc:lem:2:RD} below.

\subsubsection{Mathematical Setting, Cocycle Formulation} 
Regarding the mathematical formulation of
\eqref{eq:RD:cont}--(\ref{eq:RD:bcs}) we consider \emph{weak
  solutions}.   Smoother classes of solutions of
\eqref{eq:RD:cont}--(\ref{eq:RD:bcs}) could just as well be considered
but we omit detailed discussion for simplicity and clarity of presentation.\footnote{Note also that \eqref{eq:RD:cont}
is just one example of a wide variety of reaction-diffusion equations which are in principal accessible to the formalism
developed above in Sections~\ref{sec:sat} and \ref{sec:app:SPDEs}.  
See also, Remark~\ref{rmk:its:complicated} below.}
We refer to, e.g., \cite{Smoller2012,Temam2012} for further background on general 
mathematical theory surrounding \eqref{eq:RD:cont} and its variants.

The phase space for the cocycle associated with
\eqref{eq:RD:cont}--~\eqref{eq:RD:bcs} is taken to be the Hilbert space 
$X= L^2= L^2([0,\pi])$ equipped with the standard norm $\| \ccdot \|$ 
and inner product $\langle \ccdot, \ccdot \rangle$. 
 Note that $\{ \sigma_k : k = 1, 2, 3, \ldots\}$ provides an orthogonal basis for $L^2$.
As usual we take $H^1_0 = H^1_0([0,\pi])$ to be all of the elements in $L^2$
whose (weak) derivative is in $L^2$ and which vanishes at $0$ and
$\pi$.  Some of the forthcoming estimates will also involve 
bounds in $L^{2n}= L^{2n}([0,\pi])$ where $2n-1$ is the degree of the
polynomial $f$.

Of course, \eqref{eq:RD:cont} does not make 
sense directly for generic elements $V$ in $\Omega$.
Following Remark~\ref{rem:pdecocycle}, we define the solution $u(t, u_0, V)$ of \eqref{eq:RD:cont} corresponding to the initial condition
 $u(0,u_0, V)=u_0 \in L^2$  by $u(t, u_0, V):= v(t, u_0,V) + \sigma \cdot V$ where $\sigma\cdot V= \sum_{k\in \mathcal{Z}} \sigma_k(x) V_k(t)$ 
and $v=v(t, u_0, V)$ solves the equation
\begin{align}
\label{eq:RD:shift}
\partial_t v - \kappa \partial_{xx} (v + \sigma \cdot V) = f(v+\sigma \cdot V), \,\, v(0)= u_0,
\end{align} 
in the weak sense.

To make this all precise, we have the following well posedness results for \eqref{eq:RD:cont}--(\ref{eq:RD:bcs})
which is compatible with the setting of Section~\ref{sec:sat} and much of Section~\ref{sec:app:SPDEs}.  
The proof of this well posedness result is based on standard a priori estimates which we outline below in Appendix~\ref{sec:apriori:R-D}.      
\begin{proposition}
\label{prop:contco:rd}
Consider \eqref{eq:RD:cont}--\eqref{eq:RD:bcs} with $f$ as in \eqref{eq:RD:nonlin} and $\mathcal{Z}$ finite.  
Then, for any  $V\in \Omega$ and $u_0 \in L^2$, there exists a unique weak solution $u=u(\ccdot , u_0, V)$, namely
\begin{align}
  u \in L^2_{loc}([0,\infty); H^1_0) \cap C([0,\infty), L^2) \cap L^{2n}_{loc}([0,\infty), L^{2n})
\end{align}
with $u(0) = u_0$ and such that $v :=u-\sigma \cdot V$ solves \eqref{eq:RD:shift} in the weak 
sense; that is integrated against smooth, compactly supported test functions. 
Furthermore, the induced mapping $\phi: [0, \infty) \times L^2 \times \Omega \rightarrow L^2$ given 
by $\phi_t(u_0, V) = u(t, u_0, \sigma \cdot V)$ is a continuous adapted cocycle in the sense 
of Definition~\ref{def:co-cyc}.    
\end{proposition}

\begin{remark}\label{rmk:shit:sln:RD}
The notion of solutions to \eqref{eq:RD:cont} given in Proposition~\ref{prop:contco:rd} subsumes two more classical notions which arise as special cases. 
If $V$ belongs to the Cameron-Martin space $\mathcal{H}_t$ as defined in~\eqref{eq:Cam:Mar:Sp}, 
then the more usual sense of weak solution solutions of~\eqref{eq:RD:cont} are well defined and coincide with 
the solutions provided by Proposition~\ref{prop:contco:rd}.  
On the other hand, if we replace $V\in \Omega$ by a standard two-sided $|\mathcal{Z}|$-dimensional Brownian motion $W$, 
then \eqref{eq:RD:cont} may be regarded as a stochastic partial differential equation for which we may obtain solutions in the 
setting of infinite-dimensional stochastic analysis as in, e.g., \cite{ZabczykDaPrato1992}.
Regardless, since the noise is additive, upon replacing $V$ by $W$ in~\eqref{eq:RD:shift} and defining $u(t, u_0, W)=v(t, u_0, W)+ \sigma\cdot W$ 
we obtain the same pathwise solution as the one defined using the stochastic analysis approach. 
See, e.g., \cite{CrauelFlandoli1994,CrauelDebusscheFlandoli1997}.      
\end{remark}

\subsubsection{Statement of the main results for equation~\eqref{eq:RD:cont}} 

 In order to state the main control results for the reaction-diffusion equation~\eqref{eq:RD:cont}, 
 we first outline further assumptions we make on the noise/control
 directions $\sigma= (\sigma_k : {k\in \mathcal{Z}})$.  Let 
 \begin{align}
 f_*(v_1, v_2, \ldots, v_{2n-1})= b_{2n-1} v_1 v_2 \ldots v_{2n-1}
 \end{align}
 denote the multilinear form corresponding to the leading-order term in $f$ and define $L^2$-subspaces $X_m$, $m\geq 0$, by 
 \begin{align}
   Y = X_0 = \text{span} \{ \sigma_k : k \in \mathcal{Z}\}
   \label{eq:dir:cont:RD}
 \end{align}
and 
 \begin{align}
 X_m = \text{span}\{ X_{m-1} \cup \{f_*(h_1, h_2, \ldots, h_{2n-1}) \, : \, h_i \in X_{m-1} \} \}.    
      \label{eq:it:dir:RD}
 \end{align}
We recall as in Definition~\ref{def:hor:cond} that the pair $(f_*, \sigma)$ satisfies \emph{H\"{o}rmander's condition} on 
$L^2$ if $\bigcup_{m\geq 0} X_m$ is dense in $L^2$.  
 
 Our main control result is the following: 
 \begin{theorem}
   \label{thm:RD:main}
   Suppose that we are under the conditions of Proposition~\ref{prop:contco:rd}
   and that $(f_*, \sigma)$ satisfies H\"{o}rmander's condition on $L^2$. Let $\pi:~L^2 \rightarrow L^2$ be any continuous linear projection 
   onto a finite-dimensional subspace $\pi(X)\subseteq L^2$.  Then the associated cocycle $\phi$ is approximately controllable on $L^2$ 
   and exactly controllable on $\pi(X)$ in the sense of Definition~\ref{def:control:cyc}.  
\end{theorem}

\begin{example}
  With the use of elementary trigonometric identities, one may verify the H\"{o}rmander condition algebraically for a wide variety of configurations
  of $\mathcal{Z}$ and $f$ in (\ref{eq:RD:cont}).
  For example recall that
  \begin{align*}
    \sin(jx)&\sin(kx)\sin(lx)\\ 
    &= \frac{1}{4}( \sin((l + j -k) x) + \sin((l - j + k) x) - \sin((l + j +k) x) - \sin((l - j -k)x))
  \end{align*}
  for any $j, k, l$. Thus, in the case when the degree of $f$ is $3$, the H\"{o}rmander condition is satisfied if, for instance, 
  $\{1,2\} \subset \mathcal{Z}$.
\end{example}

Recall that the structure of the cocycle $\phi$ allows us to define a Markov semigroup $P_t$ with associated transitions 
$P_t(u_0, A)$, $u_0 \in L^2$ and $A\subseteq L^2$ Borel, as in~\eqref{eq:mar:sg:cyc} and~\eqref{eq:trn:fn:cyc}.  Combining the previous 
result with Lemma~\ref{l:control} and Corollary~\ref{cor:full:sup} of Section~\ref{sec:app:SPDEs}, we have the following immediate consequence.  

\begin{corollary}
Suppose the assumptions of Theorem~\ref{thm:RD:main} are satisfied.   
For any $t>0$, $u_0 \in L^2$ and $B\subseteq L^2$ open we have $P_t (u_0, B) >0.$  
Furthermore, any invariant measure $\mu$ for $P_t$ satisfies $\mu(B)>0$ for all $B\subseteq L^2$ open.   
\end{corollary}

\begin{remark}\label{rmk:its:complicated}
Note that a much broader class reaction-diffusion of equations in regards to boundary conditions, the structure of the reaction term $f$
and the spatial dimension are all accessible to the formalisms presented in Section~\ref{sec:sat} and Section~\ref{sec:app:SPDEs}.
We choose to focus on the special case presented in \eqref{eq:RD:cont}-\eqref{eq:RD:bcs} for simplicity and clarity of exposition 
in our first example.   

Similarly, to keep the presentation of this first example simple, we will avoid the Malliavin calculus 
and focus on the rigorous scaling arguments giving control on the phase space $L^2$ as stated in Theorem~\ref{thm:RD:main}.  
Indeed, to be able to apply the results of Section~\ref{sec:den:pos}, we need to establish a non-degeneracy for the Malliavin 
matrix associated with (\ref{eq:RD:cont}) al la Definition~\ref{def:Mal:Mtrx}.   
In previous work, \cite{HM09}, the analysis of this operator was carried out for (\ref{eq:RD:cont})
in a smoother space where the maximum principal is immediately applicable.  It is expected that the Malliavin analysis carried out in \cite{HM09}
could be readily extended to the $L^2$ setting followed here.  Conversely, with some further work, the controlability results of this section 
could be generalized to arbitary higher order Sobolev spaces which was the setting of \cite{HM09}.    

We leave both questions, along with more general more general formulations of (\ref{eq:RD:cont}), for future work.
\end{remark}

\subsubsection{Proof of the main control result}

Given the existence of the cocycle $\phi$, observe that by taking the parameter space
$Y= \text{span}\{ \sigma_k \, :\, k \in \mathcal{Z}\}$, we have defined a one-parameter family of continuous (global) semigroups 
$(t, u_0, h) \mapsto \Phi_t^h u_0 : [0, \infty) \times L^2\times Y \rightarrow L^2$ by setting 
\begin{align*}
\Phi_t^h u_0= \phi_t(u_0, V_h)
\end{align*}
 where $V_h \in \Omega$ is defined by $V_h(t) = t h$.  See Definition~\ref{def:oneparamlocal} in Section~\ref{sec:exactcontrol}.
Throughout, we will denote this one-parameter family using the notation $(\Phi, Y)$ and take $\mathfrak{F}=\{ (\Phi, Y)\}$.

The proof of Theorem~\ref{thm:RD:main} follows immediately from Corollary~\ref{cor:exact:con} once we establish the following result.  
\begin{theorem}
\label{thm:satrd}
For each $m\geq 0$, $(\rho, X_m) \in \text{\emph{Sat}}_u(\mathfrak{F})$.  
\end{theorem}

Here we recall that $(\rho, X_m)$ is the one-parameter family of continuous (global) semigroups defined by 
\begin{align*}
\rho_t^h u_0= u_0 + t h, \, t\geq 0, \, u_0\in L^2 , \, h\in X_m, 
\end{align*}  
and $X_m$ is as in \eqref{eq:it:dir:RD}.  The notion of the uniform saturate $\text{Sat}_u(\mathfrak{F})$
of a one-parameter family of continuous (global) semigroups  $\mathfrak{F}$ is given in Definition~\ref{def:uni:sat} above.

Theorem~\ref{thm:satrd} will be proven inductively using the next two scaling estimates.  
The first result, Lemma~\ref{lem:sc:lem:1:RD}, starts the inductive generation of the subspaces 
$X_m$ by showing that $(\rho, X_0)\in \text{Sat}_u (\mathfrak{F})$.  
In light of the heuristics outlined in Section~\ref{sec:overview}, 
the second scaling estimate, Lemma~\ref{lem:sc:lem:2:RD} allows us to `push' existing directions 
through the nonlinearity $f$ to iteratively show that $(\rho, X_m) \in \text{Sat}_u(\mathfrak{F})$ for all $m\geq 0$.  

\begin{remark}
By the proof of Proposition~\ref{prop:horeq}, we recall that $X_m$, $m\geq 1$, satisfies
\begin{align*}
X_m = \text{span} \Big\{ X_{m-1} \cup \{ f_*(h) \, : \, h \in X_{m-1}\} \Big\}.  
\end{align*}
Note that this simplifies the argument since we will only need to see how to generate directions of the form $f_*(h)$ for $h\in X_{m-1}$.  
\end{remark}

\begin{lemma}
\label{lem:sc:lem:1:RD}
Let $K_1\subseteq L^2$ and $K_2 \subseteq X_0$ be compact and fix $\eps, t> 0$.  Then there exists $\lambda_0 >0$ such that for all $\lambda \geq \lambda_0$ 
\begin{align}
\sup_{u_0 \in K_1, h\in K_2} \|  \SRD_{t/\lambda}^{\lambda h} u_0 -   \rho_t^h u_0 \| <\eps.    
  	\label{eq:fc:sc:rd}
\end{align}
In particular, $(\rho, X_0) \in \text{\emph{Sat}}_u(\mathfrak{F})$.    
\end{lemma}

\begin{lemma}
\label{lem:sc:lem:2:RD}
Fix $m\geq 0$ and let $K_1 \subseteq L^2$ and $K_2\subseteq X_m$ be compact.  Then for all $\eps, t>0$, there exists $\lambda_m >0$ such that for all $\lambda\geq \lambda_m$
\begin{align}
  	\sup_{u_0 \in K_1, h\in K_2} \| \rho_{\lambda^{-1}}^{-\lambda^2 h} \, \,  \Phi_{t/\lambda^{2n-1}}^{0} \,\,  \rho_{\lambda^{-1}}^{\lambda^2 h} u_0  
  	&-   \rho_t^{f_*(h)} u_0 \| < \eps
          \label{eq:sc:lem:2:RD}
	\end{align}
	where we recall that $2n-1$ is the degree of the polynomial nonlinearity $f$. \end{lemma}

Before turning to the proof of these two lemmata let us first make precise how Theorem~\ref{thm:satrd} follows assuming these two bounds. 
The proofs of these each of these lemmata are given immediately afterwards.
	
	\begin{proof}[Proof of Theorem~\ref{thm:satrd}]
	We note that $(\rho, X_0 ) \in \text{Sat}_u(\mathfrak{F})$ by Lemma~\ref{lem:sc:lem:1:RD}.  
        Also, since $f_*(\alpha u)= \alpha^{2n-1} f_*(u)$ and $2n-1$ is odd, Lemma~\ref{lem:sc:lem:2:RD} 
        implies that if $(\rho, X_m) \in \text{Sat}_u(\mathfrak{F})$ for some $m\geq 0$, then for all 
        $h\in X_m$, $(\rho, Y_{m+1}(h)) \in \text{Sat}_u(\mathfrak{F})$ where $Y_m(h) := \{ \alpha f_*(h) \, : \, \alpha \in \R\}$.  
        Since the ray semigroup has the property that 
	\begin{align*}
	\rho_{t}^{\alpha g} \rho_t^{\beta h} u_0 = \rho_{t}^{ (\alpha g + \beta h)}u_0
	\end{align*}  
	for all $g,h,u_0 \in L^2$, $t >0$ and $\alpha, \beta \in \R$, it follows 
        that $(\rho, X_{m+1}) \in \text{Sat}_u(\mathfrak{F})$.  This finishes the proof.    
	\end{proof}

	\begin{proof}[Proof of Lemma~\ref{lem:sc:lem:1:RD}]
          We proceed to establish a bound suitable for (\ref{eq:fc:sc:rd}) by estimating as follows 
          \begin{align}
            \|  \SRD_{t/\lambda}^{\lambda h} u_0 -   \rho_t^h u_0 \|  \leq    \|  \SRD_{t/\lambda}^{\lambda h} u_0 -   \rho_t^h \pi_N u_0 \|  + \|u_0 - \pi_N u_0\|
            \label{eq:stpd:split}
          \end{align}
          where $N > 0$ is to be determined.   Here, 
          $\pi_N: L^2 \rightarrow L^2$ is the projection onto the Fourier modes of size $N$ or less, i.e.,
          \begin{align*}
            \pi_N u = \sum_{k=1}^N u_k \sin(kx), \quad \text{ where } u_k = \langle u, \sin(kx) \rangle
          \end{align*}
          and $\langle \ccdot, \ccdot \rangle$ denots the $L^2$ inner product.  
          Introducing the shorthand notations
          \begin{align*}
            u_\lambda(t) = \SRD_{t/\lambda}^{\lambda h} u_0,
            \quad 
            w^N(t) = \rho_t^h   \pi_N u_0,
            \quad
            v_\lambda^N (t) =   \SRD_{t/\lambda}^{\lambda h} u_0 -  \rho_t^h  \pi^N u_0
                      =   u_\lambda(t) - w^N,
          \end{align*}
we observe that $v_\lambda^N$ satisfies, 
\begin{align}
    \partial_t v_\lambda^N- \frac{\kappa}{\lambda} \partial_{xx}v_\lambda^N  = \frac{1}{\lambda} f(u_{\lambda}) + \frac{\kappa}{\lambda} \partial_{xx}w^N,
  \label{eq:diff:ray:RD:rscl}
\end{align}
cf. \eqref{eq:rescal:0} above.
Thus, taking an $L^2$ inner product with $v_\lambda$
\begin{align*}
\frac{d}{dt} \| v_\lambda^N\|^2  + \frac{2 \kappa}{\lambda} \| \partial_x v_\lambda^N\|^2  
  = - \frac{2 \kappa}{\lambda}\langle \partial_{x} v_\lambda^N, \partial_{x} w_\lambda^N \rangle+ \frac{2}{\lambda}\langle v_\lambda^N,  f(w^N)\rangle
                                   + \frac{2}{\lambda}\langle v_\lambda^N, f(u_\lambda) - f(w^N)\rangle.   
\end{align*}
Since $f(u)-f(v)= f'(\xi_{u,v})(u-v)$ for some $\xi_{u,v}$ lying between $u$ and $v$ and since, cf. \eqref{eq:RD:nl:prop},
we have that $\sup_{z \in \RR} f'(z) \leq K$, we infer
\begin{align*}
\frac{d}{dt} \| v_\lambda^N\|^2 
  \leq&   \frac{2K}{\lambda} \| v_\lambda^N \|^2 + \frac{\kappa}{\lambda}     \| w^N \|_{H^1}^2  + \frac{C}{\lambda \kappa} \| f(w^N)\|_{L^1}^{2}\\
  \leq&   \frac{2K}{\lambda} \| v_\lambda^N \|^2 + \frac{C}{\lambda} (1+  \| w^N\|_{H^1}^{2(2n-1)}),
\end{align*}
where we have also used the $1D$ Sobolev embedding of $H^1 \subseteq L^\infty$ and Young's inequality.
Here, crucially  $C = C(\kappa) > 0$ is independent of $\lambda > 0$.  Gronwall's inequality then implies that 
\begin{align*}
\| v_\lambda^N(t) \|^2 \leq& \bigg\{\| u_0 - \pi_N u_0\|^2 +\frac{C}{\lambda} \int_0^t (1+  \| w^N\|_{H^1}^{2(2n-1)}) \, ds \bigg\} e^{\frac{2K}{\lambda} t}\\
  \leq& C \biggl(\| u_0 - \pi_N u_0\|^2 +\frac{1+  \| \pi_N u_0 \|^{2(2n-1)}_{H^1} + \| h \|^{2(2n-1)}_{H^1}}{\lambda} \biggr).
\end{align*}
With this bound, \eqref{eq:stpd:split} and the inverse Poincar\'e inequality, we conclude
that, for any $u_0 \in L^2$ and $h \in X_0$,
\begin{align}
  \|  \SRD_{t/\lambda}^{\lambda h} u_0 -   \rho_t^h u_0 \| 
  \leq C_1 \biggl(\| u_0 - \pi_N u_0\| +\frac{1+  N^{2n-1} \| u_0 \|^{2n-1} + \tilde{N}^{2n-1} \| h \|^{2n-1}}{\sqrt{\lambda}}  \biggr)
    \label{eq:scale:pt:bound}
\end{align}
where $\tilde{N} = \max\{k > 0: k \in \mathcal{Z}\}$.  Here we emphasize that the constant $C_1 = C_1(\kappa, K , t) > 0$ 
is independent of $\lambda > 0, u_0 \in L^2, h \in X_0$ as well 
as $N > 0$ and $\tilde{N}$.

Let $\eps, t >0$ and $K_1 \subset L^2$, $K_2 \subset X_0$ compact be arbitrarily given.  Cover $K_1$ with a finite collection of balls
$B_{\bar{\eps}}(v_0^{(1)}), \ldots, B_{\bar{\eps}}(v_0^{(M)})$
where $\bar{\eps} = \eps/(4 C_1)$ and the constant $C_1$ is as in \eqref{eq:scale:pt:bound}.  Then, from \eqref{eq:scale:pt:bound}, we obtain
\begin{align}
  \sup_{u_0 \in K_1, h \in K_2}\|  \SRD_{t/\lambda}^{\lambda h} u_0 -   \rho_t^h u_0 \|  
  \leq& C_1 \max_{j = 1, \ldots, M} \sup_{u_0 \in B_{\bar{\eps}}(v_0^{(M)})} (2\|u_0 - v_0^{(j)}\| + \|v_0^{(j)} - \pi_N v_0^{(j)}\|)  \notag\\
     &+\frac{N^{2n-1} C_1}{\sqrt{\lambda}} \biggl(1+   \sup_{u_0 \in K_1} \| u_0 \|^{2n-1} + \sup_{h \in K_2} \| h \|^{2n-1}\biggr) \notag\\
  \leq& C_1 \max_{j = 1, \ldots, M} \|v_0^{(j)} - \pi_N v_0^{(j)}\|  + \frac{C_2 N^{2n-1} }{\sqrt{\lambda}} + \frac{\eps}{2}
        \label{eq:cmp:sets:bnd}
\end{align}
for any $N \geq \tilde{N}$ where the constant $C_1 = C_1(\kappa, K)$ is independent of the compact sets $K_1, K_2$ both $C_1$ and $C_2$ are 
independent of $N, \lambda > 0$ and $\epsilon >0$.  We can thus choose $N > 0$ sufficiently large so that 
\begin{align}
  C_1 \max_{j = 1, \ldots, M} \|v_0^{(j)} - \pi_N v_0^{(j)}\|  \leq \frac{\epsilon}{4}
\end{align}
Since this choice of $N$ is taken independent of $C_2$ the desired bound, \eqref{eq:fc:sc:rd}, follows for all $\lambda  \geq \lambda_0 = \frac{16C_2^2N^{2(2n-1)}}{\sqrt{\eps}}$.
This finishes the proof of Lemma~\ref{lem:sc:lem:1:RD}. 
\end{proof}

\begin{proof}[Proof of Lemma~\ref{lem:sc:lem:2:RD}]
Fix $m\geq 0$ and let $K_1\subseteq L_0^2$ and $K_2\subseteq X_m$ be compact.  Fix $\eps, t>0$.  
Maintaining the notation that $\pi_N: L^2\rightarrow L^2$ denotes the projection onto the modes of size $N$ or less, 
we let
\begin{align*}
  w_\lambda^N(t) &=  \rho_{\lambda^{-1}}^{-\lambda^2 h}\Phi_{t/\lambda^{2n-1}}^{0} \,\,  \rho_{\lambda^{-1}}^{\lambda^2 h} u_0,\\
  r^N(t)&= \rho_t^{f_*(h)} \pi_N u_0 ,\\  	
  \phi_\lambda^N(t) &= \rho_{\lambda^{-1}}^{-\lambda^2 h} \, \,  \Phi_{t/\lambda^{2n-1}}^{0} \,\,  \rho_{\lambda^{-1}}^{\lambda^2 h} u_0  
  	-   \rho_t^{f_*(h)} \pi_N u_0 = w_\lambda(t) - r^N.
\end{align*}
As in the proof of the previous lemma $N> 0$ is a free parameter which will be fixed futher on below.

Referring back to \eqref{eq:rs:nLin:twit} we see that $w_\lambda^N$ satisfies
\begin{align*}
  \partial_t w_\lambda^N - \frac{\kappa}{\lambda^{2n-1}} \partial_{xx} (w_\lambda^N + \lambda h) = \frac{1}{\lambda^{2n-1}} f(w_\lambda^N + \lambda h).
\end{align*}
With this equation and using also that $w_\lambda^N = \phi_\lambda^N + r^N$, we conclude that $\phi_\lambda^N$ obeys
\begin{align}
  \partial_t \phi^N_\lambda - \frac{\kappa}{\lambda^{2n-1}} \partial_{xx} \phi^N_\lambda 
          =& \frac{\kappa}{\lambda^{2n-1}} \partial_{xx}(r^N + \lambda h)  
            + \frac{f(w_\lambda^N + \lambda h) - f(r^N + \lambda h)}{\lambda^{2n-1}} \notag\\
  &+ \frac{f(r^N + \lambda h) - f^*(\lambda h)}{\lambda^{2n-1}}.
    \label{eq:nlt:error:eqn}
\end{align}
Multiplying \eqref{eq:nlt:error:eqn} by $\phi^N_\lambda$ and integrating we obtain
\begin{align}
\frac{d}{dt}\| \phi_\lambda^N \|^2 + \frac{2\kappa}{\lambda^{2n-1}}&\|\partial_x \phi_\lambda^N\|^2
  = \frac{2\kappa}{\lambda^{2n-1} }\langle \phi_\lambda^N,  \partial_{xx}r^N+ \lambda \partial_{xx}h\rangle  \notag\\
  &+ \frac{2}{\lambda^{2n-1}} 
    \langle \phi_\lambda^N, (f(w_\lambda^N + \lambda h)- f(r^N + \lambda h)) + (f(r^N + \lambda h) - f^*(\lambda h)) \rangle.
    \label{eq:nlt:en:bnd:1}
\end{align}
Since $f(w_\lambda^N + \lambda h)- f(r^N + \lambda h) = f'(\xi)\phi$ for some $\xi$ between $w_\lambda^N + \lambda h$
and $r^N-\lambda h$ and using \eqref{eq:RD:nl:prop} we have that
\begin{align}
  \frac{2 \langle \phi_\lambda^N, 
       f(w_\lambda^N + \lambda h)- f(r^N +\lambda h) \rangle}{\lambda^{2n-1}} 
         \leq \frac{2K}{\lambda^{2n-1}} \|\phi_\lambda^N\|^2
  \label{eq:nlt:en:bnd:2}
\end{align}
On the other hand, from \eqref{eq:RD:nonlin}, we have
\begin{align*}
 f(r^N + \lambda h) - f^*(\lambda h)
   = \sum_{k =0}^{2n-2}b_k \sum_{l =0}^{k} {k \choose l} \lambda^l h^l (r^N)^{k - l} 
    + b_{2n-1}\sum_{l = 0}^{2n-2}{2n-1 \choose l} \lambda^l h^l (r^N)^{2n-1 - l}.
\end{align*}
Thus, for $\lambda \geq 1$, we have
\begin{align}
\frac{\langle \phi_\lambda^N, f(r^N + \lambda h) - f^*(\lambda h) \rangle}{\lambda^{2n-1}}
  \leq \frac{C}{\lambda}\|\phi_\lambda^N\| (1 + \|h\|_{L^{4(2n -1)}}^{2n-1})(1+ \|r^N\|_{L^{4(2n -1)}}^{2n-1})
  \label{eq:nlt:en:bnd:3}
\end{align}
for a constant $C$ depending only on $f$ and which is in particular independent of $\lambda \geq 1$.  

Combining \eqref{eq:nlt:en:bnd:1}--\eqref{eq:nlt:en:bnd:3} we infer:
\begin{align*}
\frac{d}{dt} \| \phi_\lambda^N(t) \|
  &\leq \frac{2K}{\lambda^{2n-1}} \|\phi_\lambda^N\| + 
    \frac{C (\| r^N\|_{H^2} + \lambda \| h\|_{H^2} +\lambda^{2n-2} (1 + \|h\|_{L^{4(2n -1)}}^{2n-1})(1+ \|r^N\|_{L^{4(2n -1)}}^{2n-1}))}{\lambda^{2n-1}} \\
  &\leq \frac{2K}{\lambda^{2n-1}} \|\phi_\lambda^N\| + C
     \frac{  (1 + \|h\|_{H^2}^{2n-1})(1+ \|r^N\|_{H^2}^{2n-1})}{\lambda} \\
  &\leq \frac{2K}{\lambda^{2n-1}} \|\phi_\lambda^N\| + C
     \frac{  (1 + \|h\|_{H^2}^{2n-1})(1+ N^{2(2n-1)}\| u_0\|^{2n-1} + \|f^*(h)\|_{H^2}^{2n-1})}{\lambda} 
\end{align*} 
where we have also used Sobolev embedding and the inverse Poincar\'e inequality.    Here the generic constant $C > 0$ is independent
of $N, \lambda$ and the data.
Hence, with Gr\"onwall's inequality 
\begin{align}
\| \phi_\lambda^N(t) \| \leq C \bigg(\|u_0 - \pi_N u_0 \| +
                        \frac{  (1 + \|h\|_{H^2}^{2n-1})(1+ N^{2(2n-1)}\| u_0\|^{2n-1} + \|f^*(h)\|_{H^2}^{2n-1})}{\lambda} \bigg)
  \label{eq:nlt:gron:bnd}
\end{align}
where again the constant $C = C(t, \kappa, f)$ is independent of $\lambda, N, u_0$ and $h$.

With \eqref{eq:nlt:gron:bnd} in hand we infer the desired bound \eqref{eq:sc:lem:2:RD} by employing an argument very similar to the one used in the proof of the previous Lemma.  See \eqref{eq:cmp:sets:bnd} above.  The proof of Lemma~\ref{lem:sc:lem:2:RD} is now complete.
\end{proof}

\subsection{2D Incompressible Navier-Stokes Equations}
\label{sec:2DNSE}

For our next example, we explain how the scaling and saturation 
framework can be applied to the 2D Navier-Stokes equations and its stochastic 
counterpart.  In particular, it is worth underlining that the control framework developed here allows us to show 
unique ergodicity for the 2D stochastic Navier-Stokes equations
established in \cite{HM06} even when the equations are subject
to a more or less arbitrary background forcing.

Many of the results in this subsection have
been previously established, although with different techniques.  
We will therefore be more sparing in technical details in this section.
We refer the reader to \cite{AgrachevSarychev2005,AgrachevSarychev2006,
MattinglyPardoux06, AgrachevKuksinSarychevShirikyan2007} 
and to the introduction for further references concerning the low-mode control problem for the 2D Navier-Stokes equations.

The 2D Navier-Stokes equations take the form
\begin{align}
  &\pd_t \bfU  + \bfU \cdot \nabla \bfU - \nu \Delta \bfU + \nabla p
    = \mathbf{f} + \mathbf{\rho} \cdot \partial_t V, 
    \quad \nabla \cdot \bfU = 0,
    \label{eq:NSE:con}
\end{align}
where the unknowns are the velocity $\bfU = (u_1, u_2): \TT^2 \to \RR^2$ and pressure 
$p: \TT^2 \to \RR$, the latter of which is a constraint maintaining the divergence-free
condition of the flow.  As in the aforementioned works, we consider the 2D Navier-Stokes
equations on $\TT^2$ so that the nonlinear interactions 
are more tractable to analyze.   The parameter $\nu > 0$ in  \eqref{eq:NSE} is the kinematic viscosity.
The volumetric source term $\mathbf{f} = (f_1, f_2): \TT^2 \to \RR^2$ may be used to model components of a large-scale
stirring mechanism, but it will be taken to have an essentially arbitrary form for the mathematical results which follow.
As in the previous example, $\rho$ represents a finite set of sinusoidal control directions driven by the actuators $V$.  
See \eqref{eq:NSE:con:form} below for the precise formulation of the controls we consider.

As with the other systems in this section, when we consider the stochastic process that results from 
taking $V$ to be Brownian motion in \eqref{eq:NSE:vort:shit} we obtain solutions of the stochastic Navier-Stokes equation.
Of course, for general $V \in \Omega$ we do not make sense
of~\eqref{eq:NSE:con} directly.  Rather, we
use the additive structure of the noise/control term and work with a shifted equation.  
See Remark~\ref{rem:pdecocycle} and Proposition~\ref{prop:2D:NSE:Well:Pos} below.

For simplicity and to connect with previous results, we consider \eqref{eq:NSE:con} 
in its vorticity formulation.  Taking $\Vort = \mbox{curl}(\bfU) = \nabla^{\perp} \cdot \bfU =  \partial_{x_1}u_2 - \partial_{x_2} u_1$, we have
\begin{align}
  \pd_t \Vort + \bfU \cdot \nabla \Vort - \nu \Delta \Vort 
    = g + \sigma \cdot \partial_t V, \quad  \bfU = \mathcal{K} \ast \Vort
    \label{eq:NSE:vort:con}
\end{align}
where $g = \mbox{curl}(\mathbf{f}) = \partial_{x_1}f_2 - \partial_{x_2} f_1$.  In the above, $\mathcal{K}$ denotes the Biot-Savart kernel which recovers $\bfU$ from $\Vort$, thus allowing us to consider the vorticity
formulation of \eqref{eq:NSE:con} in a closed form.\footnote{Here recall
that given any $\Vort \in H$ we define the stream function $\psi$ as the solution of
$\Delta \psi = \xi$ supplemented with periodic boundary conditions.  We then take
\begin{align}
  \bfU = \nabla^\perp \psi = \nabla^\perp (\Delta)^{-1} \xi
  \label{eq:biot:savart}
\end{align}
so that the operator $\mathcal{K}$ has the symbol $k^\perp/|k|^2$.}

Following the setup in~\cite{HM06}, we will consider controls of the type
\begin{align}
 \sigma \cdot V = \sum_{k \in \Zb} (v_k^0(t) \cos(k \cdot x) + v_k^1(t) \sin(k \cdot x)), \,\, V(t) = (v_k^0(t), v_k^1(t))_{k\in \mathcal{Z}}, 
\label{eq:NSE:con:form}
\end{align}
where $V\in \Omega := \{ V: (-\infty, \infty) \rightarrow \RR^{2|\mathcal{Z}|}\, \,\text{ continuous with } \,\, V(0)=0\}$.  Here the controlled
set of frequencies $\mathcal{Z}$  sit in the upper half plane
\begin{align*}
  \Zb \subseteq \Z^{2}_{+} := \left\{ j = (j_1, j_2) \in \Z^2_{\neq 0}: j_1 > 0 \textrm{ or } j_1 = 0, j_2 > 0 \right\}.
\end{align*}
We show that configurations of $\Zb$ for which we have two non-orthogonal frequencies of distinct length are controllable in what follows.

\subsubsection*{Mathematical Formulations}
Let us begin by futher recalling the mathematical setting of \eqref{eq:NSE:vort:con}.   As in the previous example, cf. Remark~\ref{rem:pdecocycle},
we will consider solutions of the shifted system
\begin{align}
  &\pd_t \VortS + \mathcal{K}(\VortS + \sigma \cdot V) \cdot \nabla (\VortS + \sigma \cdot V) - \nu \Delta (\VortS + \sigma \cdot V)
    = g
    \label{eq:NSE:vort:shit}
\end{align}
where $\mathcal{K}$ is the Biot-Savart operator.  We thus \emph{define} the solution of \eqref{eq:NSE:vort:con} corresponding to 
initial condition $\Vort_0$ and `noise path' $V \in \Omega$ as $\Vort(\ccdot, \Vort_0, V) := \VortS(\ccdot, \Vort_0, V) + \sigma \cdot V$, where 
$\VortS(\ccdot, \Vort_0, V)$ is the solution of \eqref{eq:NSE:vort:shit} with the given $V$ starting from $\Vort_0$.  Regarding the functional setting for \eqref{eq:NSE:vort:con} and its associated cocycle, we consider solutions 
evolving on 
\begin{align}
  H = \bigg\{ \Vort \in L^2(\TT^2): \int_{\TT^2} \Vort \, dx = 0 \bigg\}
  \label{eq:NSE:PhSp:V}
\end{align}
recalling that solutions of \eqref{eq:NSE:vort:con} maintain the following mean-free condition: 
\begin{align*}
\int_{\TT^2} \Vort \, dx =0
\end{align*}
provided that the source $g$ does.\footnote{As usual this 
assumption is mathematically convenient as it guarantees that the Poincar\'e inequality holds.  The general case follows in any case from 
a Galilean transformation.}  In what follows, we maintain the notation $\| \ccdot \|$ and $\langle \ccdot, \ccdot \rangle$ 
for the usual $L^2$ norm and inner product.  We denote the higher order Sobolev spaces according to $H^m := H^m(\TT^2) \cap H$ for $m \geq 1$.

We have the following: 
\begin{proposition}\label{prop:2D:NSE:Well:Pos}
  Fix any $g \in H$ and assume that $\sigma$ consists of a finite number of frequencies (i.e. suppose $\Zb$ is a finite set).
  Then, for any $\Vort_0 \in H$ and any $V \in \Omega$, there is a unique $\Vort(\ccdot) = \Vort(\ccdot, \Vort_0, V)$ with
  \begin{align}
    \Vort \in L^2_{loc}([0,\infty); H^1)  \cap C([0,\infty); H)
  \end{align}
  with $\VortS = \Vort - \sigma \cdot V$ solving \eqref{eq:NSE:vort:shit} (in the usual weak sense).  This solution $\Vort(\ccdot, \Vort_0, V)$
  depends continuously in $[0,\infty) \times H \times \Omega$ on $t, \Vort_0$ and $V$ so that \eqref{eq:NSE:vort:con} uniquely defines a continuous
  adapted cocycle $\phi$ in the sense of Definition~\ref{def:co-cyc}.  Additionally, this cocycle satisfies Assumption~\ref{ass:Mal:der}.
\end{proposition}

Much of Proposition~\ref{prop:2D:NSE:Well:Pos} is essentially standard and we refer the reader to, e.g., \cite{ConstantinFoias1988, Temam1995} for the necessary estimates and technical details.  In regards to the cocycle associated with \eqref{eq:NSE:vort:con} satisfying Assumption~\ref{ass:Mal:der}, 
see the Appendix in \cite{HM06}.

\subsubsection*{Low Mode Control Results}
With the mathematical framework provided by Proposition~\ref{prop:2D:NSE:Well:Pos} in hand, we now state the main controllability results
for \eqref{eq:NSE:vort:con}.    For these results we make use of the following spanning condition on the controlled directions $\mathcal{Z}$
in \eqref{eq:NSE:con:form} found in \cite{HM06}:
\begin{definition}\label{def:Span:cond}
  We say that $\mathcal{Z}$ is a \emph{sufficent control set} if:
  \begin{itemize}
    \item[(i)]  There exists two elements $k_1, k_2 \in \mathcal{Z}$ such that $|k_1| \not = |k_2|$.
    \item[(ii)]  Integer linear combinations of elements of $\mathcal{Z}$ generate $\Z^2$.  
  \end{itemize}
\end{definition}

As a concrete example we have that $\mathcal{Z}$ is a \emph{sufficent control set} if $\mathcal{Z} \supseteq \{ (1,0), (1,1) \}$.

\begin{remark}\label{rmk:rel:deg:2DNSE}
As shown in \cite{HM06}, the condition given in Definition~\ref{def:Span:cond} yields the spanning condition
\eqref{eq:den:brak:prelim}.    Also note that, for any $k$
\begin{align}
  (\mathcal{K} \ast \sin(k \cdot x)) \cdot \nabla \sin(k \cdot x) = \cos^2(k \cdot x) \frac{k^\perp \cdot k}{|k|^2} = 0
  \label{eq:rel:deg:NSE}
\end{align}
and similarly for other combinations of sines and cosine functions.  Thus the `relative degree conditions' given above in Section~\ref{rem:even}
apply and we may obtain the spaces \eqref{eq:brak:gen} in an iterative fashion with scaling and saturation arguments.  Properties analogous
to \eqref{eq:rel:deg:NSE} also play a key role for the Boussinesq and Euler examples below.
\end{remark}

Our main control result is the following:
\begin{theorem}\label{thm:aprx:ex:cont:NSE}
  Suppose that $\mathcal{Z} \subseteq \Z^2_{+}$ defining the controlled directions in \eqref{eq:NSE:vort:con} 
  is a sufficient control set in the sense of Definition~\ref{def:Span:cond}.  For any $g \in H$, consider the
  cocycle $\phi$ corresponding to $\mathcal{Z}$ and $g$ for the 2D Navier-Stokes equations (defined according to Proposition~\ref{prop:2D:NSE:Well:Pos}).  Then for any continuous, finite dimensional projection $\pi: H \to \RR^m$, $\phi$ is approximately controllable and exactly controllable on $\pi(H)$ 
  as in Definition~\ref{def:control:cyc}; that is, for any $\Vort_{i}, \Vort_{f} \in H$ and any time $t > 0, \eps > 0$ there exists a $V \in \Omega$ such that
  \begin{align*}
    \|\Vort(t, \Vort_{i},V) - \Vort_{f}\| < \eps  \quad \text{ and } \quad  \pi(\Vort(t, \Vort_{i},V)) = \pi(\Vort_{f}).
  \end{align*}
\end{theorem}

This result follows immediately from scaling estimates of the type \eqref{eq:fc:sc:rd} and \eqref{eq:sc:lem:2:RD} analyzed in the previous example.  
Since we will detail such estimates the fluids setting in the analogous but more difficult cases of the Boussinesq equation 
\eqref{eq:b:1}--\eqref{eq:b:2} (see Lemmas \ref{lem:0:conv}, \ref{lem:sa:bouseq}) and the 3D Euler equations \eqref{eq:NSE}
(cf. Lemmas \ref{lem:NSfa}, \ref{lem:NSsa}) we omit further details here.  Note that, as for the 3D Euler equations below, the approximate controllability
of 2D Navier-Stokes equations under the conditions in Theorem~\ref{thm:aprx:ex:cont:NSE} can also be established in $H^m$ for every $m \geq 1$.

\subsubsection*{Implications for the Stochastic Navier-Stokes Equations}

Let us now describe some implications of Theorem~\ref{def:Span:cond} for the stochastic 2D Navier-Stokes equations.  That is, in vorticity form we now consider 
\begin{align}
   d\Vort + (\bfU \cdot \nabla \Vort - \nu \Delta \Vort) \, dt
    = g \, dt + \sigma\,  dW, \quad  \bfU = \mathcal{K} \ast \Vort.
  \label{eq:SNSE:2D}
\end{align}
As before, the solution evolves on the periodic box $\TT^2$, $\mathcal{K}$ denotes the Biot-Savart kernel, and $g$ is any element in $H$.  The stochastic forcing $\sigma \, dW$ maintains the structure given in \eqref{eq:NSE:con:form}.  Solutions $\Vort(t, \Vort_0, W)$ of \eqref{eq:SNSE:2D} define a 
Markov transition kernel via
\begin{align*}
P_t(\Vort_0, A) = \PP(\Vort(t, u_0,W) \in A)
\end{align*}
where $\Vort_0$ is any element on $H$ and $A \subseteq H$ is Borel.
As in the previous example, we take $\phi$ to be the cocycle 
corresponding to \eqref{eq:NSE:vort:con} defined according to Proposition~\ref{prop:2D:NSE:Well:Pos}.  

We have the following results concerning \eqref{eq:SNSE:2D}:
\begin{theorem}\label{thm:SNSE:Imp}
  Consider any $g \in H$ and any $\sigma$ corresponding to a $\mathcal{Z}$
  which is a sufficent control set in the sense Definition~\ref{def:Span:cond}.
  Then:
  \begin{itemize}
    \item[(i)] The resulting Markov kernel defined by \eqref{eq:SNSE:2D} possesses exactly one invariant measure 
      $\mu$ which is ergodic.  Moreover, $\supp \mu = H$.\footnote{The Markov semigroup $\{P_t\}$ may furthermore 
      be shown to be mixing in a suitable Wasserstein distance.  See \cite{HairerMattingly2008}.}
    \item[(ii)] For any $\Vort_0 \in H$, $t > 0$ and any continuous projection $\pi: H \to \RR^m$
      onto a finite-dimensional subspace, the probability law of $\pi(\Vort(t, \Vort_0))$ is absolutely
      continuous with respect to Lebesgue measure on $\RR^m$ and its probability density is almost everywhere positive.
  \end{itemize}
\end{theorem}

Theorem~\ref{thm:SNSE:Imp} may be established using \cite{MattinglyPardoux06,HM06,HM09} combined with 
Theorem~\ref{thm:aprx:ex:cont:NSE} and the results in Section~\ref{sec:app:SPDEs}.  For the first 
item, we proceed by establishing the condition required by Corollary~\ref{cor:uniq:criteria}.  The 
asymptotic strong Feller condition \eqref{eq:ASF} is demonstrated exactly as in \cite{HM06,HM09}
 and relies in particular on a spectral analysis of the Malliavin matrix associated to \eqref{eq:SNSE:2D}.
See Definition~\ref{def:Mal:Mtrx} and the condition \eqref{eq:M:pos:cond} above.    The other condition 
in Corollary~\ref{cor:uniq:criteria} concerns approximate controlability and follows from
Theorem~\ref{thm:aprx:ex:cont:NSE}.  The full support of the invariant measure is an immediate consequence of
Theorem~\ref{thm:aprx:ex:cont:NSE} combined with Corollary~\ref{cor:full:sup}.
Regarding the second item concerning the regularity of the law of $\pi(\xi)$,
we again combing the estimate \eqref{eq:M:pos:cond} with the exact controllability guaranteed by Theorem~\ref{thm:aprx:ex:cont:NSE}
to infer the desired support properties from Theorems~\ref{prop:exod}, \ref{thm:pos}.

\begin{remark}
In the case when $g=0$, unique ergodicity of~\eqref{eq:SNSE:2D} follows using the same methods as above but one does not need the control theoretic approach outlined in Section~\ref{sec:sat}.  Indeed when $g=0$, the solution in the absence of noise relaxes to zero as $t\rightarrow \infty$; i.e., the global attractor is trivial.  Thus setting the control to be identically zero then implies that $0$ is in the support of any ergodic invariant probability measure.  Hence by ergodic decomposition and the asymptotic strong Feller property, there can be only one such measure. On the other hand when $g\neq 0$, the time infinity deterministic dynamics are highly non-trivial.  Thus one needs further, delicate control arguments to establish topological irreducibility needed to ensure unique ergodicity of the stochastic system.  It is important to point out that the main control result in the case of $g\neq 0$ follows by the main results in~\cite{AgrachevSarychev2005,AgrachevSarychev2006} using the Agrachev-Sarychev approach.             
\end{remark}

\subsection{Boussinesq Equations}
\label{sec:be:eq}

We next consider an example involving the Boussinesq Equations for convective fluids.
These equations couple the Navier-Stokes equation to an active scalar equation evolving the
temperature (or some other proxy determining the density)
of the fluid.   The crucial approximation here is that the density may be regarded
as constant with the important exception of terms due to buoyancy forces.  

In this example, we are interested in the case where a volumetric random
forcing/control acts  only in the equation for the density (or
temperature) through a few
select frequencies.  Specifically, we consider a 2D formulation of 
the Boussinesq equations in the absence of boundaries.  This specific setup is partially motivated 
by the recent work \cite{FoldesGlattHoltzRichardsThomann2013}.
More generally, note that stochastic perturbations acting in the temperature equation as in \eqref{eq:b:2} below has a 
significant physical motivation as a model for radiogenic heating and other volumetric heat 
sources driving turbulent convection.  See \cite{SwHo1977, ScTuOl2001, FoGlRiWh2016,FoldesGlattHoltzRichardsWhitehead2017}. 

 From the point of view of the control theoretic formalism developed
 here, it is worth emphasizing that the Boussinesq equations present 
 a more delicate set of nonlinear interactions compared with 
 the other examples considered in this section.    In particular, this example
 illustrates that scalings detailed in Section~\ref{sec:overview},
 while very powerful, are by no means the only way of leveraging the
 saturation formalism introduced in Section~\ref{sec:sat}.

\subsubsection{Mathematical Formulation}
Following \cite{FoldesGlattHoltzRichardsThomann2013} it will be convenient to consider the Boussinesq Equations in terms of
the vorticity of the flow.  In this formulation the equations read 
\begin{align}
  &\pd_t \Vort + u \cdot \nabla \Vort - \nu \Delta \Vort = g \pd_x \Th,  \quad \Vort(0) = \Vort_0 \label{eq:b:1}\\
  &\pd_t \Th + u \cdot \nabla \Th - \kappa \Delta \Th =  h^0 + \sigma \cdot \partial_t V,  \quad \Th(0) = \Th_0.
  \label{eq:b:2}
\end{align}
where  $\Vort = \nabla^\perp \cdot u = \pd_{x}u_2 - \pd_{y} u_1$ is
the vorticity of the velocity $u=(u_1, u_2)$ and $\Th$ is the
temperature of the fluid.    The system \eqref{eq:b:1}--\eqref{eq:b:2} is posed on $\TT^2 = [-\pi, \pi]^2$
with periodic boundary conditions.\footnote{Note that, as with the 2D
Navier-Stokes equations in \eqref{eq:NSE:vort:con} the vorticity
formulation in \eqref{eq:b:1}--\eqref{eq:b:2} represents a closed
system of equations as $u$ is uniquely recovered from $\Vort$ 
via the Biot-Savart law.  See \eqref{eq:biot:savart} above.}    

The physical parameters in the problem are $\nu, \kappa, g >0$, which
correspond to the kinematic viscosity, thermal diffusivity and
gravitation constants, respectively.  
The thermal body force $h^0 + \sigma \cdot \partial_t V$ is such that $h^0: \TT^2 \to \RR$ is any fixed sufficiently smooth function and  
\begin{align}
\label{eq:bcontrol:form}
 \sigma \cdot V = \sum_{k \in \Zb} (v_k^0(t) \cos(k \cdot x) + v_k^1(t) \sin(k \cdot x)), \,\, V(t) = (v_k^0(t), v_k^1(t))_{k\in \mathcal{Z}}, 
\end{align}
where $V\in \Omega := \{ V: (-\infty, \infty) \rightarrow \RR^{2|\mathcal{Z}|}\, \,\text{ continuous with } \,\, V(0)=0\}$.  
Also, in the sum above, 
\begin{align*}
  \Zb \subseteq \Z^{2}_{+} := \left\{ j = (j_1, j_2) \in \Z^2_0: j_1 > 0 \textrm{ or } j_1 = 0, j_2 > 0 \right\}
\end{align*}
are the directions which are directly actuated by the term $\sigma
\cdot \partial_t V$.   We will make further assumptions on
$\mathcal{Z}$ below for the control results in Theorem~\ref{thm:cntrl:BE}.  See also Remark~\ref{rmk:mo:dirs:buca}.

Note that, as in the previous example of the 2D Naiver-Stokes 
equations, the system \eqref{eq:b:1}--\eqref{eq:b:2} preserves the mean
value of solutions.   As such we will again restrict our discussion
to mean-zero solutions. In particular, we will invoke the Poincar\'e
inequality in the estimates below.

For most of the following discussion, we consider solutions of the Boussinesq Equations
evolving continuously in $L^2$.  Thus accounting for the mean zero 
condition, we take the phase space to be
\begin{align*}
  H = \left\{ U = (\Vort, \Th) \in (L^2(\TT^2))^2: \int \Vort \, dx = \int \Th \, dx = 0  \right\}.
\end{align*}
We will at times also consider smoother solutions of \eqref{eq:b:1}--\eqref{eq:b:2}
and hence make use of the Hilbert spaces $H^m:= H^m(\TT^2)^2 \cap H$
for $m\geq 1$.

Following the discussion in Remark~\ref{rem:pdecocycle} as well as the
setting of Section~\ref{eq:rd:ex}, we recall that the solution
$U(t) = (\Vort(t), \Th(t))=(\Vort(t, U_0,  V), \theta(t, U_0,  V))$ 
of \eqref{eq:b:1}-\eqref{eq:b:2}  with initial condition $U_0  = (\Vort_0, \Th_0) \in H$ is defined by 
\begin{align*}
  U(t) = (\Vort(t), \Th(t)):=
  (\tom(t,U_0, V), \tth(t, U_0 , V)) + (0, \sigma \cdot V) 
\end{align*}
where $\tilde{U} := (\tom(t,U_0, \sigma \cdot V), \tth(t, U_0 , \sigma \cdot V))$ satisfies $(\tom(0), \tth(0))=(\Vort_0, \Th_0)$ and
\begin{align}
\label{eq:sb:1}
&\partial_t \tom + \tilde{u} \cdot \nabla \tom - \nu \Delta \tom = g \partial_x (\tth + \sigma \cdot V)\\
\label{eq:sb:2}
&\partial_t \tth + \tilde{u} \cdot \nabla (\tth + \sigma \cdot V) - \kappa \Delta (\tth + \sigma \cdot V) = h^0 .
\end{align} 
We recall that the shifted equation above allows us to consider solutions of \eqref{eq:b:1}--\eqref{eq:b:2} when $V$ is merely continuous.  
Additionally, if $V$ is replaced by a standard two-sided Brownian motion $W$ on $\R^{2|\mathcal{Z}|}$ in the equations above, 
the resulting random process $U = (\Vort, \Th)$ is the same as the 
one defined by~\eqref{eq:b:1}-\eqref{eq:b:2}, again with $V$ replaced by $W$, using the It\^{o} calculus.  See 
Remarks~\ref{rem:pdecocycle},~\ref{rmk:shit:sln:RD} above.

With these preliminaries in hand, we next state the main structural result which
allows us to apply the results of Section~\ref{sec:sat} and Section~\ref{sec:app:SPDEs}.  
\begin{proposition}
\label{prop:b:structure}
For every $U_0=(\Vort_0, \Th_0) \in H$ and $V\in \Omega$, 
there exists a unique 
\begin{align*}
  U = (\Vort, \Th) 
  \in L^2_{loc}([0,\infty); H^1) 
  \cap C([0,\infty); H)
\end{align*}
such that $U(0) = U_0$ and $\tilde{U} = ( \Vort, \Th - \sigma \cdot V)$
solves \eqref{eq:sb:1}-\eqref{eq:sb:2} in the usual weak sense.
Moreover, 
\begin{itemize}
\item[(i)] The mapping $\phi: [0, \infty) \times H \times
  \Omega\rightarrow H$ defined by $\phi_t(U_0, V)= (\Vort(t,U_0,
  \sigma \cdot V), \Th(t, U_0 , \sigma \cdot V))$ is a continuous
  adapted cocycle in the sense of Definition~\ref{def:co-cyc} which
  moreover satisfies each of the conditions in Assumption~\ref{ass:Mal:der}.  
\item[(ii)] If $U_0 \in H^m$ for some $m \geq 1$, the corresponding solution $U$ maintains
  the regularity
  \begin{align*}
      U \in  L^2_{loc}([0,\infty); H^{m+1}) 
                                         \cap C([0,\infty); H^{m}).
  \end{align*}
  \item[(iii)] For $m\geq 0$ let $X_0(m)$ denote the subspace of $H^m$ given by 
\begin{align}
  X_0(m)= \text{\emph{span}}\left\{  
  \begin{pmatrix} 0 \\  \sin(k \cdot x) \end{pmatrix},
  \begin{pmatrix} 0 \\  \cos(k \cdot x) \end{pmatrix}
  \, : \, k \in \mathcal{Z} \right\}. 
  \label{eq:init:sp}
\end{align}
  Then the mapping $\Phi:[0, \infty) \times H^m \times X_0(m)\rightarrow H^m$ defined by 
\begin{align*}
\Phi_t^h U_0= (\tom(t,U_0, V_h), \tth(t, U_0 , V_h)) + (0, \sigma \cdot V_h),
\end{align*}
where $U_0= (\Vort_0, \theta_0) \in H^m$ and $V_h(t)= t h$, is a one-parameter family of continuous (global) semigroups on $H^m$
in the sense of Definition~\ref{def:oneparamlocal}. 
\end{itemize}
\end{proposition}    

Proposition~\ref{prop:b:structure} is proved in Appendix~\ref{sec:a priori:est} using standard a priori bounds.

\subsubsection*{Statement of the main results}  

Our main goal in this section is to prove the following control result concerning \eqref{eq:b:1}-\eqref{eq:b:2}.  

\begin{theorem}\label{thm:cntrl:BE}
  Suppose that $\Zb \supseteq \{(1,0), (0,1)\}$.  Then we have the following controllability results
  (cf. Definitions~\ref{def:acc:cont}, \ref{def:control:cyc} above):
  \begin{itemize}
  \item[(i)] For any continuous, linear projection $\pi: H \rightarrow H$ onto a finite dimensional subspace $\pi(H)$, $\phi$ 
  is approximately controllable on $H$ and exactly controllable on $\pi(H)$.    
  \item[(ii)] Let $m \geq 0$ and $\mathfrak{F}(m)$ denote the one-parameter family of continuous (global) semigroups defined by 
  \begin{align}
      \mathfrak{F}(m):=\{(\Phi, X_0(m)) \}         
            \label{eq:para:SE:BE}
    \end{align}
    where $X_0(m)$ was defined in~\eqref{eq:init:sp}.
    Then for any continuous, linear projection $\pi:H^m\rightarrow H^m$ onto a finite-dimensional subspace 
    $\pi(H^m)$, $\mathbb{D}(\mathfrak{F}(m))$ is approximately controllable on $H^m$ and exactly controllable on $\pi(H^m)$.
  \end{itemize}
\end{theorem}

\begin{remark}
  \label{rmk:mo:dirs:buca}
  We make the assumption that $\Zb\supseteq \{(1,0), (0,1)\}$ for
  concreteness and simplicity of presentation.  Similar to the the low-mode control problem 
  for the Navier-Stokes equation presented in the previous example, 
  this assumption can be replaced with a general algebraic condition that $\mathcal{Z}$ contains elements that 
  generate $\Z^2$ with the appropriate integer linear combinations.  See Remark~5.3 in \cite{FoldesGlattHoltzRichardsThomann2013} and the accompanying diagrams for a further discussion of this point.
\end{remark}

Invoking the results in Section~\ref{sec:app:SPDEs}, we obtain the following corollary as a 
simple consequence of the previous control results and Proposition~\ref{prop:b:structure}.            

\begin{corollary}
\label{cor:b:pos}
Let $P_t$ denote the Markov transition kernel
associated to the cocycle $\phi_t(U_0, W)$ defined by \eqref{eq:b:1}--\eqref{eq:b:2} 
via Proposition~\ref{prop:b:structure}.  Then we have the following:
\begin{itemize}
\item[(1)]  For all $t>0$ and $U_0 \in H$, $\text{\emph{supp}}(P_t(U_0, \, \cdot \,))  = H$.  
\item[(2)]  There exists a unique invariant measure $\mu$ for $P_t$ and this measure has full support, i.e., $\text{\emph{supp}}(\mu) = H$. 
\item[(3)]  Suppose that $\pi: H\rightarrow H$ is a continuous, linear
  projection onto a finite-dimensional subspace $\pi(H)$ and let
  $t>0$, $U_0 \in H$.  Then the random variable $\pi \phi_t(U_0, W)$
  has a density $p_t$ with respect to Lebesgue measure on $\pi(H)$
  which is strictly positive almost everywhere.  
\end{itemize}
\end{corollary}

\begin{remark}
Although we establish the stronger control result above on the spaces
$H^m$, we  remain in the $L^2$ phase space to deduce properties of
random variables $\phi_t(U_0, W)$, $U_0 \in L^2$ and $t>0$, where $W$ 
is a standard two-sided, $2|\mathcal{Z}|$-dimensional Brownian motion 
defined on the Wiener space $(\Omega, \PP)$.  This is allows us to
connect seamlessly with the results
in~\cite{FoldesGlattHoltzRichardsThomann2013} 
concerning the spectral properties of the Malliavin covariance matrix
$M_t(U_0, W)$ corresponding to $\phi_t(U_0, W)$.  
\end{remark}

\begin{proof}[Proof of Corollary~\ref{cor:b:pos}]
Conclusion (1) of the result follows immediately by combining
Theorem~\ref{thm:cntrl:BE} and Proposition~\ref{prop:b:structure} with
Lemma~\ref{l:control}.   Regarding the second item (2), the existence of 
an invariant measure $\mu$ is established with standard energy estimates
and the Krylov-Bogoliubov averaging procedure.  For the question of 
the uniqueness of $\mu$, we rely 
on Corollary~\ref{cor:uniq:criteria}. 
Here the asymptotic strong Feller condition \eqref{eq:ASF}
follows precisely as in \cite[Proposition 2.6]{FoldesGlattHoltzRichardsThomann2013}.\footnote{Strictly speaking,
\cite{FoldesGlattHoltzRichardsThomann2013} establishes~\eqref{eq:ASF}
without the inhomogeneous term $h^0$.  However, this additional term does not 
introduce further complications for establishing the non-degeneracy condition.}
The second requirement of Corollary~\ref{cor:uniq:criteria}, the approximate controllability condition, 
is precisely the content of Theorem~\ref{thm:cntrl:BE}.

Finally to establish conclusion (3), we combine 
Theorem~\ref{thm:cntrl:BE} and Proposition~\ref{prop:b:structure} with Theorems~
\ref{prop:exod}, \ref{thm:pos} from Section~\ref{sec:app:SPDEs}.
Fixing $t>0$ and $U_0 \in L^2$ and applying
Theorem~4.1 of~\cite{FoldesGlattHoltzRichardsThomann2013} with
Remark~\ref{rem:malliavin:spec} of this paper, we find that the
Malliavin matrix $M_s(U_0, W)$ associated with (\ref{eq:b:1})-(\ref{eq:b:2}) 
is non-degenerate for any $0<s\leq t$.   The fact that the random variable
$\pi \phi_t(U_0, W)$ is absolutely continuous with respect to Lebesgue measure follows 
by combining Theorem~\ref{thm:cntrl:BE}, Proposition~\ref{prop:b:structure} and 
Theorem~\ref{prop:exod}.  Regarding the claim that the associated density is almost surely
positive, we fix any $0<s<t$ so that $M_s(U_0, W)$ is non-degenerate.  This then implies the
existence of a fixed deterministic path $V\in \Omega$ such that the Gramian matrix
$M_s(U_0, V)$ is non-degenerate and Theorem~\ref{thm:pos} applies.  The proof is now complete.

\end{proof}

The remainder of this section is devoted to establishing Theorem~\ref{thm:cntrl:BE}.  Before diving
into this proof, we introduce some further notation which 
eases the presentation below and allows us to connect to the setup
presented in \cite{FoldesGlattHoltzRichardsThomann2013}.

\subsubsection*{Notation}  
For $U=(\Vort, \Th) \in \R^2$, we define mappings
$\pi_\Vort, \pi_\Th: \R^2 \rightarrow \R $ by $\pi_\Vort U = \Vort$ and $\pi_\Th U = \Th$. 
On the other hand for $\alpha \in \R$, we let $\iota_\Vort, \iota_\Th : \R\rightarrow \R^2$ be given by 
\begin{align*}
\iota_\Vort \alpha = \begin{pmatrix} \alpha \\
0 \end{pmatrix}\,\,\text{ and } \,\, \iota_\Th \alpha = \begin{pmatrix} 0 \\ \alpha
\end{pmatrix}.   
\end{align*}
For $U=(\Vort, \Th), \, \tilde{U}= (\tilde{\Vort}, \tilde{\Th}):\TT^2\rightarrow \R^2$ sufficiently smooth, we define the following operators
\begin{align}
  &A U := -\nu \iota_\Vort \Delta \Vort - \kappa \iota_\Th \Delta \Th,
  \notag \\ 
  &B(U,\tilde{U}) := \iota_\Vort  [  (\mathcal{K} \ast \Vort) \cdot \nabla \tilde{\Vort} ] 
      + \iota_\Th[ (\mathcal{K} \ast \Vort) \cdot \nabla \tilde{\Th}],
    \label{eq:buca:n:lin:abs}\\
   &GU = -g \iota_\Vort \partial_x \Th,  
  \notag
\end{align}
where recall that $\mathcal{K}$ denotes the Biot-Savart kernel and $\nu, \kappa, g > 0$
are the positive constants defined in \eqref{eq:b:1}--\eqref{eq:b:2}.  We let  
\begin{align}
  F(U) = AU + B(U, U) + G U.   
  \label{eq:abs:LHS:BE}
\end{align}
For the basis elements, we write 
\begin{align}
  \sigma_k^0(x) =
   \iota_\Th \cos(k\cdot x)  : =
   	\iota_\Th e_k^0( x),
   \quad
   \sigma_k^1(x) = \iota_\Th \sin(k \cdot x)    :=
   	\iota_\Th e_k^1( x),  
  \label{eq:BE:basis:1}
\end{align}
and
\begin{align}
  \psi_k^0(x) =
 \iota_\Vort   \cos(k\cdot x)
       :=
   	\iota_\Vort e_k^0( x), 
     \quad
  \psi_k^1(x) = 
\iota_\Vort  \sin(k \cdot x)
  := \iota_\Vort e_k^1( x),
    \label{eq:BE:basis:2}
\end{align}
where $k \in \Z^2$ and $x \in \TT^2$.   For $N \geq 1$, we take 
\begin{align}
  H_N = \{\psi_k^j, \sigma_k^j : |k| \leq N, j \in \{0,1\}\}
  \label{eq:proj:sp:buca}
\end{align}
and take $P_N$ to be the projection onto this subspace of $H^m$ for $m \geq 0$.
Combining these notations, observe that we may rewrite \eqref{eq:b:1}--\eqref{eq:b:2} in an abbreviated fashion as
\begin{align}
  \frac{d}{dt} U + F(U) = \iota_\Th h^0+ \iota_\Th  (\sigma \cdot \partial_t V), \quad U(0) = U_0.
  \label{eq:be:abs}
\end{align}

\subsubsection*{Proof of the main control result}

With these preliminaries in hand we now prove Theorem \ref{thm:cntrl:BE}. 
As in the reaction-diffusion and 2D Navier-Stokes examples, Theorem~\ref{thm:cntrl:BE}
is established via a suitable sequence of scaling lemmata.  In this example, 
however, the path taken to produce new directions inductively using the nonlinearity is 
different than the one taken for the reaction-diffusion equation,
hence different than the one discussed in the hueristics section (Section~\ref{sec:overview}).  

We proceed by first stating the scaling lemmata without proof. 
We then combine them and leverage Corollary~\ref{cor:exact:con}  
to prove Theorem~\ref{thm:cntrl:BE}. 
The section concludes by proving each of the scaling estimates
based on energy bounds and commutator estimates.  

First, we state the pure noise scaling estimate which 
starts the inductive generation of controllable directions       
\begin{lemma}\label{lem:0:conv}
Fix $m \geq 0$, $t>0$ and suppose that $K_1\subseteq H^{m}$
and $K_2 \subseteq X_0(m)$ are compact sets where recall that $X_0(m)$ is
defined in \eqref{eq:init:sp}.  Then  
\begin{align}
	\lim_{\lambda \to \infty} \sup_{U_0 \in K_1, h \in K_2}  \|  \SGf^{\lambda h}_{t/\lambda} U_0 - \rho_t^{h} U_0\|_{H^m} = 0.
  \label{eq:buca:first:sc}
\end{align}
where, as usual $\rho$ denotes the ray semigroup \eqref{eq:ray:sg}.
Consequently, we have that $(\rho, X_0(m)) \in \text{\emph{Sat}}_u(\mathfrak{F}(m)))$
with $\mathfrak{F}(m)$ defined as in \eqref{eq:para:SE:BE}.
\end{lemma}

For the next scaling estimate, fixing $\alpha \in \R$ we introduce the following dynamics defined by the equation
\begin{align}
\label{eq:bous:fna}
\partial_t U = \alpha\begin{pmatrix}
 g \partial_x e_j^\ell \\
   \kappa \Delta e_j^\ell  - b(\pi_\Vort U, e_j^\ell)  
  \end{pmatrix} = -A \alpha \sigma_j^\ell - G \alpha\sigma_j^\ell - B(\alpha U, \sigma_j^\ell), \qquad U(0)= U_0,
  \end{align}
  where
  \begin{align}
  b(\Vort, \Th) = (\mathcal{K} \ast \Vort)\cdot \nabla \Th,
    \label{eq:buca:ad:not}
  \end{align}
  with $\ell \in \{ 0,1\}$, $j \in \Z^2_+$ and the elements $e_j^\ell$, $\sigma_j^\ell$
  are the sinusoidal directions defined in \eqref{eq:BE:basis:1}.
  One can readily check that for any $U_0 \in
  H^m$, equation~\eqref{eq:bous:fna} 
  has a unique global solution belonging to $H^m$.  Furthermore, using
  cancelations like~\eqref{eq:rel:deg:NSE}
  one infers that the solution of \eqref{eq:bous:fna} is explicitly given by 
\begin{align}
    \Gamma_t^{\alpha \sigma_j^\ell}U_0 :=  U(t) = U_0 + t  \alpha \begin{pmatrix}
 g \partial_x e_j^\ell\\
   \kappa \Delta e_j^\ell - b(\pi_\Vort U_0, e_j^\ell)  
  \end{pmatrix}.
  \label{eq:l:evol:bc:brak}
\end{align}
Thus, for each $\ell \in \{ 0,1\}$, $j \in \Z^2_+$ 
and $m \geq 0$, $(\Gamma, \{ \alpha \sigma_j^\ell: \alpha \in \RR\})$ defines a
one-parameter family of global semigroups on $H^m$ according
to \eqref{eq:l:evol:bc:brak}.

\begin{lemma}
\label{lem:sa:bouseq}
Let $m \geq 0$, $t > 0$ and fix $j\in \Z_+^2$, $\ell \in \{0,1\}$.  Also, let $K_1\subseteq H^m$ and 
$K_2 \subseteq \{\alpha \sigma_j^\ell \, : \, \alpha \in \R \}$ be compact.  Then 
\begin{align}
	\lim_{\lambda \to \infty} \sup_{U_0 \in K_1, \phi \in K_2}\| \rho^{-\lambda^2  \iota_\Th \phi}_{\lambda^{-1}}\, \Phi^{0}_{t/\lambda} \, \rho^{\lambda^2 \iota_\Th \phi}_{\lambda^{-1}}U_0 
	- \Gamma_t^{\phi} U_0\|_{H^m} =  0.
  \label{eq:m:buca:brak:estJ}
\end{align}
\end{lemma}
\begin{remark}
  Observe that although $F$ defined by \eqref{eq:abs:LHS:BE} is a second-degree polynomial, 
  the second-order terms that appear in the expansion governing 
  $W_\lambda(t) = \rho^{-\lambda^2 \iota_\Th \phi}_{\lambda^{-1}} \SGf^{0}_{t/\lambda}  \rho^{\lambda^2 \iota_\Th \phi}_{\lambda^{-1}}U_0$ 
  are zero since $B(\sigma^l_k, \sigma^l_k) = 0$ for any $l \in \{0,1\}$, $k \in \Z^2_+$.  See 
  \eqref{eq:be:sc:2:evol} below and recall (\ref{eq:rs:nLin:twit}) in the heuristics section above.
\end{remark}

Our next scaling `estimate' is somewhat surprising given that it produces an exact formula.   See
\cite[Lemma~5.1]{FoldesGlattHoltzRichardsThomann2013} and the surrounding computations.  Note that 
particular composition of $\Gamma$'s below is motivated by the definition 
of the Lie bracket between two vector fields, for it immediately follows from \eqref{eq:l:evol:bc:brak} 
that $\Gamma^{-\alpha \sigma_k^n}_t= \Gamma_{-t}^{\alpha \sigma_k^n}$.       

\begin{lemma}
\label{lem:BE:exact}
Let $m \geq 0$.  Then for any $U_0 \in H^m$, $k, j \in \Z^2_+$, $\ell, n \in \{0,1\}$ and 
any $\alpha, \beta \in \RR$,
\begin{align}
\Gamma_t^{-\alpha \sigma_k^n} \Gamma_t^{-\beta \sigma_j^\ell}\Gamma_t^{\alpha \sigma_k^n}\Gamma_t^{\beta \sigma_j^\ell} U_0
= U_0 + t^2 \alpha \beta [[F, \sigma_j^\ell] , [F, \sigma_k^n]],
\end{align}
for every $t \geq 0$ where
\begin{align}
  [[F, \sigma_j^\ell] , [F, \sigma_k^n]] =  
  g \begin{pmatrix} 
    0 \\
    b(\partial_x e_k^n, e_j^\ell) - b(\partial_x e_j^\ell, e_k^n) 
  \end{pmatrix}.  
  \label{eq:new:Th:dir}
\end{align}
\end{lemma}

We recall that 
\begin{align}
  [G_1, G_2] = D G_2 G_1 - D G_1 G_2
  \label{eq:Lie:the:Brak}
\end{align}
is the \emph{Lie Bracket} of $C^1$ vector fields $G_i: H^m \to H^m$.
The computation yielding \eqref{eq:new:Th:dir} is detailed in 
\cite[Section~5.1]{FoldesGlattHoltzRichardsThomann2013}.

Lastly, we note the following scaling result, whose proof we omit as
it is nearly identical to the proof of Lemma~\ref{lem:sa:bouseq}.
This will allow us to generate all nontrivial rays in basis vorticity
directions.

\begin{lemma}
\label{lem:BE:fourthapprox}
Fix $m \geq 0$, $j, k \in \Z^2_+$ and $\ell,n \in \{ 0, 1\}$
and $t> 0$. Let $K_1\subseteq H^m$ be compact and $0 < N < \infty$.  
Then we have
\begin{align*}
  \lim_{\lambda \to \infty} \sup_{U_0 \in K_1, |\alpha|+|\beta|\leq N }
  \| \rho^{- \lambda^2(\alpha \psi_j^l + \beta \psi_k^n)}_{\lambda^{-1}} \Phi_{t/\lambda^2}^0 \rho^{\lambda^2 (\alpha \psi_j^l + \beta \psi_k^n)} _{\lambda^{-1}}
  U_0-  \rho^{ \alpha \beta  [[F, \psi_k^\ell], \psi_j^n]}_t U_0 \|_{H^m} = 0
 \end{align*}
 where, cf. \eqref{eq:Lie:the:Brak},
 \begin{align*}
 [[F, \psi_k^\ell], \psi_j^n]: 
   = \iota_\Vort [(\mathcal{K} * e_k^\ell) \cdot \nabla e_j^n + (\mathcal{K}* e_j^n) \cdot \nabla e_k^\ell]
  \end{align*} 
with $F$ defined according to \eqref{eq:abs:LHS:BE} and the elements $\psi$ and $e$ are defined by \eqref{eq:BE:basis:2}.
\end{lemma}

With the scaling Lemmas~\ref{lem:0:conv},~\ref{lem:sa:bouseq},~\ref{lem:BE:exact} and \ref{lem:BE:fourthapprox}
in hand we now proceed to combine these bounds to prove Theorem~\ref{thm:cntrl:BE}

\begin{proof}[Proof of Theorem~\ref{thm:cntrl:BE}]
  Once again, we show that the conditions 
  of Corollary~\ref{cor:exact:con} apply for a suitable sequence of
  subspaces $X_n$.  Let $m \geq 0$ and 
  observe that Lemma~\ref{lem:0:conv} implies that
  $(\rho, X_0(m)) \in \text{Sat}_u(\mathfrak{F}(m))$.  
  Let $Y_0^{\Th}(m)=X_0(m)$.  For $n\geq 1$ 
  we iteratively define 
  \begin{align*}
    Y^{\Th}_{n}(m)
    := \mbox{span} \left\{ Y_{n-1}^{\Th}(m) 
         \cup \{ [[F, \phi],[F, \tilde{\phi}]]\, : \, \phi, \tilde{\phi} \in
                           Y_{n-1}^{\Th}(m)\}  \right\}.  
  \end{align*}
  Invoking Proposition 5.2 from \cite{FoldesGlattHoltzRichardsThomann2013}, one may show
  inductively that, for each $n \geq 1$, $Y^{\Th}_n$ consists of spans of elements of the form $\sigma_j^\ell$ defined according to \eqref{eq:BE:basis:1}.  In particular, $Y_n^\Th(m)\subseteq H^m$.
  Combining Lemma~\ref{lem:sa:bouseq} with Lemma~\ref{lem:BE:exact}  
  and Remark~\ref{rmk:stupid:spans}, we conclude that
  \begin{align}
    (\rho, Y_n^\Th(m)) \in \text{Sat}_u(\mathfrak{F}(m)),
    \label{eq:ds:dirs}
  \end{align}
  for all $n\geq 1$.

  Now, according Remark 5.3 and Lemma 6.10 of~\cite{FoldesGlattHoltzRichardsThomann2013}, 
  observe that 
  \begin{align}
      \mbox{span}\{ \sigma_j^\ell : j \in \Z^2_+, \ell \in \{0,1\} \} \subseteq  \bigcup_{n\geq 0} Y_n^\Th(m).
    \label{eq:ds:dirs:span}
  \end{align}
  Thus, due to Lemma~\ref{lem:sa:bouseq}, we have that, cf. \eqref{eq:l:evol:bc:brak},
  \begin{align}
    (\Gamma, \{ \alpha \sigma_j^\ell: \alpha \in \RR\}) \in \text{Sat}_u(\mathfrak{F}(m)) 
        \text{ for every } j\in \Z_+^2, \ell \in \{0,1\}.    
    \label{eq:us:pl:junk:dirs}
  \end{align}
  Now since $\{ e_j^\ell = \pi_\Th \sigma_j^\ell : j \in \Z^2_+, \ell \in \{0,1\} \}$ 
  is a basis for $H^m(\TT^2)$ for any $m \geq 0$, we combine (\ref{eq:ds:dirs}) 
  and (\ref{eq:us:pl:junk:dirs}) with (\ref{eq:ds:dirs:span}) to deduce
  \begin{align}
     (\rho, \alpha \psi_j^\ell) \in \text{Sat}_u(\mathfrak{F}(m)) 
    \text{ whenever } j =(j_1, j_2) \in \Z^2_+, j_1 \not = 0, \ell \in \{0,1\}.
        \label{eq:us:mn:y:axis}
  \end{align}
  See Definition~\ref{def:uni:sat} above.  Note carefully that, due to
  the presence of the $\partial_x$ in~\eqref{eq:l:evol:bc:brak},  the 
  ray semigroups in \eqref{eq:us:mn:y:axis} omit the directions $\psi_j^\ell$
  along the $y$-axis where $j_1 = 0$.  To recover these missing directions, we 
  invoke Lemma~\ref{lem:BE:fourthapprox} and elementary trigonometric
  identities as in, e.g., \cite{HM06}.  Combining this observation with \eqref{eq:ds:dirs} and \eqref{eq:us:mn:y:axis} and invoking 
  Remark~\ref{rmk:stupid:spans}, we finally conclude that
  $(\rho, X_n(m)) \in \text{Sat}_u(\mathfrak{F}(m))$
  where
  \begin{align*}
  X_n(m) = 
    \span\left\{ \{ \psi_j^\ell : j \in \Z^2_+, |j| \leq n,  \ell \in \{0,1\}\} 
    \cup \{ \sigma_j^\ell : j \in \Z^2_+, |j| \leq n, \ell \in \{0,1\}\} \right\}.
  \end{align*}
  Since $\cup_{n \geq 1} X_n(m)$ is a dense subset of $H^m$, we now infer
  Theorem~\ref{thm:cntrl:BE} from Corollary~\ref{cor:exact:con}, 
  thus concluding the proof.
\end{proof}

We now turn to proving each scaling estimate.  

\begin{proof}[Proof of Lemma~\ref{lem:0:conv}] We begin by introducing the following shorthand notation 
\begin{align*}
  V_\lambda(t)= (\Bvl(t), \Btl(t))  = \SGf^{\lambda h}_{t/\lambda} U_0 - \rho_t^{\iota_\Th h} \tilde{U}_0\,\, 
  \text{ and } \,\, 
  \rho(t)= (\rhV(t), \rhT(t)) = \rho_t^{\iota_\Th h} \tilde{U}_0 .    
\end{align*}
Here choice of the initial condition $\tilde{U}_0$ is made precise
below.

Arguing as in \eqref{eq:rescal:0} and \eqref{eq:diff:ray:RD:rscl}, we obtain the following system
for $V_\lambda$:
\begin{align}
  &\pd_t \Bvl + \frac{1}{\lambda}[(\mathcal
  {K} \ast(\Bvl + \rhV)  \cdot \nabla (\Bvl + \rhV) - \nu \Delta (\Bvl + \rhV)]
    =  \frac{1}{\lambda}g \pd_x (\Btl + \rhT), 
    \label{eq:scaled:ge:L:vort}\\
  &\pd_t \Btl + \frac{1}{\lambda}[(\mathcal{K} \ast (\Bvl + \rhV)) \cdot \nabla (\Btl + \rhT) - \kappa \Delta (\Btl + \rhT)] 
    =  \frac{1}{\lambda} h^0  .  
  \label{eq:scaled:ge:L:th}
\end{align}
We begin by establishing \eqref{lem:0:conv} in the $L^2$ topology. 
Observe that 
\begin{align}
  \frac{1}{2} \frac{d}{dt}( \| \Bvl\|^2 + \|\Btl\|^2) 
     +& \frac{1}{\lambda}( \nu \| \nabla \Bvl\|^2 + \kappa \|\nabla \Btl\|^2)
  \notag\\
     &=  -\frac{1}{\lambda}\bigl(
       \langle (\mathcal{K} \ast(\Bvl + \rhV)  \cdot \nabla \rhV, \Bvl  \rangle
       +       
       \langle (\mathcal{K} \ast(\Bvl + \rhV)  \cdot \nabla \rhT, \Btl  \rangle
  \notag\\
       & \qquad \qquad +
       \nu \langle \nabla \Bvl ,\nabla \rhV\rangle
       +
       \kappa \langle \nabla \Btl , \nabla \rhT \rangle
       -
       \langle g \partial_x(\Btl + \rhT), \Bvl \rangle
       -
       \langle h^0, \Btl\rangle
       \bigr)
  \notag\\
     &= -\frac{1}{\lambda}(T_1^0 + T_2^0 + T_3^0 + T_4^0 + T_5^0 + T_6^0) .  
    \label{eq:b:sc:L2:evol}
\end{align}
With Agmon's inequality and the smoothing properties of the Biot-Savart kernel, we obtain
\begin{align}
  |T_1^0 + T_2^0| &\leq C \|\mathcal{K} \ast(\Bvl + \rhV)\|_{L^\infty} \|\rho\|_{H^1} \|V_\lambda\|
                 \leq C (\|\rho\|_{H^1}^2\|V_\lambda\| +\|\rho\|_{H^1}\|V_\lambda\| \|V_\lambda\|_{H^1})\\
                 &\leq C \|\rho\|_{H^1}^2(\|V_\lambda\|^2 + 1) + \frac{\nu}{2}\|\nabla  \Bvl \|^2 + \frac{\kappa}{2} \|\nabla \Btl\|^2 
                   \label{eq:b:sc:L2:1}
\end{align}
where $C$ does not depend on $\lambda$.  For the remaining terms we simply estimate
\begin{align}
  |T_3^0 + T_4^0 +T_5^0 +T_6^0| \leq 
      C(\|V_\lambda\|^2 +  \|\rho\|_{H^1}^2   +  \| h^0\|^2)  + \frac{\nu}{2}\|  \nabla \Bvl \|^2 + \frac{\kappa}{2} \| \nabla \Btl\|^2.
                   \label{eq:b:sc:L2:1}
\end{align}
Here again $C$ is independent of $\lambda$.  Combining the preceding two bounds with \eqref{eq:b:sc:L2:evol}
yields
\begin{align*}
  \frac{d}{dt} \|V_\lambda\|^2 
  \leq  \frac{C}{\lambda}[(\|\rho  \|^2_{H^1} + 1)\|V_\lambda\|^2 
                     + \|\rho\|_{H^1}^2   +  \| h^0\|^2].
\end{align*}
Hence, from Gr\"onwall's inequality and recalling the definition of $\rho$
we infer
\begin{align}
  \|V_\lambda(t) \|^2 
      &\leq \exp\left( C (\|h\|_{H^1}^2 + \|\tilde{U}_0\|^2 + 1) \right) 
  \left( \| U_0 - \tilde{U}_0 \|^2 + \frac{ 1 + \|h^0\|^2}{\lambda}  \right)
        \label{eq:gron:d:wall}
\end{align}
where note carefully that $C$ may depend on $t$ (and other universal quantities)
but is independent of  $\lambda$.  Arguing in the same fashion as in (\ref{eq:cmp:sets:bnd})
where $\tilde{U}_0$ is taken to be suitable Fourier truncation of $U_0$, now yields 
the desired bound \eqref{eq:buca:first:sc} for $m = 0$.

We next turn to estimates in $H^m$ for $m \geq 1$.  Here we have
\begin{align*}
  \frac{1}{2} \frac{d}{dt}( \| \Bvl\|^2_{H^m}& + \|\Btl\|^2_{H^m}) 
     + \frac{1}{\lambda}( \nu \| \Bvl\|^2_{H^{m+1}} + \kappa \|\Btl\|^2_{H^{m+1}})\\
     &=  -\frac{1}{\lambda}\bigg(
       \sum_{|\beta| \leq m}  \langle (\partial^\beta (\mathcal{K} \ast(\Bvl + \rhV)  \cdot \nabla (\rhV+ \Bvl)) 
                                  - (\mathcal{K} \ast(\Bvl + \rhV)  \cdot \nabla \partial^\beta \Bvl), \partial^\beta \Bvl  \rangle\\
       &\qquad +       
       \sum_{|\beta| \leq m} \langle \partial^\beta(\mathcal{K} \ast(\Bvl + \rhV)  \cdot \nabla (\rhT+ \Btl)) 
                                 - (\mathcal{K} \ast(\Bvl + \rhV)  \cdot \nabla \partial^\beta \Btl) ,\partial^\beta \Btl  \rangle\\
       &\qquad +
       \nu \langle  \Bvl , \rhV\rangle_{H^{m+1}}
       +
       \kappa \langle \Btl ,  \rhT \rangle_{H^{m+1}}
       -
       \langle g \partial_x(\Btl + \rhT), \Bvl \rangle_{H^{m}}
       -
       \langle h^0, \Btl\rangle_{H^{m}}
       \bigg)\\
     &= -\frac{1}{\lambda}(T_1 + T_2 + T_3 + T_4 + T_5 + T_6).
\end{align*}
Regarding the first two terms, Sobolev embedding, interpolation and the one-degree smoothing of the
Biot-Savart kernel implies
\begin{align}
  |T_1 + T_2| &\leq 
                        C(\|\mathcal{K} \ast(\Bvl + \rhV) \|_{W^{m,4}} \|V_\lambda \|_{W^{m,4}} \| V_\lambda \|_{H^m} 
                                +  \|\mathcal{K} \ast(\Bvl + \rhV) \|_{W^{m,4}} \|\rho \|_{W^{m+1,4}} \| V_{\lambda} \|_{H^m})
  \notag\\
                 &\leq C(\|\Bvl + \rhV \|_{H^m} \|V_\lambda \|^{1/2}_{H^{m+1}}  \| V_\lambda \|_{H^m}^{3/2}
                                +  \|\Bvl + \rhV\|_{H^m} \|\rho \|_{H^{m+2}}\| V_{\lambda} \|_{H^m}) 
  \notag\\
                & \leq 
                       \frac{\nu}{2} \| \Bvl \|^2_{H^{m+1}}  + \frac{\kappa}{2}\| \Btl \|^2_{H^{m+1}} +
                       C(\| V_\lambda \|_{H^m}^{6} + \|\rho\|_{H^{m+2}}^4 +1).  
                       \label{eq:crude:bnd}
\end{align}
Regarding the remaining terms we have
\begin{align*}
  |T_3 + T_4 +T_5 +T_6| \leq 
      C(\|V_\lambda\|^2_{H^m} +  \|\rho\|_{H^{m+1}}^2   +  \| h^0\|^2_{H^{m}})  + \frac{\nu}{2}\| \Bvl \|^2_{H^{m+1}} + \frac{\kappa}{2} \| \Btl\|^2_{H^{m+1}}.
\end{align*}
Fixing $T>0$ and combining these inequalities we find
\begin{align}
  \frac{d}{dt} \|V_\lambda\|^2_{H^m} 
  \leq   \frac{C}{\lambda} \left( \| V_\lambda \|_{H^m}^{6} 
                   + \|\tilde{U}_0\|_{H^{m+2}}^4 + T^4\|h\|_{H^{m+2}}^4   +  \| h^0\|^2_{H^{m}} +1 \right),
  \label{eq:b:sc:1:h:f}
\end{align}
for all $t \in [0, T]$, where $C$ is independent of $\lambda > 0$.  Invoking Lemma~\ref{lem:comp:lem:app},
we obtain the bound
\begin{align}
  \|V_\lambda(t)\|^2_{H^m} \leq&
     \| U_0 - \tilde{U}_0\|^2_{H^m} 
     R_\lambda(t, \| U_0 - \tilde{U}_0\|^2_{H^m}  + \|\tilde{U}_0\|_{H^{m+2}}^4 + T^4\|h\|_{H^{m+2}}^4   +  \| h^0\|^2_{H^{m}} +1) 
  \notag\\
     &+ (\|\tilde{U}_0\|_{H^{m+2}}^4 + T^4\|h\|_{H^{m+2}}^4   +  \| h^0\|^2_{H^{m}} +1) 
  \notag\\
     &\quad \times R_\lambda(t, \| U_0 - \tilde{U}_0\|^2_{H^m}  +\|\tilde{U}_0\|_{H^{m+2}}^4 + T^4 \|h\|_{H^{m+2}}^4   +  \| h^0\|^2_{H^{m}} +1) -1)
       \label{eq:g:gw:con}
\end{align}
for every  $t\in [0, T_\lambda^* \wedge T)$ where we recall that $R_{\lambda}$ is defined in \eqref{eq:blow:cntrl}.  

We now complete the proof by using \eqref{eq:g:gw:con} in similar fashion to \eqref{eq:scale:pt:bound},
\eqref{eq:cmp:sets:bnd} above to infer (\ref{eq:buca:first:sc}).  Fix any $\epsilon >0$
and cover $K_1$ with a finite number of $\epsilon$ balls $B(\epsilon, V_0^j)$, $j = 1, \ldots, M$.
For $N > 0$, define
\begin{align*}
  \gimel_N :=  \sup_{U_0 \in K_1, h \in K_2} ( 4 \| U_0\|^2_{H^m} + \|P_N U_0\|_{H^{m+2}}^4 + T^4 \|h\|_{H^{m+2}}^4   +  \| h^0\|^2_{H^{m}} +1)
\end{align*}
where recall that $P_N$ is the projection onto $H_N$ as in (\ref{eq:proj:sp:buca}).  Note that this quantity is finite 
for every $N$ in view of the standing assumptions on the compact sets $K_1$, $K_2$.
Noting the monotonicity of $R_{\lambda}$ and invoking \eqref{eq:g:gw:con} we find
\begin{align*}
 \sup_{ U_0 \in K_1, h \in K_2} &\|  \SGf^{\lambda h}_{t/\lambda} U_0 - \rho_t^{h} U_0\|_{H^m}^2\\
  \leq& 4 \sup_{ U_0 \in K_1, h \in K_2} \|  \SGf^{\lambda h}_{t/\lambda} U_0 - \rho_t^{h} P_N U_0\|_{H^m}^2
        +4 \sup_{ U_0 \in K_1} \|  P_N U_0 - U_0 \|_{H^m}^2\\
  \leq& 16\epsilon +  16 \max_{j = 1, \ldots, M} \|  P_N V_0^j - V_0^j \|_{H^m}^2 (R_{\lambda}(t, \gimel_N) +1)
              +\gimel_N (R_{\lambda}(t, \gimel_N) -1)
\end{align*}
which holds on the interval $t \in [0, T^*_\lambda(\gimel_N)\wedge T)$.  Picking $N$ large enough that 
$ \max_{j = 1, \ldots, M}\|  P_N V_0^j - V_0^j \|_{H^m}^2 \leq \epsilon$ and noting that for any
fixed $N$, $\lim_{\lambda \to \infty} T^*_\lambda(\gimel_N) = \infty$ and $\lim_{\lambda \to \infty} R_{\lambda}(t, \gimel_N) = 1$
we obtain that 
\begin{align*}
 \limsup_{\lambda \to \infty} \sup_{ U_0 \in K_1, h \in K_2} &\|  \SGf^{\lambda h}_{t/\lambda} U_0 - \rho_t^{h} U_0\|_{H^m}^2 \leq 48 \epsilon.
\end{align*}
for all $t\in [0, T]$.  Since this holds for any $\epsilon >0$ (\ref{eq:buca:first:sc}) now follow for $m \geq 1$,
completing the proof.   
\end{proof}

\begin{remark}
  Note that the bound~\eqref{eq:crude:bnd} is rather crude and can be significantly sharpened 
  to improve \eqref{eq:b:sc:1:h:f} and hence the rates of convergence as 
  $\lambda \to \infty$ in Lemma~\ref{eq:l:evol:bc:brak}.  Similar remarks
  apply to the bound \eqref{eq:mo:bnds:buca:1} below and hence to rates of convergence in
  Lemma~\ref{lem:sa:bouseq}.  Since these rates have no
  immediate bearing on our results, here we omit the more refined estimates.
\end{remark}

\begin{proof}[Proof of Lemma~\ref{lem:sa:bouseq}]
Fix $U_0 \in K_1$ and $\phi \in K_2$.
Recycling some of the notation used in the proof of the previous scaling result, in this proof we let   
\begin{align*}
  W_\lambda(t) &= \rho^{-\lambda^2 \iota_\Th \phi}_{\lambda^{-1}} \SGf^{0}_{t/\lambda}  \rho^{\lambda^2 \iota_\Th \phi}_{\lambda^{-1}}U_0, \\
  \Gamma(t) &= (\gmV, \gmT) :=  \Gamma_t^\phi \tilde{U}_0,\\
  V_\lambda(t) &=  (\Bvl(t), \Btl(t)) = W_\lambda(t) - \Gamma(t).
\end{align*}
Arguing as in (\ref{eq:rs:nLin:twit}) and referring back to \eqref{eq:bous:fna}, using the extended phase space notation we obtain
\begin{align}
  \frac{d}{dt} V_\lambda 
  =&- \frac{1}{\lambda}[A( W_\lambda + \lambda \phi) 
                     + B( W_\lambda + \lambda \phi,  W_\lambda + \lambda \phi)
                     + G( W_\lambda + \lambda \phi) - \iota_\Th h^0]
  \notag\\
  &+ A \phi + G \phi + B(\Gamma, \phi)
  \notag\\
  =&- \frac{1}{\lambda}[A( V_\lambda + \Gamma) 
                     + B( V_\lambda + \Gamma,  V_\lambda + \Gamma)
                     + G( W_\lambda + \Gamma) - \iota_\Th h^0] 
     - B(V_\lambda, \phi)
     \label{eq:be:sc:2:evol}
\end{align}
where, recalling \eqref{eq:buca:n:lin:abs}, we have used that $B(\phi, W_\lambda + \lambda \phi) = 0$ as $\pi_\xi \phi =0$.  

Notice that \eqref{eq:be:sc:2:evol} is quite similar in formulation to (\ref{eq:scaled:ge:L:vort})--(\ref{eq:scaled:ge:L:th}).
Here $B(V_\lambda, \phi)$ the only `new' term.  Observe that 
\begin{align*}
  |\langle B(V_\lambda, \phi), V_\lambda \rangle| =  |\langle(\mathcal{K} \ast \Bvl) \cdot \nabla \phi , \Bvl\rangle| \leq C \|V_\lambda\|^2.
\end{align*}
Proceeding otherwise as in Lemma~\ref{lem:0:conv} with estimates analogous to (\ref{eq:b:sc:L2:1}) and (\ref{eq:b:sc:L2:1}), we obtain 
\begin{align*}
  \frac{d}{dt} \|V_\lambda \|^2 
  \leq& C\| V_\lambda\|^2 + \frac{C}{\lambda}[(\|\Gamma \|^2 + 1)\|V_\lambda\|^2 + \|\Gamma \|_{H^1}^2   +  \| h^0\|^2]
  \notag\\
  \leq& C\| V_\lambda\|^2 + \frac{C}{\lambda}[(\|\tilde{U}_0 \|^2 + 1)\|V_\lambda\|^2 + \|\tilde{U}_0 \|_{H^1}^2   +  \| h^0\|^2 + 1]
\end{align*}
where the constant $C$ may depend on the compact set $K_2$ and $t > 0$ but is crucially independent of $\lambda > 0$.  Here recall
(\ref{eq:l:evol:bc:brak}) to justify the second bound.  Taking $X(s) = e^{Cs}\|V_\lambda (s)\|^2$ we infer that 
$\frac{d}{dt} X \leq \frac{C}{\lambda}[(\|\tilde{U}_0 \|^2 + 1) X + \|\tilde{U}_0 \|_{H^1}^2   +  \| h^0\|^2 + 1]$.  We therefore 
obtain a bound very similar to (\ref{eq:gron:d:wall}) but which has a constant prefactor $e^{Ct}$ which is still independent of $\lambda$. 
The bound \eqref{eq:m:buca:brak:estJ} now follows for $m =0$ by arguing as in the proof of Lemma~\ref{lem:sc:lem:1:RD}.

Regarding the convergence in $H^m$ for $m \geq 1$, notice that,
\begin{align*}
  |\sum_{|\beta| \leq m} \langle \partial^\beta B(V_\lambda, \phi), \partial^\beta V_\lambda \rangle |
   \leq C \|V_\lambda\|_{H^m}^2.
\end{align*}
Otherwise, arguing as in \eqref{eq:crude:bnd}, we obtain
 \begin{align}
  \frac{d}{dt}\|V_\lambda\|^2_{H^m} \leq  C \|V_\lambda\|^2_{H^m} +
                   \frac{C}{\lambda} \left( \| V_\lambda \|_{H^m}^{6} 
                   + \|\tilde{U}_0\|_{H^{m+2}}^4   +  \| h^0\|^2_{H^{m}} +1 \right)
   \label{eq:mo:bnds:buca:1}
\end{align}
with $C$ independent of $\lambda$.  Here we take $X(s) = \|V_\lambda\|^2_{H^m} e^{Ct}$
so that $\frac{d}{dt} X \leq \frac{C}{\lambda}( X^{6} + \|\tilde{U}_0\|_{H^{m+2}}^4   +  \| h^0\|^2_{H^{m}} +1 )$.
Thus, from Lemma~\ref{lem:comp:lem:app}, we deduce a bound very similar 
to \eqref{eq:g:gw:con} but with a constant prefactor $e^{Ct}$.  The desired convergence \eqref{eq:m:buca:brak:estJ}
thus follows for any $m \geq 1$ by arguing mutatis mutandis as in the proof of Lemma~\ref{lem:0:conv}.  
\end{proof}

 \begin{proof}[Proof of Lemma~\ref{lem:BE:exact}]
 The proof is a direct computation.  From (\ref{eq:l:evol:bc:brak}) 
 it immediately follows that  
 \begin{align*}
 \Gamma_t^{\alpha \sigma_k^n}\Gamma_t^{\beta \sigma_j^\ell} U_0
  =U_0 +t 
   \begin{pmatrix}
     g \partial_x( \alpha e_k^n + \beta e_j^\ell) \\
 	\kappa \Delta ( \alpha e_k^n + \beta e_j^\ell)- b(\pi_\Vort U_0,  \alpha e_k^n + \beta e_j^\ell))  
  \end{pmatrix} 
   - t^2 g \alpha \beta 
   \begin{pmatrix} 0\\
 	b(\partial_x e_j^\ell, e_k^n)
  \end{pmatrix},
 \end{align*}
 where recall $b$ is as in \eqref{eq:buca:ad:not}.
 Since $\alpha, \beta \in \R$ and $U_0 \in H^m$ were arbitrary, 
 we can use the formula above and cancelations in $b$ like (\ref{eq:rel:deg:NSE})
 to conclude that 
 \begin{align}
 \nonumber
 \Gamma_t^{-\alpha \sigma_k^n} \Gamma_t^{-\beta \sigma_j^\ell}\Gamma_t^{\alpha \sigma_k^n}\Gamma_t^{\beta \sigma_j^\ell} U_0 &= U_0 + t^2 \alpha \beta g \begin{pmatrix} 0 \\
 b(\partial_x e_k^n, e_j^\ell) - b(\partial_x e_j^\ell, e_k^n) \end{pmatrix}\\
 \label{eq:sn:bk:p}&:= U_0 + t^2 \alpha \beta [[F, \sigma_j^\ell], [F, \sigma_k^n]],
 \end{align}
 which is the desired identity.
 \end{proof}

\subsection{3D incompressible Euler equation}
\label{sec:e:nse:ex}

We next turn to the low mode control problem for the three-dimensional
incompressible Euler equation.   In contrast to the previously
considered equations, this example is notable since:
\begin{itemize}
\item[(i)] We will see that dissipation is not needed to 
establish controllability results using the formalism developed
in Section~\ref{sec:sat}.
\item[(ii)] A major open problem is to determine whether the Euler 
equations develop singularities starting from smooth initial 
conditions.   See, e.g.,  \cite{Constantin2007}.
 Our saturation formalism introduced in Section~\ref{sec:overview} 
allows us to show that the addition of a low mode control to the
Euler equations can act to prevent blow up of solutions.  See Theorem~\ref{thm:NSEmain} and
Remark~\ref{rmk:non:blow} below.
\end{itemize}

Before proceeding further, a few preliminary remarks are in order.  First, since we will consider the Euler
equations in the absence of boundaries, the results and techniques
presented in this section also apply to the 3D incompressible
Navier-Stokes equations with only minor modifications.  We
omit details for the simplicity and clarity of presentation.  Second,
it may be noted that our presentation does not focus on applications to 
the stochastic counterpart of the Euler equations. Note that while it is
technically feasible to generalize some of the results
Section~\ref{sec:app:SPDEs} to locally defined dynamics, we
avoid this generalization here given the complexity of the results
as they already stand.  To see how such a generalization is possible in the finite-dimensional setting of SDEs, see \cite{HerMat15, Her11, BirHerWehr12}.   

Regarding existing literature concerning the controllability of 
the Euler equation, let us mention \cite{Shirikyan2008, Nersisyan2010, Nersesyan2015}
and also \cite{Rom_04,Shirikyan2007} for related work on the 3D 
Navier-Stokes equations.  The reference~\cite{Nersisyan2010} treats the same control problem as below but using the Agrachev-Sarychev approach in the functional setting of $H^m$ for an arbitrary but fixed $m\in \N$.  Below we treat the dynamics on the space $C^\infty= \cap_{m\geq 0} H^m$ using the methods of Section~\ref{sec:sat}.  In particular, because $m\in \N$ can be arbitrary the main result in~\cite{Nersisyan2010} implies the main control result for this dynamics (Theorem 5.112 below).  For general 
background on the mathematical theory of inviscid, incompressible flow, see \cite{MajdaBertozzi2002, MarchioroPulvirenti2012}.

\subsubsection*{Mathematical Formulation} 
The 3D Euler equations are 
\begin{align}\label{eq:NSE}
&\partial_t \bfU + (\bfU\cdot \nabla) \bfU + \nabla p =  \bfg+ \bfh,\\
 &\nonumber \quad \nabla \cdot  \bfU =0, \quad \bfU(0) = \bfU_0.
\end{align}
The equations~\eqref{eq:NSE} are posed on the torus $\TT^3=[0, 2\pi]^3$
with periodic boundary conditions, and the unknowns are
the fluid velocity field $\bfU = (u_1,u_2,u_2): \TT^3 \to \RR^3$ and the pressure $p: \TT^3 \to \RR $.  
The term $\bfg+\bfh$ represents an external volumetric force.    
We assume that $\bfg$ is a fixed background forcing and that
$\bfh$ is a control which takes values in a finite dimensional
control parameter space $X_0$.   Specifically we consider examples
where $X_0$ consists of trigonometric vector fields in order to make
our computations tractable.    A precise possible formulation for $h$ 
is given below.  See \eqref{eq:cnt:spc:eulr} and \eqref{eq:cont:form}.

Throughout what follows, we will assume that there is no mean flow
on the initial condition $\bfU_0$ or on the external forcing terms $\bfg$ and $\bfh$;
that is, 
\begin{align*}
  \int_{\TT^3} \bfU_0(x) \, dx =  \int_{\TT^3} \bfg(x) \, dx =    \int_{\TT^3} \bfh(x) \, dx = 0.
\end{align*}
Consequently, this mean-free condition will be preserved by the solution of~
\eqref{eq:NSE}.

Regarding the local semigroup formulation of~\eqref{eq:NSE}, we consider
$C^\infty$ smooth solutions as follows.  We define the spaces $H^m$ for $m \geq 0$ by 
 \begin{align}
   H^m = \bigg\{ \bfU \in H^m(\TT^3  )^3: 
   \nabla \cdot \bfU = 0, \int \bfU \, dx=0 \biggr\}.
   \label{eq:H:m:euler}
 \end{align}
We recycle previously used notation for the $L^2$ norm $\| \ccdot \|$ and inner product $\langle \ccdot, \ccdot \rangle$, as well as the notation used for $H^m$ norms $\| \ccdot \|_{H^m}$.   
We let 
\begin{align}
  \mathcal{X} : = \bigg\{ \bfU \in C^\infty(\TT^3)^3: \nabla \cdot \bfU = 0,\int \bfU \, dx=0 \biggr\}
   = \bigcap_{m\geq 0} H^m.
\end{align}
where $C^\infty(\TT^3)$ is the collection of smooth, periodic
functions.   In this example, the ambient (Frech\'{e}t) phase space is
$(\mathcal{X},d_\infty)$ where the metric $d_\infty$ is given by\footnote{Note
  that this is equivalent to the usual Fr\'echet topology on $C^\infty$
  via Sobolev embedding.}
 \begin{align*}
   d_\infty(\bfV, \bar{\bfV}) 
      = \sum_{m=0}^\infty 2^{-m} (1\wedge \| \bfV -\bar{\bfV}\|_{H^m}).
 \end{align*}

Let us next recall some results concerning the local (in time) existence and
uniqueness of smooth solutions of \eqref{eq:NSE}.  For this
we fix a `death state' $\death\notin \mathcal{X}$.  
\begin{proposition}
\label{prop:NSEsemi}
Fix any $\bfg \in \mathcal{X}$ and any finite dimensional space $X_0 \subseteq \mathcal{X}$. 
\begin{itemize}
\item[(i)] For any $\bfU_0 \in \mathcal{X}$ and any $\bfh \in X_0$, there exists a unique $0 < T_{\bfU_0, \bfh} \leq \infty$
and  
\begin{align}
    \bfU(\ccdot) = \bfU(\ccdot, \bfU_0, \bfh) \in  C([0, T_{\bfU_0, \bfh}), \mathcal{X})
\end{align}
solving \eqref{eq:NSE} such that if $T_{\bfU_0, \bfh} < \infty$, then
\begin{align}
  \limsup_{t \to T_{\bfU_0, \bfh}} \|\nabla \bfU(t)\|_{L^\infty} =
  \infty.
  \label{eq:its:gonna:blow}  
\end{align} 
\item[(ii)] Take $(T_{\bfU_0, \bfh})_{\bfU_0\in \mathcal{X}, \bfh \in X_0}$ to be the collection of positive times defined in (i) and 
\begin{align*}
  \Phi_t^\bfh \bfU_0 := 
  \begin{cases}
    \bfU(t, \bfU_0, \bfh) & \text{ when } t < T_{\bfU_0, \bfh},\\
    \death   & \text{ when } t \geq T_{\bfU_0, \bfh}.
    \end{cases}
  \end{align*}
  Then the mapping $(t, \bfU_0, \bfh) \mapsto \Phi_{t}^\bfh \bfU_0: [0, \infty) \times \mathcal{X} \times X_0\rightarrow \mathcal{X} \cup \{ \death \}$ 
  is a one-parameter family of continuous local semigroups on $(\mathcal{X}, d_\infty)$ parametrized by $X_0$ in the sense 
  of Definition~\ref{def:oneparamlocal}.  
\end{itemize}
\end{proposition}

The proof of Proposition~\ref{prop:NSEsemi} is fairly standard (see 
\cite{MajdaBertozzi2002, MarchioroPulvirenti2012}) and is based on 
a priori estimates which we recall below in Appendix \ref{sec:NSE:apriori}.

\subsubsection*{The Control Parameter Space and Algebraic Conditions}
With the basic mathematical setting for \eqref{eq:NSE} in hand, we detail the assumptions on the control parameter space $X_0$ which will allow us to prove exact control results.

For this purpose, we begin by defining a divergence-free trigonometric basis as follows.
For each $\bk \in \Z^3_{\neq0}$, pick $\aa_{\bk}^{(0)}, \aa_{\bk}^{(1)} \in \RR^3$ such that
\begin{align}
 \aa_{\bk}^{(0)} \cdot \bk = \aa_{\bk}^{(1)} \cdot \bk = \aa_{\bk}^{(0)} \cdot \aa_{\bk}^{(1)} = 0, \quad | \aa_{\bk}^{(0)} |^2 = |\aa_{\bk}^{(1)}|^2 = \frac{1}{4\pi}.
  \label{eq:cnd:r3:dir}
\end{align}
For $\bk \in \Z^3_{\neq 0}$ and $l, m \in \{ 0,1\}$, we define
\begin{align*}
  \bfe_{\bk, l,m} 
    =2 \aa_{\bk}^{(l)} \mbox{Re}( i^m e^{-i \bk \cdot x})
    = 
  \begin{cases}
    2\aa_{\bk}^{(l)} \cos(\bk \cdot x)  &\text{ if } m = 0,\\
    2\aa_{\bk}^{(l)} \sin(\bk \cdot x)  &\text{ if } m = 1.
    \end{cases}
\end{align*}
We denote
\begin{align}
  F_{\bk} := \mbox{span} \{  \bfe_{\bk, l,m} : l,m \in \{0,1\} \},
  \label{eq:k:vec:span}
\end{align}
for any $\bk \in \Z^3_{\neq 0}$.  Notice that $F_{\bk} = F_{-\bk}$ for any $\bk \in \ZZ^3_{\neq 0}$.

To specify the control space $X_0$ for (\ref{eq:NSE}), we consider any subset $\mathcal{Z} \subseteq \Z^3_{\neq0}$ and 
define 
\begin{align}
  X_0 = \mbox{span}\{ \bfe_{\bk,l, m} : \bk \in \mathcal{Z}, m, l \in \{0,1\} \} = \mbox{span}\{ F_{\bk}: \bk \in \mathcal{Z} \}
    \label{eq:cnt:spc:eulr}
\end{align}
so that, in particular, the control $\bfh$ has the form
\begin{align}
  \bfh(t) = \sigma \cdot \alpha =
  \sum_{\substack{\bk \in \mathcal{Z},\\ l,m \in \{0,1\}}} \alpha_{\bk, l, m}(t) \bfe_{\bk,l, m}.
  \label{eq:cont:form}
\end{align}
Note that the control parameter $\alpha$ takes values in $\RR^{4 | \mathcal{Z}|}$.
\begin{remark}
  To simplify our presentation, we restrict to the case when each 
  wave vector $\bk$ is `fully-controlled'.  Note that a very similar restriction
  on the control configuration was imposed above for both the
  2D Navier-Stokes equations and Boussinesq equations studied previously above;
  see (\ref{eq:NSE:con:form}) and (\ref{eq:bcontrol:form}), respectively.
\end{remark}

Below we will show that following algebraic condition on $\mathcal{Z}$, identified in \cite{Rom_04}, is sufficient to establish
controllability properties for (\ref{eq:NSE}).
\begin{definition}
  \label{def:det:modes:E}
Let $\bj, \bk \in \Z_{\neq 0}^3$.  We say that \emph{$ \bj +\bk$ is an admissible move from $\bj, \bk$} 
if
  \begin{align}
        \bj, \bk  \text{ are linearly independent and  } |\bj| \neq |\bk|.
\label{eq:Euler:WN:intcon}
\end{align}
Here $|\cdot|$ denotes the standard Euclidean norm.  Let $\mathcal{Z}_0:= \mathcal{Z}$ and for $n\geq 1$ define $\mathcal{Z}_n$ inductively by  
\begin{align*}
  \mathcal{Z}_n 
  := \{ \mathbf{l} \in \Z^3_{\neq 0} : \mathbf{l} = \bk + \bj, \, \mathbf{l} \text{ an admissible move from } \bk, \bj \in \mathcal{Z}_{n-1} \} \cup \mathcal{Z}_{n-1}.  
\end{align*}
We say that \emph{$\Zb$ is a determining set of modes} if 
\begin{align}
  \mathcal{Z}_\infty := \bigcup_{n \geq 0} \mathcal{Z}_n = \Z^3_{\neq 0}.
  \label{eq:det:modes:set}
\end{align}
\end{definition}

\begin{remark}
It is possible to give a complete algebraic 
characterization of configurations $\Zb \subset \Z^3_{\neq 0}$ 
which are determining sets of modes.  
See Proposition~5.2 in \cite{Rom_04} for a detailed discussion of this point.  On the other hand, variations on the conditions given in Definition~\ref{def:det:modes:E}
are possible to guarantee the controllability of (\ref{eq:NSE}).  In particular, we will show that
controllability follows if, for example,  
$\{(1, 0, 0), (0,1,0), (0,0,1)\} \subset \mathcal{Z}$.  
\end{remark}

\subsubsection*{Statement of the main result}

With these preliminaries in hand, we now state the main result of this section.  
\begin{theorem}\label{thm:NSEmain}
  Take $\mathfrak{F} = \{(\Phi, X_0)\}$ to be the associated one-parameter family of continuous local semigroups defined by (\ref{eq:NSE}) and
  let $\pi :\mathcal{X}\rightarrow \mathcal{X}$ be any continuous,  linear projection operator onto 
    a finite-dimensional subspace $\pi(\mathcal{X})\subseteq \mathcal{X}$.    Suppose that either of the following conditions is satisfied:
  \begin{itemize}
  \item[(i)] $\Zb$ is a determining set of modes according 
    to Definition~\ref{def:det:modes:E}.
  \item[(ii)] $\{(1,0,0), (0,1,0), (0,0,1) \} \subseteq \Zb.$
  \end{itemize}
  Then $\mathbb{D}(\mathfrak{F})$ is approximately controllable on $\mathcal{X}$ and exactly controllable on $\pi(\mathcal{X})$ in the sense of
  Definition~\ref{def:acc:cont}.  Here recall that $\mathbb{D}(\mathfrak{F})$ is defined in \eqref{eq:unwrap:opp}.       
\end{theorem}

\begin{remark}\label{rmk:non:blow}
A notable consequence of this result is that it implies blow-up can be averted in equation~\eqref{eq:NSE} 
by allowing control over a few low modes.  
Moreover, this is still true even in the presence of an
arbitrary the fixed background forcing term $\bfg \in \mathcal{X}$.  
\end{remark}

The proof of Theorem~\ref{thm:NSEmain} is based on three lemmata which we state next.  The first concerns the algebraic structure of the nonlinear terms in (\ref{eq:NSE}), while the second and third results provide quantitative bounds on the 
usual scalings (\ref{eq:ray:scal}) and (\ref{eq:non:lin:tw}). 

For the first Lemma it is convenient to introduce some notation for nonlinear portion of (\ref{eq:NSE}).
Given any $\bff, \tilde{\bff} \in \mathcal{X}$,
\begin{align}
  B(\bff, \tilde{\bff}) = P( \bff \cdot \nabla \tilde{\bff} + \tilde{\bff} \cdot \nabla \bff).
  \label{eq:E:NLT}
\end{align}
where $P$ is the Leray projection operator onto mean-free, divergence-free 
vector fields.
Equivalently we may write 
\begin{align}
     B(\bff, \tilde{\bff}) = \bff \cdot \nabla \tilde{\bff} + \tilde{\bff} \cdot \nabla \bff + \nabla q
     \label{eq:under:pres}
\end{align}
where $q: \TT^3 \to \RR$ solves
\begin{align*}
    -\Delta q = \nabla \cdot (\bff \cdot \nabla \tilde{\bff} + \tilde{\bff} \cdot \nabla \bff).
\end{align*}
In particular this shows that $B(\bff, \tilde{\bff}) \in \mathcal{X}$ whenever
$\bff, \tilde{\bff} \in \mathcal{X}$. See, for example, \cite{ConstantinFoias1988,Temam1995} for further details.

Note that even though $B$ is a second-degree polynomial nonlinearity, a cancellation
condition similar to the situation described in Section~\ref{rem:even}
holds for $B$.  This cancellation allows us to `reach' successively higher frequencies 
though the scaling analysis.  
\begin{lemma}
\label{lem:you:can:call:me:al}
The following algebraic relationships between $B$ defined by
\eqref{eq:E:NLT} and $F_\bk$ given as \eqref{eq:k:vec:span}
hold.
\begin{itemize}
\item[(i)] For every $\bk \in \Z^3_{\neq 0}$ we have that 
  \begin{align}
    B(\bfe, \tilde{\bfe}) = 0
    \quad \text{ for all }\quad  \bfe, \tilde{\bfe} \in F_\bk.
    \label{eq:B:can:3D}
  \end{align}
\item[(ii)]  For every $\bj, \bk \in \Z^3_{\neq 0}$,
\begin{align}
  \spa \, \{ B(\bfe, \tilde{\bfe}) \, : \, \bfe \in F_\bj, \tilde{\bfe} \in F_\bk \} 
     \subseteq \spa \{ F_{\bj-\bk}\cup F_{\bj+\bk} \}.
\label{eqn:inequality1}
\end{align}
If $\bj + \bk$ is an admissible move from $\bj, \bk\in \Z_{\neq 0}^3$, 
then equality holds in~\eqref{eqn:inequality1}; that is, 
\begin{align}
  \spa \, \{ B(\bfe, \tilde{\bfe}) \, : \, \bfe \in F_\bj, \tilde{\bfe} \in F_\bk \} 
     = \spa \{ F_{\bj-\bk}\cup F_{\bj+\bk} \}.
\label{eqn:equality1}
\end{align}
\item[(iii)] Finally,
  \begin{align}
   F_{(1,1,1)} 
    \subseteq  
    \spa \, \{ B(B(\bfe, \tilde{\bfe}), \tilde{\tilde{\bfe}}) 
      \, : \, \bfe, \tilde{\bfe}, \tilde{\tilde{\bfe}} \in F_{(1,0,0)} \cup F_{(0,1,0)} \cup F_{(0,0,1)} \}.
      \label{eq:el:sp:ex}
  \end{align}
\end{itemize}
\end{lemma}

Turning to quantitative bounds on scalings we have:
\begin{lemma}
\label{lem:NSfa}
Let $t>0$ and fix compact sets $K_1\subseteq \mathcal{X}$ and $K_2 \subseteq X_0$.  Then there exists $\lambda_0= \lambda_0(K_1, K_2, t) >0$ 
sufficiently large such that for all $\lambda \geq \lambda_0$ we have
that $\NS_{t/\lambda}^{\lambda \bfh} u_0 \in \mathcal{X}$ for all
$\bfU_0 \in K_1$, $\bfh\in K_2$.  In other words, defining
$T_{\bfU_0, \bfh}$ as in Proposition~\ref{prop:NSEsemi}, we have for $\lambda \geq \lambda_0$
\begin{align}
   \lambda  \inf_{\bfU_0 \in K_1, \bfh \in K_2} T_{\bfU_0,\lambda \bfh} \geq t.
  \label{eq:too:comp:to:blow}
\end{align}
Moreover
\begin{align}
  \lim_{\lambda \to \infty} \sup_{\bfU_0 \in K_1, \bfh \in K_2} 
  d_\infty(\NS_{t/\lambda}^{\lambda \bfh} \bfU_0, \rho_t^\bfh \bfU_0) = 0.
  \label{eq:smth:land:sc:1}
\end{align} 
Here recall that $\rho$ is the ray semigroup defined in \eqref{eq:ray:sg}.  
Consequently, $(\rho, X_0 ) \in \text{\emph{Sat}}_u(\mathfrak{F})$.  
\end{lemma}

\begin{lemma}
\label{lem:NSsa}
Fix $t >0$ and let $K_1, K_2 \subseteq \mathcal{X}$ be compact sets.  
Then there exists $\lambda_0=\lambda_0(K_1, K_2 ,t)>0$ large enough such that for all $\lambda \geq \lambda_0$,
$\NS_{t/\lambda^2}^0  \, \rho_{1/\lambda}^{\lambda^2 \bfh} \bfU_0 \in \mathcal{X}$ 
for all $\bfU_0\in K_1$ and $\bfh \in K_2$.  Moreover,
\begin{align}
  \lim_{\lambda \to \infty} \sup_{\bfU_0 \in K_1, \bfh \in K_2} 
  d_\infty(\rho_{1/\lambda}^{-\lambda^2 \bfh} \, \NS_{t/\lambda^2}^0  \, \rho_{1/\lambda}^{\lambda^2 \bfh} \bfU_0, 
                   \rho_{t}^{-B(\bfh,\bfh)} \bfU_0) = 0
  \label{eq:smth:land:sc:2}
 \end{align}
 where $B$ is defined in \eqref{eq:E:NLT}.
 \end{lemma}

\subsubsection*{Proof of the Main Results}

Before proving the three lemmata above, we first see why combining them implies Theorem~\ref{thm:NSEmain}.

\begin{proof}[Proof of Theorem~\ref{thm:NSEmain}]
As with the main results in the previous examples, the proof
proceeds by establishing the conditions for controllability given in Corollary~\ref{cor:exact:con}.
Under the assumption (i) define subspaces 
\begin{align}
  X_n := \spa \bigg(\bigcup_{\bk \in \Zb_n} F_\bk\bigg)
  \label{eq:access:set:E}
\end{align}
for every $n \geq 0$, where $\Zb_n$ is as in Definition~\ref{def:det:modes:E}.  Note 
that \eqref{eq:det:modes:set} implies $\cup_{n \geq 1} X_n$ is a dense subset of $\mathcal{X}$.   
As such, by proving inductively that $(\rho, X_n) \in \mbox{Sat}_u(\mathfrak{F})$ for every $n \geq 0$, the desired
controllability result immediately follows from Corollary~\ref{cor:exact:con}.

According to Lemma~\ref{lem:NSfa}, we have that $(\rho, X_0) \in \text{Sat}_u(\mathfrak{F})$.  
Next, utilizing Lemma~\ref{lem:you:can:call:me:al} and the cancellation \eqref{eq:B:can:3D} we infer that 
\begin{align}
    \lim_{\lambda \to \infty} \sup_{\bfU_0 \in K, |\alpha| \leq R} 
  d_\infty(\rho_{1/\lambda}^{-\lambda^2 (\bfe + \alpha \tilde{\bfe})} \, \NS_{t/\lambda^2}^0  \, \rho_{1/\lambda}^{\lambda^2 (\bfe + \alpha \tilde{\bfe})} \bfU_0, 
                   \rho_{t}^{-\alpha B(\bfe,\tilde{\bfe})} \bfU_0) = 0
  \label{eq:gen:sc:E:can}
\end{align}
for any $R >0$, any compact set $K \subseteq \mathcal{X}$ and any pair $\bfe \in F_\bk, \tilde{\bfe} \in F_\bj$, $\bk, \bj \in \Z^3_{\neq 0}$.
Thus if $(\rho, X_{n-1}) \in \text{Sat}_u(\mathfrak{F})$ for some $n\geq 1$, we immediately infer
that $(\rho, \spa \{ B(\bfe, \tilde{\bfe})\})\in \text{Sat}_u(\mathfrak{F})$ for any $\bfe \in F_\bk, \tilde{\bfe} \in F_\bj$
such that $\bk, \bj \in \mathcal{Z}_{n-1}$.  Invoking Remark~\ref{rmk:stupid:spans} with \eqref{eqn:equality1} and the assumed 
structure of the sets $\Zb_{n-1}$ and $\Zb_n$, we infer that $(\rho, X_n) \in \text{Sat}_u(\mathfrak{F})$.  
This completes the proof under assumption (i).

To show the result under assumption (ii), define
\begin{align*}
  \mathcal{Z}_0 = \{(1,0,0), (0,1,0), (0,0,1), (1,1,1)\} 
\end{align*}
and then iteratively define sets $\mathcal{Z}_n$ precisely as in 
Definition~\ref{def:det:modes:E} starting from this
particular choice of $\mathcal{Z}_0$.   Using this
definition of the index sets $\Zb_n$, we define $X_n \subseteq
\mathcal{X}$ as in \eqref{eq:access:set:E}.  As in the previous case,
we will show inductively that $(\rho,X_n) \in \mbox{Sat}_u(\mathfrak{F})$
for every $n \geq 0$.   After that, we will show explicitly that $\mathcal{Z}_\infty = \Z^3_{\neq 0}$, cf. (\ref{eq:det:modes:set}), thus completing the proof under assumption (ii).

The implication that $(\rho,X_{n-1}) \in \mbox{Sat}_u(\mathfrak{F})$
implies $(\rho,X_{n}) \in \mbox{Sat}_u(\mathfrak{F})$ for $n\geq 1$ is demonstrated exactly as in the case of assumption (i).  
We now show that $(\rho,X_{0}) \in 
\mbox{Sat}_u(\mathfrak{F})$.  Define
\begin{align*}
  \mathcal{Z}_{-1} := \{(1,0,0), (0,1,0), (0,0,1)\}  \quad \text{ and } \quad X_{-1} := \spa \bigcup_{\bk \in \Zb_{-1}} F_\bk . 
\end{align*}
Invoking Lemma~\ref{lem:NSfa}, we see that $(\rho,X_{-1}) \in \mbox{Sat}_u(\mathfrak{F})$.  Next, the estimate \eqref{eq:gen:sc:E:can} implies that $(\rho, \spa \{ B(\bfe, \tilde{\bfe})\})\in \mbox{Sat}_u(\mathfrak{F})$
for any pair $\bfe, \tilde{\bfe} \in F_{(1,0,0)} \cup F_{(0,1,0)} \cup F_{(0,0,1)}$.  Making note of the 
containment \eqref{eqn:inequality1} and second use of \eqref{eq:gen:sc:E:can} 
we infer $(\rho, \spa \{ B(B(\bfe, \tilde{\bfe}), \tilde{\tilde{\bfe}})\})\in \mbox{Sat}_u(\mathfrak{F})$
for any  $\bfe, \tilde{\bfe}, \tilde{\tilde{\bfe}} \in F_{(1,0,0)} \cup F_{(0,1,0)} \cup F_{(0,0,1)}$.  
With \eqref{eq:el:sp:ex} and Remark~\ref{rmk:stupid:spans}, we now conclude $(\rho, X_0)\in \mbox{Sat}_u(\mathfrak{F})$.

With the induction for case (ii) now in hand, we have left to show that $\mathcal{Z}_\infty = \Z^3_{\neq0}$.
There are many ways to do this explicitly.  For example,
note that $(1,1,1)$ paired with any of $(1,0,0),(0,1,0),(0,0,1)$ satisfies \eqref{eq:Euler:WN:intcon}.  Also,
we note that if $\bk \in \Zb_\infty$ then $-\bk \in \Zb_\infty$.  Consequently, we obtain 
\begin{align*}
  \{ (1,0,0),(0,1,0),(0,0,1), (1,1,0), (0,1,1), (1,0,1) \} \subseteq \Zb_\infty.
\end{align*}
Starting from these directions, it is not hard to show that by using a sequence of admissible moves 
(in the sense of Definition \ref{def:det:modes:E}) the set $\Zb_\infty$
includes all three axes; namely,
\begin{align*}
  \{(n,0,0): n \in \Z_{\neq0}\} \cup \{(0,n,0): n \in \Z_{\neq0}\} \cup \{(0,0,n): n \in \Z_{\neq0}\}  \subseteq \Zb_\infty.
\end{align*}
Now take an arbitary element $(n_1, n_2, n_3) \in \Z^3_{\neq0}$.  If $n_1 \not = n_2$, we obtain
$(n_1, n_2, 0) \in \Zb_\infty$ as the admissible move from $(n_1,0,0), (0,n_2,0) \in \Zb_\infty$.  Otherwise
if $n_1= n_2$, we can obtain successively $(n_1 \pm 1, n_2, 0) \in \Zb_\infty$ and then $(n_1,n_2,0) \in \Zb_\infty$
via admissible moves.  Similar if $n_3^2 \not = n_1^2 + n_2^2$ we find that $(n_1,n_2, n_3) \in \Zb_\infty$
via the admissible move from $(n_1, n_2, 0), (0,0,n_3) \in \Zb$.  Otherwise if $n_3^2 = n_1^2 + n_2^2$
we simply make the admissible move to $(n_1, n_2, n_3 \pm 1) \in \Zb_\infty$ from $(n_1,n_2,0), (0,0,n_3 \pm 1) \in \Zb_\infty$. We then obtain $(n_1, n_3, n_3) \in \Zb_\infty$ from $(n_1, n_2, n_3 \pm
1), \mp (0,0,1) \in \Zb_\infty$.   With this we have thus completed
the proof of case (ii) and hence of Theorem~\ref{thm:NSEmain}.

\end{proof}

\begin{proof}[Proof of Lemma~\ref{lem:you:can:call:me:al}]

Consider basis elements $\bfe_{\bk, l_1,m_1}, \bfe_{\bj,l_2 , m_2} $ for $\bj, \bk \in \Z_{\neq 0}^3$ with $l_i, m_i\in \{ 0,1\}$.  First observe that we can extend the definition of these elements naturally to include any $m_i \in \Z$, and these new elements are clearly constant multiples of the original basis elements.  We will use this fact below.  Now for any $l_i \in \{ 0,1\}$ and any $m_i \in \Z$, a tedious but routine computation yields
\begin{align*}
  \bfe_{\bk, l_1,m_1} \cdot \nabla \bfe_{\bj, l_2,m_2} +& \bfe_{\bj, l_2,m_2} \cdot \nabla \bfe_{\bk, l_1,m_1}\\
  =& -2\biggl((\aa^{(l_1)}_\bk\cdot \bj) \aa^{(l_2)}_\bj + (\aa^{(l_2)}_\bj\cdot \bk) \aa^{(l_1)}_\bk\biggr)
     \mbox{Re}(i^{m_1+m_2 + 1} e^{-i(\bk + \bj) \cdot x})\\
   &+2(-1)^{m_2} \biggl((\aa^{(l_1)}_\bk\cdot \bj)\aa^{(l_2)}_\bj- (\aa^{(l_2)}_\bj\cdot \bk)\aa^{(l_1)}_\bk \biggr)
       \mbox{Re}(i^{m_1+m_2+1} e^{-i(\bk -\bj) \cdot x}).
\end{align*}
Recalling that the Leray projection operator $P$ acts as  
\begin{align*}
  P[\mbox{Re}(\bfV e^{-i \bk \cdot x})] 
  = \mbox{Re}\biggl(\bfV- \frac{\bfV\cdot \bk}{|\bk|^2} \bk \biggr)e^{-i\bk \cdot x}
\end{align*}
for any $\bk \in \Z^3$ and $\bfV \in \C^3$, we therefore obtain
\begin{align}
  B(\bfe_{\bk, l_1,m_1}&,  \bfe_{\bj, l_2,m_2}) \notag\\
 &= -2  \bfr_{\bk, \bj}^{l_1, l_2}
    \mbox{Re} ( i^{m_1 + m_2 + 1} e^{-i(\bk+ \bj)})
   +2(-1)^{m_2} \bfs_{\bk,\bj}^{l_1, l_2}
          \mbox{Re}  (i^{m_1 + m_2+1} e^{-i(\bk - \bj)\cdot x}))
  \label{eq:e:brak:ekj}
\end{align}
where
\begin{align*}
 \bfr_{\bk, \bj}^{l_1, l_2} &:= (\aa^{(l_1)}_\bk\cdot \bj) \left(\aa^{(l_2)}_\bj -
   \frac{\aa^{(l_2)}_\bj\cdot \bk }{|\bk + \bj|^2}(\bk + \bj) \right)
   + (\aa^{(l_2)}_\bj\cdot \bk) \left( \aa^{(l_1)}_\bk -
   \frac{\aa^{(l_1)}_\bk\cdot \bj }{|\bk + \bj|^2}(\bk + \bj) \right),\\
 \bfs_{\bk, \bj}^{l_1, l_2} &:= (\aa^{(l_1)}_\bk\cdot \bj)\left(\aa^{(l_2)}_\bj -
   \frac{\aa^{(l_2)}_\bj\cdot \bk }{|\bk - \bj|^2}(\bk - \bj) \right) -
     (\aa^{(l_2)}_\bj\cdot \bk) \left( \aa^{(l_1)}_\bk +
   \frac{\aa^{(l_1)}_\bk\cdot \bj }{|\bk - \bj|^2}(\bk - \bj) \right).
\end{align*}
In particular this shows that if $\bj = \bk$ then $B(\bfe_{\bk, l_1,m_1}, \bfe_{\bj,l_2 , m_2}) =0$.  This
implies the first item, \eqref{eq:B:can:3D}.  

We also use \eqref{eq:e:brak:ekj} to address part (ii) of the result. Since by definition of the Leray projection
\begin{align*}
  \bfr_{\bk, \bj}^{l_1, l_2}\cdot (\bk + \bj) = 0,  \quad  \bfs_{\bk, \bj}^{l_1, l_2} \cdot (\bk - \bj) = 0,
\end{align*}
we infer that $\bfr_{\bk, \bj}^{l_1, l_2}$ and $\bfs_{\bk, \bj}^{l_1, l_2}$ can be written as linear combinations
of elements $\aa^{(l)}_{\bk + \bj}$ and $\aa^{(l)}_{\bk - \bj}$, respectively.  As such we have that 
\begin{align}
  \spa \, \{ B(\bfe, \tilde{\bfe}) \, : \, \bfe \in F_\bj, \tilde{\bfe} \in F_\bk \} 
     \subseteq \spa \{ F_{\bj-\bk}\cup F_{\bj+\bk} \}.
  \label{eq:easier:inc:3El}
\end{align}
Next notice that
\begin{align}
  B(\bfe_{\bk, l_1,m},  \bfe_{\bj, l_2,0}) + B(\bfe_{\bk, l_1,m-1},  \bfe_{\bj, l_2,1}) 
   = -4 \bfr_{\bk, \bj}^{l_1, l_2}
    \mbox{Re} ( i^{m + 1} e^{-i(\bk+ \bj)})
  \label{eq:e:brak:ekj:1}
\end{align}
and similarly
\begin{align*}
  B(\bfe_{\bk, l_1,m},  \bfe_{\bj, l_2,0}) - B(\bfe_{\bk, l_1,m-1},  \bfe_{\bj, l_2,1}) 
   = -4 \bfs_{\bk, \bj}^{l_1, l_2}
    \mbox{Re} ( i^{m + 1} e^{-i(\bk - \bj)}).
\end{align*}
Thus, taking linear combinations we find that the sets
\begin{align}
  \tilde{F}_{\bk+\bj} := \biggl\{ \biggl[ (\aaa_\bk\cdot \bj) &\bigl(\aaa_\bj -
   \frac{\aaa_\bj\cdot \bk }{|\bk + \bj|^2}(\bk + \bj) \bigr)
   + (\aaa_\bj\cdot \bk) \bigl( \aaa_\bk -
   \frac{\aaa_\bk\cdot \bj }{|\bk + \bj|^2}(\bk + \bj) \bigr)
                        \biggr] \mbox{Re}(i^m e^{-i(\bk +\bj) \cdot x}) \notag\\
                  &:m \in \{0,1\}, \aaa_\bk, \aaa_\bj \in \RR^3 \text{ with } \aaa_\bk \cdot \bk = 0 = \aaa_\bj \cdot \bj \biggr\}
  \label{eq:euler:alg:mess:1}\\
  \tilde{F}_{\bk-\bj} := \biggl\{ \biggr[(\aaa_\bk\cdot \bj) &\bigl(\aaa_\bj -
   \frac{\aaa_\bj\cdot \bk }{|\bk - \bj|^2}(\bk - \bj) \bigr)
   - (\aaa_\bj\cdot \bk) \bigl( \aaa_\bk +
   \frac{\aaa_\bk\cdot \bj }{|\bk - \bj|^2}(\bk - \bj) \bigr) \biggr]  \mbox{Re}(i^m e^{-i(\bk -\bj) \cdot x}) \notag\\
                  &:m \in \{0,1\}, \aaa_\bk, \aaa_\bj \in \RR^3 \text{ with } \aaa_\bk \cdot \bk = 0 = \aaa_\bj \cdot \bj \biggr\}
              \notag
\end{align}
are both subsets of $\spa \, \{ B(\bfe, \tilde{\bfe}) \, : \, \bfe \in F_\bj, \tilde{\bfe} \in F_\bk \}$.  
Thus to obtain the opposite inclusion in \eqref{eq:easier:inc:3El} and
complete the proof it is sufficient to show $\tilde{F}_{\bk+\bj}= F_{\bk+\bj}$  and $\tilde{F}_{\bk-\bj}= F_{\bk-\bj}$.  For this purpose we simply exhibit suitable choices of elements 
$\aaa_\bk, \aaa_\bj \in \R^3$, orthogonal to, respectively, $\bk, \bj$ so that the $\RR^3$-valued pre-factors of the elements 
in $\tilde{F}_{\bk+\bj}$  and $\tilde{F}_{\bk-\bj}$ span the planes orthogonal
to $\bk + \bj$ and $\bk - \bj$ respectively.  Take
$\bar{\aaa}_\bk$ and $\bar{\aaa}_\bj$ non-zero vectors which are orthogonal to $\bk, \bk \times \bj$ and 
$\bj, \bk \times \bj$ respectively.  Thus
\begin{align*}
  \biggl[ (\bar{\aaa}_\bk\cdot \bj) \bigl(\bar{\aaa}_\bj -
    \frac{\bar{\aaa}_\bj\cdot \bk }{|\bk + \bj|^2}(\bk + \bj) \bigr)
    + (\bar{\aaa}_\bj\cdot \bk) \bigl( \bar{\aaa}_\bk -
    \frac{\bar{\aaa}_\bk\cdot \bj }{|\bk + \bj|^2}(\bk + \bj) \bigr)
                         \biggr] 
\end{align*}
is a pre-factor of an element in $\tilde{F}_{\bk + \bj}$.   Also, taking $\aaa_\bk \in \R^3$ with $\aaa_\bk \cdot \bk=0$ abitrary and $\aaa_\bj= \bk \times \bj$, we see that 
\begin{align} 
 \label{eq:euler:alg:mess:2}
(\aaa_\bk\cdot \bj)(\bk \times \bj)  
  \quad 
\end{align}
is a pre-factor of an element in $\tilde{F}_{\bk + \bj}$.  Clearly, the previous two vectors are orthogonal. 
To show that these vectors can be chosen non-zero, we invoke the algebraic assumptions (\ref{eq:Euler:WN:intcon}) on $\bk, \bj$.
For the second vector, we may obviously choose $\aaa_\bk$ so that $\aaa_\bk \cdot \bj \neq 0$.  Regarding the first vector, doting  with $\bj$ we obtain
\begin{align*}
  (\bar{\aaa}_\bj\cdot \bk) ( \bar{\aaa}_\bk \cdot \bj)\biggl( 1 -\frac{2\bk \cdot \bj + 2|\bj|^2}{|\bk + \bj|^2} \biggr)
\end{align*}
Since this expression can only be zero when one of $\bar{\aaa}_\bj\cdot \bk$, $\bar{\aaa}_\bk \cdot \bj$ is zero 
or $|\bj| = |\bk|$. According to (\ref{eq:Euler:WN:intcon}) neither occur and we infer the second vector must be non-zero.
A very similar argiment also shows that $\tilde{F}_{\bk-\bj}= F_{\bk - \bj}$ so that now (\ref{eqn:equality1}) follows.

For the final item, \eqref{eq:el:sp:ex}, arguing precisely as
in \eqref{eq:e:brak:ekj:1}, \eqref{eq:euler:alg:mess:1} and 
choosing the first vector in \eqref{eq:euler:alg:mess:2} we find that 
\begin{align*}
  \spa \, \{ (0,0,1) \mbox{Re}(i^m e^{-i(1,1,0) \cdot x}) : m \in \{0,1\}\}
  \subseteq \spa \, \{ B(\bfe, \tilde{\bfe}) \, : \, \bfe \in F_{(1,0,0)} , \tilde{\bfe} \in F_{(0,1,0)} \}
\end{align*}
Next, another laborious but routine computation similar to \eqref{eq:e:brak:ekj:1} reveals that
the span of elements of the form
\begin{align*}
   B(\bfe_{(0,0,1), l,m}, (0,0,1) \mbox{Re}(e^{-i(1,1,0)\cdot x})) 
           + B(\bfe_{(0,0,1), l,m-1},  (0,0,1) \mbox{Re}(i e^{-i(1,1,0)\cdot x}))
\end{align*}
for $m, l \in \{0,1\}$ contains $F_{(1,1,1)}$.  This implies \eqref{eq:el:sp:ex} now completing the proof.
\end{proof}

We conclude this section by establishing the two scaling estimates, Lemmas~\ref{lem:NSfa} and \ref{lem:NSsa}.  
For these estimates we will make use of the ODE comparison stated in Proposition~\ref{prop:odecomp}.

\begin{proof}[Proof of Lemma~\ref{lem:NSfa}]
Fix any $K_1, K_2 \subseteq \mathcal{X}$ compact.  For any $\bfU_0 \in K_1$
and $\bfh \in K_2$ we set
\begin{align}
  \bfw_\lambda(\tau)  = \NS_{\tau/\lambda}^{\lambda \bfh} \bfU_0 - \rho_\tau^\bfh \bfU_0 \quad  \text{ and } \quad
  \rho(\tau) =\rho_\tau^h \bfU_0.      
\end{align}
which are well defined elements for $\tau$ in the interval $[0, \lambda T_{\bfU_0, \lambda \bfh})$ 
where $T_{\bfU_0, \lambda \bfh}$ is the time of existence of $\NS^{\lambda\bfh}_{\cdot}\bfU_0$; 
see Proposition~\ref{prop:NSEsemi}.   Arguing as in (\ref{eq:rescal:0}) we find
\begin{align}
  \partial_t \bfw_\lambda + \frac{1}{\lambda}( \rho + \bfw_\lambda) \cdot \nabla (\rho + \bfw_\lambda) + \nabla p_\lambda = \frac{1}{\lambda} \bfg \qquad
  \nabla \cdot \bfw_\lambda = 0 = \nabla \cdot \rho
  \label{eq:scale:L1:eul}
\end{align}
Here note that the pressure term  $p_\lambda: \TT^3 \to \RR$ is a smooth function which 
maintains the divergence-free condition~\eqref{eq:scale:L1:eul}.

Fixing any $m \geq 3$, we estimate $\bfw_\lambda$ in the $H^m$ norm as follows.  Using that $\bfw_\lambda$ is divergence-free we have
\begin{align}
\frac{d}{dt}\| \bfw_\lambda \|^2_{H^m} 
  =& \sum_{|\beta| \leq m}\frac{2}{\lambda}\langle \partial^\beta \bfw_\lambda, - \partial^\beta (( \rho + \bfw_\lambda) \cdot \nabla (\rho + \bfw_\lambda) + \bfg)\rangle 
  \notag\\
  =& \sum_{|\beta| \leq m} \frac{2}{\lambda}\langle \partial^\beta \bfw_\lambda, 
    ( \rho + \bfw_\lambda) \cdot \nabla \partial^\beta \bfw_\lambda - \partial^\beta (( \rho + \bfw_\lambda) \cdot \nabla \bfw_\lambda)     \rangle 
  \notag\\
   &+ \langle \bfw_\lambda, \bfg - ( \rho + \bfw_\lambda) \cdot \nabla \rho\rangle_{H^m}
  \notag\\
  :=& \frac{1}{\lambda} (T_1 + T_2).
  \label{eq:eul:sc:1}    
\end{align}    
Using standard Sobolev embeddings and interpolation (see, for example, \cite{MajdaBertozzi2002}), we estimate the first term as follows: 
\begin{align}
  |T_1 | 
  \leq& C\left(\| \rho + \bfw_\lambda \|_{H^m} \|\nabla \bfw_\lambda \|_{L^\infty} 
             + \|\nabla( \rho + \bfw_\lambda) \|_{L^\infty} \|\bfw_\lambda\|_{H^m}\right)\|\bfw_\lambda\|_{H^m}
  \notag\\
  \leq& C (\|\bfw_\lambda\|_{H^m}^3 +  \|\bfw_\lambda\|_{H^m}^2\|\rho\|_{H^m} ) \leq C (\|\bfw_\lambda\|_{H^m}^3 +\|\rho\|_{H^m}^3 ).
  \label{eq:eul:sc:2}    
\end{align}
To estimate $T_2$, using that  $H^m$ is an algebra for $m \geq 2$, we obtain
\begin{align}
 |T_2| \leq&  \|\bfw\|_{H^m}( \| \bfg \|_{H^m} + \| ( \rho + \bfw_\lambda) \cdot \nabla \rho\|_{H^m})
      \notag\\
      \leq& \|\bfw\|_{H^m}( \| \bfg \|_{H^m} + (\| \rho\|_{H^m} + \| \bfw_\lambda\|_{H^m})\|\rho\|_{H^{m+1}})
      \notag\\
      \leq& C(\|\bfw\|_{H^m}^3 + \| \bfg \|_{H^m}^{3/2} + \|\rho\|_{H^{m+1}}^3).  
            \notag
\end{align}
Fixing $T>0$ arbitrary and combining these estimates with \eqref{eq:eul:sc:1}, we conclude
\begin{align}
  \frac{d}{dt} \| \bfw \|^2_{H^m}  
  \leq \frac{C}{\lambda}(\|\bfw\|_{H^m}^3 + \| \bfg \|_{H^m}^{3/2} + \| \bfU_0\|_{H^{m+1}}^3 + T^3 \| \bfh \|_{H^{m+1}}^3)
  \label{eq:lc:bnd:fn:sc:1}
\end{align}
for all $t\in [0, T \wedge \lambda T_{\bfU_0, \lambda \bfh})$.  Here we note carefully that the constant $C$ does not depend on $\lambda > 0$, $\bfU_0, \bfh \in \mathcal{X}$.  

With \eqref{eq:lc:bnd:fn:sc:1} and the criteria \eqref{eq:its:gonna:blow} we
now establish the desired result \eqref{eq:too:comp:to:blow} and \eqref{eq:smth:land:sc:1} 
by invoking the comparison lemma (Lemma~\ref{lem:comp:lem:app}).  
According to \eqref{eq:its:gonna:blow} and Agmond's inequality
\begin{align*}
  \limsup_{s \to \lambda T_{\bfU_0, \lambda \bfh}} \|\bfw_\lambda(s) \|_{H^m} = \infty,
\end{align*}
for every $m \geq 3$.   Thus, noting that $\bfw_\lambda(0) = 0$, Remark~\ref{rem:comp2} implies for all $t\in [0, T]$ and 
\begin{align*}
\lambda \geq C T( \| \bfg\|_{H^m}^{3/2} + \| \bfU_0 \|_{H^{m+1}}^3 + T^3 \|\bfh \|_{H^{m+1}}^3)
\end{align*}
we have $\lambda T_{\bfU_0, \lambda \bfh}\geq T$ as well as the comparison 
\begin{align*}
  \| \bfw_\lambda(t) \|^2_{H^m} \! \!
  \leq (\| \bfg \|_{H^m}^{3/2} + \| \bfU_0\|_{H^{m+1}}^3 + T^3 \| \bfh \|_{H^{m+1}}^3)
     (R_\lambda(t, \| \bfg \|_{H^m}^{3/2} + \| \bfU_0\|_{H^{m+1}}^3 + T^3 \| \bfh \|_{H^{m+1}}^3) \! - 1).
\end{align*}
Here note that $R_\lambda$ is defined in~\eqref{eq:blow:cntrl}. 
Now, for $m \geq 3$ take 
\begin{align*}
  M_m :=  \sup_{\bfU_0 \in K_1, \bfh \in K_2}( \| \bfg \|_{H^m}^{3/2} + \| \bfU_0\|_{H^{m+1}}^3 + T^3 \| \bfh \|_{H^{m+1}}^3).
\end{align*}
Since $K_1, K_2$ are compact subsets of $\mathcal{X}$, $M_m$ is a finite for any $m$.    Thus for $\lambda \geq CT M_m$ we obtain that $\inf_{\bfU_0 \in K_1, \bfh \in K_2} \lambda T_{\bfU_0, \lambda \bfh} \geq T$
and
\begin{align}
  \sup_{\bfU_0 \in K_1, \bfh \in K_2} \|\NS_{\tau/\lambda}^{\lambda \bfh} \bfU_0 - \rho_\tau^\bfh \bfU_0\|_{H^m}^2 \leq M_m (R_\lambda(M_m)-1)
  \label{eq:dont:blow:me}
\end{align}
for every $\tau \in [0, T]$.  Noting that $\limsup_{\lambda \to \infty} M_m(R_\lambda(t, M_m) -1) = 0$ completes the proof.  
\end{proof}

\begin{proof}[Proof of Lemma~\ref{lem:NSsa}]
 Fixing $\bfU_0, \bfh \in \mathcal{X}$ we once again introduce the abbreviated notations
 \begin{align*}
   \bfU_\lambda(\tau) = \rho_{1/\lambda}^{-\lambda^2 \bfh} \, \NS_{\tau/\lambda^2}^0 \, \rho_{1/\lambda}^{\lambda^2 \bfh} \bfU_0,
   \quad
   \bfw_\lambda(\tau) = \bfU_\lambda(\tau) - \rho_\tau^{-B(\bfh, \bfh)} \bfU_0,
   \quad
   \rho(\tau) = \rho^{-B(\bfh, \bfh)}_\tau \bfU_0  
 \end{align*} 
 defined on a interval of existence $[0, \lambda^2 T_{\bfU_0 + \lambda \bfh})$; cf. Proposition~\ref{prop:NSEsemi}.
 Arguing as in (\ref{eq:rs:nLin:twit}) and referring back to \eqref{eq:under:pres} we have
\begin{align}
  \partial_t \bfw_\lambda 
  &=- \frac{1}{\lambda^2} \left( (\bfU_\lambda + \lambda \bfh) \cdot \nabla (\bfU_\lambda + \lambda \bfh )  + \bfg \right) 
                            + \bfh \cdot \nabla \bfh + \nabla p_\lambda \notag\\
  &= -\frac{1}{\lambda^2} ((\bfw_\lambda + \rho) \cdot \nabla (\bfw_\lambda + \rho) + \bfg)
    -  \frac{1}{\lambda} ((\bfw_\lambda + \rho) \cdot \nabla \bfh + \bfh \cdot \nabla (\bfw_\lambda+ \rho))
     + \nabla p_\lambda.
    \label{eq:2nd:sc:E:evol}
\end{align}
with $\nabla \cdot \bfw_\lambda = 0 =\nabla \cdot \rho$.
Here, as in the previous lemma, $p_\lambda: \TT^3 \to \RR$ enforces the divergence-free condition.

We now make estimates for the $H^m$ norm of $\bfw_\lambda$ for $m \geq 3$.  Taking derivatives of \eqref{eq:2nd:sc:E:evol},
then $L^2$ inner products and summing over multi-indies $|\beta| \leq m$ we find that
\begin{align}
   \frac{1}{2} \frac{d}{dt} \| \bfw_\lambda \|^2 =&
     \frac{1}{\lambda^2} \sum_{| \beta | \leq m } 
          \langle (\bfw_\lambda + \rho) \cdot \nabla  \partial^\beta \bfw_\lambda  
                   -  \partial^\beta [(\bfw_\lambda + \rho) \cdot \nabla (\bfw_\lambda + \rho)  + \bfg] ,
                   \partial^\beta   \bfw_\lambda \rangle 
                                                    \notag\\
     &+  \frac{1}{\lambda} \sum_{| \beta | \leq m } 
                  \langle  \bfh \cdot \nabla \partial^\beta \bfw_\lambda 
                                     - \partial^\beta[\bfh \cdot \nabla (\bfw_\lambda + \rho) +(\bfw_\lambda+ \rho) \cdot \nabla \bfh],
                   \partial^\beta   \bfw_\lambda \rangle
       \notag\\
      =& \frac{1}{\lambda^2} T_1 + \frac{1}{\lambda} T_2.
  \label{eq:Hm:2nd:sc:E}
\end{align}
Note that we have used the fact that $\bfw_\lambda$ is divergence free to obtain the commutator terms.
With commutator estimates similar to \eqref{eq:eul:sc:2} above we find that for any $m \geq 3$
\begin{align*}
    |T_1|
  \leq&  
      C (\|\bfw_\lambda + \rho  \|_{H^m} \|\bfw_\lambda\|^2_{H^m} 
        +  \|\bfw_\lambda + \rho  \|_{H^m}\|\rho\|_{H^{m+1}} \|\bfw_\lambda\|_{H^m} + \|\bfg\|_{H^m} \|\bfw_\lambda \|_{H^m} )\\
  \leq& C (\|\bfw_\lambda\|^3_{H^m} + \|\rho\|^3_{H^{m+1}} + \|\bfg\|_{H^m}^{3/2}   + 1).  
\end{align*}
Likewise we have
\begin{align*}
  |T_2| 
  \leq& C(\|\bfh \|_{H^{m+1}} \|\bfw_\lambda \|_{H^m}^2 
        + \|\bfh \|_{H^{m+1}}\|\rho \|_{H^{m+1}}\|\bfw_\lambda \|_{H^m} )\\
  \leq& C (\|\bfw_\lambda\|^3_{H^m} + (\| \rho\|_{H^{m+1}}^{3/2} +1)\| \bfh\|_{H^{m+1}}^{3/2})
\end{align*}
Note that as above in the previous lemma the constant $C$ depends only on $m$ universal quantities and is independent of 
$\lambda > 0$.

Combining these bound with \eqref{eq:Hm:2nd:sc:E} yields
\begin{align*}
  \frac{d}{dt} \| \bfw_\lambda \|^2 
  \leq \frac{C}{\lambda}( \|\bfw_\lambda \|_{H^m}^3  +  \|\rho\|^3_{H^{m+1}} + \|\bfg\|_{H^m}^{3/2}   + (\| \rho\|_{H^{m+1}}^{3/2} +1)\| \bfh\|_{H^{m+1}}^{3/2}+ 1).
\end{align*}
which is valid on $[0, \lambda^2 T_{\bfU_0 + \lambda \bfh})$ where we again emphasize that the constant $C$ does not depend on $\lambda > 0$ or $\bfU_0, \bfh \in \mathcal{X}$.
Repeating the arguments from the analogous bound 
\eqref{eq:lc:bnd:fn:sc:1} in the proof of the previous Lemma, 
yields the desired result.
\end{proof}

\appendix

\section{Supplemental PDE Bounds}

\subsection{A Priori Estimates}
\label{sec:a priori:est}
Here we present a collection of a priori estimates which assure that the solution maps in each equation have the necessary cocycle and semigroup structures.  We begin with the a priori estimates for the reaction-diffusion equations~\eqref{eq:RD:cont}.

\subsubsection{Reaction-Diffusion}
\label{sec:apriori:R-D}

Recall that for $V\in \Omega$, we define solutions $u=u(t, u_0, V)$ with $u(0)=u_0$ of~\eqref{eq:RD:cont} by $u(t, u_0, V) = v(t, u_0, \sigma \cdot V) + \sigma \cdot V$ where $v$ satisfies the shifted equation~\eqref{eq:RD:shift}.  In order to make the estimates more legible, for $k\geq 0, T>0$ we introduce the sup norms $$|V|_{k,T}= \sup_{t \in [0, T]}\|V(\ccdot,t)\|_{W^{k, \infty}([0, 2\pi])}.$$  

Proposition~\ref{prop:contco:rd} follows immediately once we establish:

\begin{proposition}
\label{prop:apest:rd}
We have the following estimates.
\begin{itemize}
\item[(a)]  Let $u_0 \in L^2$, $T>0$, $V\in \Omega$ and $v(\ccdot)=v(\ccdot, u_0, \sigma \cdot V)$.  Then there exists a constant $C_1>0$ depending only on $T, \|u_0\|, | \sigma \cdot V|_{2, T}$ such that for any $t\in [0, T]$ 
\begin{align}
\label{apest1:rd}
\| v(t)\|^2 + 2\kappa \int_0^t \| \partial_x v(s)\|^2 ds + \nu \int_0^t \| v(s)\|^{2n} \, ds \leq C_1 .  
\end{align}
\item[(b)]  Let $u_1, u_2\in L^2$, $T>0$, $V_1, V_2 \in \Omega$ and set $w(t)= v(t, u_1,\sigma \cdot V_1) - v(t, u_2, \sigma\cdot V_2)$, $w_0=u_1-v_1$ and $\bar{V}= V_1-V_2$.  Then there exists a constant $C_2>0$ depending only on  $T, \|u_i\|, | \sigma \cdot V_i|_{2, T}$ such that for any $t\in [0, T]$
\begin{align}
\label{apest2:rd}
&\| w(t) \|^2 + 2 \kappa \int_0^t \| \partial_x w (s)\|^2 ds \leq C_2( \| w_0 \|^2 + |\partial_{xx} (\sigma \cdot \bar{V})|_{T}^2 + |\sigma\cdot \bar{V}|_{T} ).   
\end{align}
\end{itemize}
\end{proposition}

\begin{proof}[Proof of Proposition~\ref{prop:apest:rd}]
To obtain the first estimate~\eqref{apest1:rd}, observe that there exists a constant $K=K(|\sigma \cdot V |_{T})>0$ such that 
\begin{align}
\nonumber \frac{1}{2}\frac{d}{dt} \| v(t) \|^2  &\leq - \kappa \| \partial_x v(t)\|^2 +  \kappa \| v(t) \| \| \partial_{xx} (\sigma \cdot V)\| + \langle v, f(v+\sigma\cdot V)\rangle\\
\label{eqn1:rd:apest1}&\leq - \kappa \| \partial_x v(t)\|^2 +  2\pi
                        \kappa  \| v(t)\| | \partial_{xx} (\sigma
                        \cdot  V)|_{T} +  K - \frac{\nu}{2} \| v(t) \|^{2n}.
\end{align}
Hence we have that 
\begin{align}
 \nonumber  \frac{1}{2}\frac{d}{dt} \| v(t) \|^2  & \leq  K_1 \|v(t)\|^2 + K_2
\end{align}
for some constants $K_1>0$ and $K_2= K_2(|\sigma \cdot V|_{2,T} )>0.$  Gronwall's inequality then implies the existence of a constant $K_3$ depending only on $T$, $\|u_0\|$,  $| \sigma \cdot V|_{2,T}$ such that   for all $t\in [0, T]$ 
\begin{align}
\label{eqn2:rd:apest1}
\|v(t) \|^2 \leq  K_3. 
\end{align}
Integrating~\eqref{eqn1:rd:apest1} with respect to time and using the estimate~\eqref{eqn2:rd:apest1} we arrive at the bound~\eqref{apest1:rd}.   

Turning our attention to the second estimate~\eqref{apest2:rd}, note that for $t\in [0, T]$ 
\begin{align}
\label{eqn1:rd:apest2}
\frac{1}{2}\frac{d}{dt} \| w(t) \|^2  &\leq - \kappa \| \partial_x w(t) \|^2 +  2\pi  \kappa \|w(t)\| | \partial_{xx}( \sigma \cdot \bar{V}) |_{T} \\
\nonumber &\qquad + \kappa \langle w, f(v_1 + \sigma \cdot V_1) - f(v_2 + \sigma \cdot V_2) \rangle.
\end{align}
To estimate the last term above, recall the explicit form form of the Mean Value Theorem applied to the polynomial $f$: For $a, b \in \R$ we have 
\begin{align}
\label{eqn:mvt}
f(b)-f(a)=(b-a) f'(\xi)=  (b-a) \int_0^1 f'(a + \beta (b-a) ) \, d\beta 
\end{align}  
for some $\xi =\xi(a,b) $ lying between $a, b$.  Hence since $f' \leq K$ for some constant $K>0$ only depending on $f$ we find that 
\begin{align}
\nonumber \langle w , f(v_1 + V_1) - f(v_2+V_2) \rangle &= \langle w, f'(\xi) w \rangle + \langle w, f'(\xi) (\sigma \cdot \bar{V}) \rangle\\
\label{eqn2:rd:apest2}&\leq K \|w\|^2 + | \sigma \cdot \bar{V}|_{0,T} \int_0^{2\pi }  |w| |f'(\xi)| \, dx.     
\end{align}
Integrating~\eqref{eqn1:rd:apest2} with respect to time using the bound~\eqref{eqn2:rd:apest2} and applying Young's inequality we obtain the estimate
\begin{align*}
&\frac{1}{2}\| w(t)\|^2 + \kappa \int_0^t \| \partial_x w(s)\|^2 \, ds \\
& \leq \| w_0 \|^2 + K_1 \int_0^t \|w(s)\|^2 \, ds + K_2 | \sigma \cdot \bar{V}|_{2,T}^2  + |\sigma \cdot \bar{V}|_{0,T} \| w\|_{L^n([0, 2\pi] \times [0, t])}  \| f'(\xi) \|_{L^{p}([0, 2\pi ]\times [0, t])}
\end{align*}
where $p=2n/(2n-2)$, for some constants $K_1, K_2>0$.  Applying the estimate~\eqref{apest1:rd} to the last term above using the explicit form for $f'(\xi)=f'(\xi(v_1 + \sigma \cdot V_1, v_2 + \sigma \cdot V_2))$, we determine the existence of a constant $K_3$ depending only on $T, \|u_i\|, | \sigma \cdot V_i|_{2,T}$
such that 
\begin{align}
\label{eqn3:rd:apest2}
&\frac{1}{2}\| w(t)\|^2 + \kappa \int_0^t \| \partial_x w(s)\|^2 \, ds \leq \| w_0 \|^2 + K_1 \int_0^t \|w(s)\|^2 \, ds +  K_2 | \sigma \cdot \bar{V}|_{2,T}^2  +  K_3 |\sigma \cdot \bar{V}|_{0,T}.   
\end{align}
From this, using Gronwall's inquality we arrive at the claimed estimate~\eqref{apest2:rd} when combined with~\eqref{eqn3:rd:apest2}.

\end{proof}

\subsection{Boussinesq Equations}
\label{sec:beq:aprioriproof}

We now provide the needed a priori estimates for the Boussinesq equations \eqref{eq:b:1}--\eqref{eq:b:2}.  We begin by establishing the $L^2$ estimates below in Proposition~\ref{prop:b:apest12} for the shifted equation~\eqref{eq:sb:1}-\eqref{eq:sb:2} so that the $\phi$ defined in the statement of Proposition~\ref{prop:b:structure} is a continuous adapted cocycle.  For $k\geq 0$ and $T>0$, we again use compact notation for sup norms, which in this context will read 
\begin{align*}
|V|_{k, T} = \sup_{t\in [0,T]} \{ \| V(\cdot, t) \|_{W^{k, \infty}(\TT^2)}\}.  
\end{align*}

\begin{proposition} 
\label{prop:b:apest12}
We have the following:
\begin{itemize}
\item[(1)]  Fix $T>0$, $\tth(0), \tom(0) \in L^2$ and $V\in \Omega$.  Then for all $t\in [0, T]$ we have that 
\begin{align}
\label{eqn:apest1:b}
 \|\tth(t)\|^2 + \|\tom(t) \|^2+ \int_0^t \nu \| \nabla \tom(s) \|^2 + \kappa  \| \nabla \tth(s) \| \, ds \leq C
\end{align}
where $C>0$ is a constant depending only on $\|\tom(0)\|, \| \tth(0) \|, T, \kappa, g, | \sigma \cdot V|_{2, T}, \|h^0 \|$.  
\item[(2)]  Let $(\tom_1, \tth_1, V_1)$ and $(\tom_2, \tth_2, V_2)$ solve~\eqref{eq:sb:1}-\eqref{eq:sb:2} with $\tom_i(0), \tth_i(0)\in L^2$ and $V_i \in \Omega$.  Then if $\bar{\xi}= \tom_1-\tom_2$, $\bar{\Th}= \tth_1-\tth_2$, $T>0$, $\bar{V}=V_1-V_2$, we have for $t\in [0, T]$
\begin{align}
\label{eqn:apest2:b}
\| \bar{\xi}(t)\|^2 + \|\bar{\Th}(t) \|^2 \leq C \big( \| \bar{\xi}(0)\|^2 + \|\bar{\Th}(0) \|^2 + |\sigma \cdot\bar{V}|_{1,T}\big)   
\end{align}
for some constant $C>0$ where $C$ depends only on $\|\tom_i(0)\|, \| \tth_i(0) \|, T, \kappa, g, |\sigma\cdot V_i|_{2,T},  \|h^0 \|$.  \end{itemize}
\end{proposition}

\begin{proof}
We begin by establishing the bound~\eqref{eqn:apest1:b}.  Let $T>0$ and $t\in [0, T]$. First observe that 
\begin{align*}
&\frac{1}{2}\frac{d}{dt} \| \tom\|^2 + \nu \| \nabla \tom\|^2 \leq g \| \tom\| \| \partial_x \Th \| + 4\pi^2 g \|\tom \| |  \sigma\cdot V|_{1,T}\leq C_1 \|\tom\|^2 + \frac{\kappa}{2}\| \partial_x \tth\|^2 + \frac{\kappa}{2}| \sigma \cdot V |^2_{1,T}
\end{align*}
for some constant $C_1>0$ depending only on $g, \kappa$.  Also note that
\begin{align*}
 \frac{1}{2}\frac{d}{dt}\| \tth \|^2 + \kappa \| \nabla \tth \|^2 &\leq C_2\| \tth \| \|\tilde{u} \| |\sigma\cdot V|_{1,T} + \kappa \| \tth \| \| \sigma\cdot V\|_{H^2} + \|\tth \| \| h^0\|\\
& \leq C_3( \|\tth\|^2 + \|\tom \|^2 | \sigma\cdot V|^2_{1,T} + \|h^0 \|^2 + | \sigma\cdot V|_{2, T})
\end{align*}
for some constant $C_3=C_3(\kappa)>0$.  Summing the previous two inequalities we obtain 
\begin{align}
\label{eqn:apest1:b:est1}
\frac{1}{2} \frac{d}{dt}( \| \tom\|^2 + \|\tth\|^2)  + \frac{\nu}{2} \|\nabla \tom \|^2 + \frac{\kappa}{2} \| \nabla \tth \|^2 &\leq C( \|\tth \|^2 + \| \tom \|^2) + D
\end{align}
for some constants $C=C(\kappa, g, |\sigma\cdot V|_{1,T})>0$ and $D=D(\kappa, \|h^0 \|, |\sigma\cdot V|_{2,T})>0$.  Applying Gronwall's inequality we obtain
\begin{align*}
\| \tom(t) \|^2 + \|\tth(t) \|^2 \leq C
\end{align*}  
for all $t\in [0, T]$ where $C>0$ is a constant depending only on $\|\tom(0)\|, \| \tth(0) \|, T, \kappa, g, \|h^0 \|, | \sigma\cdot V|_{2,T}$.  Plugging this back into the righthand side of~\eqref{eqn:apest1:b:est1} gives the desired bound~\eqref{eqn:apest1:b}.   

Moving onto the second estimate~\eqref{eqn:apest2:b}, first note that since $\langle\bar{\xi}, \tilde{u}_1 \cdot \nabla \bar{\xi} \rangle=0$
\begin{align*}
&\frac{1}{2}\frac{d}{dt} \| \bar{\xi}\|^2 + \nu \| \nabla \bar{\xi}\|^2 \leq g \| \bar{\xi}\| \| \partial_x \bar{\Th}\| + g \| \bar{\xi}\| \| \partial_x (\sigma \cdot \bar{V}) \|- \langle\bar{\xi}, (\tilde{u}_1 - \tilde{u}_2) \cdot \nabla \tom_2  \rangle,\\
&\frac{1}{2}\frac{d}{dt} \| \bar{\Th}\|^2 + \kappa \| \nabla \bar{\Th}\|^2  \leq  \langle \bar{\Th}, \tilde{u}_2 \cdot \nabla (\tth_2+ \sigma\cdot V_2) - \tilde{u}_1 \cdot \nabla (\tth_1 + \sigma\cdot V_1)\rangle + \kappa \| \bar{\theta}\| \| \sigma \cdot \bar{V}\|_{H^2} 
\end{align*}
To combine the estimates as we did above, we now bound the terms remaining in inner product form as follows:
\begin{align*}
- \langle\bar{\xi}, (\tilde{u}_1 - \tilde{u}_2) \cdot \nabla \tom_2  \rangle = \langle (\tilde{u}_1 - \tilde{u}_2) \cdot \nabla \bar{\xi}, \tom_2 \rangle & \leq C_1\| \nabla \bar{\xi}\| \| \tilde{u}_1 - \tilde{u}_2\|_{L^4} \| \tom_2 \|_{L^4}\\
& \leq C_2 \| \nabla \bar{\xi} \| \| \bar{\xi} \| \| \nabla \tom_2 \|^{1/2} \| \tom_2 \|^{1/2}\\
& \leq \frac{\nu}{2} \| \nabla \bar{\xi}\|^2 + C'\| \bar{\xi}\|^2 \| \nabla \tom_2\| \|\tom_2 \|,  
\end{align*}
for some constant $C'>0$ depending on $\nu$, and in a similar fashion  
\begin{align*}
\langle \bar{\Th}, \tilde{u}_2 \cdot \nabla (\tth_2+ \sigma\cdot V_2) - \tilde{u}_1 \cdot \nabla (\tth_1 + \sigma\cdot V_1)\rangle &=- \langle (\tilde{u}_2- \tilde{u}_1) \cdot \nabla \bar{\Th}, \tth_2 + \sigma\cdot V_2 \rangle + \langle \bar{\theta}, \tilde{u}_1 \cdot \nabla (\sigma \cdot \bar{V})\rangle\\ 
& \leq \frac{\kappa}{4}\| \nabla \bar{\Th}\|^2 + C'' \| \bar{\xi}\|^2 \| \nabla (\tth_2 + \sigma\cdot V_2) \| \| \tth_2 + \sigma\cdot V_2 \|\\
& \qquad + C''' \|\bar{\Th}\| \| \tom_1 \| |\nabla (\sigma \cdot \bar{V}) |_T
\end{align*}
for some constant $C''>0$ depending only on $\kappa$ and some constant $C'''>0$.  Thus by summing the first two inequalities, applying the inequalities above and weighting appropriately using Young's inequality we obtain 
\begin{align*}
&\frac{1}{2} \frac{d}{dt}( \| \bar{\xi}\|^2 + \| \bar{\Th}\|^2) + \frac{\nu}{2}\| \nabla \bar{\xi}\|^2 + \frac{\kappa}{2}\| \nabla \bar{\Th}\|^2 \leq C f [\| \bar{\xi}\|^2  +  \| \bar{\Th}\|^2 ]+ D \|\tom_1 \|^2 |\sigma \cdot \bar{V} |_{1,T}^2  
\end{align*}  
for some constants $C, D>0$ depending only on $\kappa, \nu, g$ and
\begin{align*}
f = 1+ \| \nabla \tom_2 \|^2 + \| \tom_2 \|^2   + \| \nabla (\tth_2+ \sigma\cdot V_2)\|^2+  \| \tth_2 + \sigma\cdot V_2 \|^2.  
\end{align*}
Gronwall's inequality then implies that for $t\in [0, T]$ 
\begin{align*}
\| \bar{\xi}(t)\|^2 + \|\bar{\Th}(t) \|^2 \leq \bigg( \| \bar{\xi}(0)\|^2 + \|\bar{\Th}(0) \|^2 + D |\sigma \cdot \bar{V}|_{1,T} \int_0^t \|\tom_1(s)\|^2  \,ds\bigg)\exp\bigg(\int_0^t C f(s) \, ds \bigg).   
\end{align*}  
Applying the first inequality~\eqref{eqn:apest1:b} to estimate $\int_0^t \| \tilde{\xi}_1(s)\|^2 \, ds$ and $\int_0^t f(s) \, ds$, we obtain the desired inequality.

\end{proof}

We next turn our attention to the a priori estimates needed to validate Assumption~\ref{ass:Mal:der}.  Here it will be convenient to express the system~\eqref{eq:b:1}-\eqref{eq:b:2} using the abstract evolution equation notation for the solution $U=(\xi(t), \theta(t))$:
\begin{align}
\label{eqn:buca:abs}
\frac{dU}{dt}+ AU + GU+ B(U,U) = \iota_\theta h^0 + \iota_\theta(\sigma \cdot \partial_t V), 
\end{align}  
which we recall was introduced above equation~\eqref{eq:be:abs}.  Following Remark~\ref{rmk:du:ham}, our principal interest will be in establishing estimates for the linear equation 
\begin{align}
\label{eqn:lin:buca}
\partial_t \rho + A\rho + G \rho + B(U, \rho) + B(\rho, U) = 0, \,\, \rho(s)=\rho_0 
\end{align}  
where $\rho_0 \in H$ and $U$ solves~\eqref{eqn:buca:abs}.  To do so, we will make use of the following inequalities for $u,v, w\in H$ 
\begin{align}
\label{eqn:buca:non1}
|\langle u, B(v, w)\rangle| &\leq C \| \nabla u \|^{1/2} \|u\|^{1/2} \|\nabla (K * v) \|^{1/2} \| K*v \|^{1/2} \| \nabla w\| \\
\label{eqn:buca:non2}& \leq C_1 \|\nabla u \| \|u\| + C_2 \|v\|^2 \| \nabla w\|^2
\end{align}
where $C, C_1, C_2>0$ are constants.  The first inequality~\eqref{eqn:buca:non1} is $L^4$-$L^4$-$L^2$ bound followed by an application of the Gagliardo-Nirenberg interpolation inequality.  The second~\eqref{eqn:buca:non2} is simply Young's inequality applied to the righthand side of ~\eqref{eqn:buca:non1}.

\begin{proposition}
    We have the following:
\begin{itemize}
\item[(1)]  Fix $T>0$, $\rho_0 \in H$ and let $\rho=(\rho_1, \rho_2)$ solve~\eqref{eqn:lin:buca} with $\rho(0)=\rho_0$ and corresponding $U$ with $U(0)=U_0\in H$.  Then there exists a constant $C>0$ depending only on $\|U_0 \|, T, \kappa, \nu, g, |\sigma \cdot V|_{2,T}, \|h^0\|$ such that for all $0\leq s\leq t \leq T$
\begin{align}
\label{eqn:lin:bucae1}
\| \rho(t) \|^2 + \int_s^t \nu \| \nabla \rho_1 (v) \|^2 + \kappa \| \nabla \rho_2(v) \|^2 \, dv \leq C. 
\end{align}   
\item[(2)]  Fix $T>0$ and let $U_1, U_2$ solve~\eqref{eqn:buca:abs} with corresponding initial data $U_1(0), U_2(0) \in H$ and corresponding $V_1, V_2 \in \Omega$.  Here we assume that $U_1$ and $U_2$ solve~\eqref{eqn:buca:abs} with the same $h^0\in L^2$.  Let $\rho_1, \rho_2 $ solve~\eqref{eqn:lin:buca} with corresponding data $\rho_1(s), \rho_2(s) \in H$ and corresponding $U_1, U_2$.  Set $\bar{\rho}_0= \rho_1(s) - \rho_2(s)$, and $\bar{U}=U_1 - U_2$.  Then there exists a constant $C>0$ depending only on $T, \| U_i(0)\|, \kappa, \nu, g, |\sigma \cdot V_i|_{C^2, T}, \|h^0\|$ such that for all $0\leq s\leq t\leq T$ 
\begin{align}
\|\bar{\rho}(t)\|^2 \leq C( \| \bar{\rho}_0\|^2 + \| \bar{U}(0)\|^2 + | \sigma \cdot \bar{V}|_{C^1, T}).  
\end{align}
\end{itemize}
\end{proposition}

\begin{proof}
We begin by establishing (1).  Observe that 
\begin{align*}
\frac{1}{2}\frac{d}{dt}\| \rho(t) \|^2 + \nu \| \nabla \rho_1(t) \|^2 + \kappa \|\nabla \rho_2(t) \|^2 + \langle \rho, G \rho \rangle + \langle \rho, B(\rho, U)\rangle=0    
\end{align*}  
and 
\begin{align*}
|\langle \rho, G \rho \rangle| = g |\langle \rho_1, \partial_x \rho_2\rangle | \leq \frac{\kappa}{2}\| \nabla \rho_2 \|^2 + C\|\rho\|^2
\end{align*}
for some constant $C>0$ depending only on $g, \kappa$.  Applying the inequality~\eqref{eqn:buca:non2}, we also find that 
\begin{align*}
| \langle \rho, B(\rho, U)\rangle| &\leq C_1 \| \nabla \rho\| \|\rho\| + C_2 \| \rho \|^2 \|\nabla U \|^2 \leq \frac{\nu \wedge \kappa}{4} \| \nabla \rho \|^2 + C (1+ \|\nabla U \|^2)  \| \rho\|^2 
\end{align*}  
for some constant $C>0$.  Putting these estimates together produces the bound
\begin{align}
\label{eqn:buca:line1b1}
\frac{1}{2}\frac{d}{dt}\| \rho(t) \|^2 + \frac{\nu}{4} \| \nabla \rho_1(t) \|^2 +\frac{\kappa}{4} \|\nabla \rho_2(t) \|^2 \leq C( 1+ \|\nabla U\|^2) \| \rho\|^2.  
\end{align}
Since we also have that 
\begin{align*}
\frac{1}{2}\frac{d}{dt}\| \rho(t) \|^2 \leq C( 1+ \|\nabla U\|^2) \| \rho\|^2,  
\end{align*}
applying Gronwall's inequality and then Proposition~\ref{prop:b:apest12} implies
\begin{align*}
\| \rho(t)\|^2 \leq \| \rho_0\|^2  \exp\bigg( \int_s^t 2 C ( 1+ \|\nabla U(v)\|^2 \, dv\bigg)\leq C
\end{align*}
for all $0\leq s\leq t \leq T$ where $C>0$ is a constant depending only on $\|U_0 \|, T, \kappa, \nu, g, |\sigma \cdot V|_{2,T}, \|h^0\|. $  Using the information on the righthand side of equation~\eqref{eqn:buca:line1b1}, integrating with respect to time, and then applying Proposition~\ref{prop:b:apest12} again we arrive at the estimate in (1). 

To see (2), note that 
\begin{align*}
0=&\frac{1}{2}\frac{d}{dt}\| \bar{\rho}(t) \|^2 + \nu \| \nabla \bar{\rho}_1 \|^2 + \kappa \| \nabla \bar{\rho}_2\|^2 + \langle \bar{\rho}, G \bar{\rho}\rangle\\
&\qquad + \langle \bar{\rho}, B(\rho_1, U_1)- B(\rho_2, U_2)\rangle+ \langle \bar{\rho}, B(U_1, \rho_1)- B(U_2, \rho_2)\rangle.
\end{align*}
We can again bound $\langle \bar{\rho}, G \bar{\rho}\rangle$ as follows: 
\begin{align*}
|\langle \bar{\rho}, G \bar{\rho}\rangle| \leq \frac{\kappa}{2}\| \nabla \bar{\rho}_2 ||^2 + C \|\bar{\rho}\|^2 
\end{align*}
for some constant $C>0$ depending only on $g, \kappa$.  Using bilinearity and~\eqref{eqn:buca:non2}, also observe that 
\begin{align*}
 |\langle \bar{\rho}, B(\rho_1, U_1)- B(\rho_2, U_2)\rangle|&\leq | \langle \bar{\rho}, B(\bar{\rho}, U_1)\rangle| + |\langle \bar{\rho}, B(\rho_2, \bar{U})\rangle|\\
 &\leq C_1 \| \nabla \bar{\rho} \| \| \bar{\rho} \| + C_2 \| \bar{\rho}\|^2 \| \nabla U_1 \|^2 + C_3 \| \rho_2 \|^2 \| \nabla \bar{U}\|^2  
 \end{align*}  
 and 
 \begin{align*}
 | \langle \bar{\rho}, B(U_1, \rho_1)- B(U_2, \rho_2)\rangle| &  = | \langle  \bar{\rho}, B(\bar{U}, \rho_2)\rangle|\\
 & \leq C_4 \| \nabla \bar{\rho} \| \| \bar{\rho} \| + C_5 \|\bar{U} \|^2 \| \nabla \rho_2\|^2 
 \end{align*}
 for some constants $C_i>0$.  Combining these estimates and applying Young's inequality to the terms $C_1 \| \nabla \bar{\rho} \| \| \bar{\rho}\|$ and $C_4 \| \nabla \bar{\rho} \| \| \bar{\rho}\|$ we find that 
 \begin{align}
\label{eqn:buca:fb0} &\frac{1}{2}\frac{d}{dt}\| \bar{\rho}(t) \|^2 + \frac{\nu}{4}\| \nabla \bar{\rho}_1 \|^2 + \frac{\kappa}{4}\| \nabla \bar{\rho}_2 \|^2 \\
 \nonumber & \leq C(1+ \| \nabla U_1\|^2) \| \bar{\rho}\|^2 + C_3 \| \rho_2\|^2 \| \nabla \bar{U}\|^2 + C_5 \| \bar{U}\|^2 \| \nabla \rho_2 \|^2.   
 \end{align} 
By Proposition~\ref{prop:b:apest12}, we note that 
\begin{align}
\label{eqn:buca:fb1}
\| \bar{U}\|^2 \leq C( \| \bar{U}(0)\|^2 + | \sigma \cdot \bar{V}|^2_{1,T})
\end{align}
 where $C>0$ is a constant depending only on $T, \| U_i(0) \|, \kappa, g, | \sigma \cdot V_i |_{C^2, T}, \|h^0\|$.  By the first part of this proposition, we also have that 
 \begin{align}
 \label{eqn:buca:fb2}
 \| \rho_2 \|^2 \leq C'
 \end{align}  
 where $C'>0$ is a constant depending only on $\|U_2(0)\|, T, \kappa, \nu, \kappa, g, |\sigma \cdot V_2|_{C^2, T}, \|h^0\|$.  Applying the inequalities~\eqref{eqn:buca:fb1}-\eqref{eqn:buca:fb2} to the righthand side of~\eqref{eqn:buca:fb0} and then applying Gronwall's inequality produces the estimate     
 \begin{align}
 \| \bar{\rho}(t)\|^2 &\leq \bigg( \|\bar{\rho}_0 \|^2 + C \int_s^t \| \nabla \bar{U}(v)\|^2 \bigg)\exp\bigg(C\int_s^t 1+ \| \nabla U_1(v)\|^2 \, dv \bigg)\\
 \nonumber & + C(\| \bar{U}(0)\|^2 + |\sigma \cdot \bar{V}|_{1, T}) \int_s^t \|U_1(v)\|^2 \, dv  \bigg)  \exp\bigg(C\int_s^t 1+ \| \nabla U_1(v)\|^2 \, dv \bigg)  \end{align}
 where $C>0$ is a constant depending only on $T, \|U_i(0)\|, \kappa, \nu, g, |\sigma \cdot V_i|_{C^2, T}, \|h^0\|$.  Applying Proposition~\eqref{prop:b:apest12} again, we arrive at the claimed bound in (2).  
\end{proof}  

All parts of Assumption~\ref{ass:Mal:der} follow from the above proposition except (v) which concerns the non-degeneracy of the $L^2$-adjoint of the Jacobi flow.  This, however, can be established by following a nearly identical process to the one used in the case of the two-dimensional Navier-Stokes equations as in Proposition~2.2 of \cite{MattinglyPardoux06}.  There, non-degeneracy follows by uniqueness of the associated backwards PDE satisfied by the adjoint.

Finally, we establish the higher-order Sobolev a priori estimates for the Boussinesq equations~\eqref{eq:b:1}-\eqref{eq:b:2} when forced by a smoother $V$; that is, we now consider the equations
\begin{align}
\label{eq:bprime:1}
\partial_t \xi + u  \cdot \nabla \xi - \nu \Delta \xi &= g \partial_x \theta, \,\, \, \xi(0)= \xi_0\\
\label{eq:bprime:2}\partial_t \theta + u \cdot \nabla \theta - \kappa \Delta \theta &= f, \,\,\, \theta(0)=\theta_0
\end{align}
where $f$ is a generic constant element in the relevant Sobolev space.  Note that the only difference between equations~\eqref{eq:b:1}-\eqref{eq:b:2} and the equations above is that the forcing term $h^0+ \sigma \cdot \partial_t V$ has replaced by $f$.

\begin{proposition}\label{prop:a:pr:est}
We have the following:
\begin{itemize}
\item[(i)] Suppose that $\Vort_0, \Th_0, f \in L^2$ and let $(\Vort, \Th)$ be the corresponding solution of \eqref{eq:bprime:1}--\eqref{eq:bprime:2}.  Then
\begin{align}
	\sup_{r \in [0,t]} (\| \Vort (r) \| +  \| \Th(r) \|) 
		\leq  C (\| \Vort_0\| + \| \Th_0 \| +  t\| f\|)
	\label{eq:L2:bnd:1}
\end{align}
and 
\begin{align}
  \int_0^t  (\| \nabla \Vort  \|^2 +  \| \nabla \Th \|^2) dr \leq  C (\| \Vort_0\|^2 + \| \Th_0 \|^2 +  t^2\| f\|^2)
  \label{eq:L2:bnd:2}
\end{align}
where the constant $C$ depends only $\kappa, \nu, g$ and universal quantities.
\item[(ii)] Suppose that $\Vort_0, \Th_0, f \in H^m$ for any $m \geq 1$.  Then
\begin{align}
		\label{eq:Hs:bnd}
&\sup_{r \in [0,t]}  (\| \Vort (r) \|_{H^m} +  \| \Th(r) \|_{H^m}) 
		+ \int_0^t (\| \Vort \|_{H^{m+1}} +  \| \Th \|_{H^{m+1}})  dr\\
	\nonumber &\leq  C\exp\left( C (\| \Vort_0\|^2 + \| \Th_0 \|^2 +  t^2\| h\|^2 + t)\right) (1+ \| \Vort_0\|_{H^m} + \| \Th_0 \|_{H^m} +  t\| f\|_{H^m}).
\end{align}
\item[(iii)]  Fix any $m \geq 0$ and suppose $U_0=(\Vort_0, \Th_0), \tilde{U}_0=(\tilde{\Vort}_0, \tilde{\Th}_0)\in H^m(\TT^2)^2$ and $f, \tilde{f} \in H^m(\TT^2)$.  Let $(\Vort, \Th)$, $(\tilde{\Vort}, \tilde{\Th})$ be the solutions
of \eqref{eq:bprime:1}--\eqref{eq:bprime:2} the corresponding to this data. Then
\begin{align}
       \label{eq:c:d:est:Hs}
 & \sup_{r \in [0,t]} (\| \Vort(r) - \tilde{\Vort}(r) \|_{H^m} +  \| \Th(r) - \tilde{\Th}(r) \|_{H^m})\\
 \nonumber   &\qquad \leq C
      \left(\| \Vort_0 - \tilde{\Vort}_0 \|_{H^m} +  \| \Th_0 - \tilde{\Th}_0 \|_{H^m} + t \| f - \tilde{f} \|_{H^m}\right). 
\end{align}
where $C$ is a constant depending only on  $\kappa, \nu, g, \|U_0\|_{H^m} ,  \| \tilde{U}_0\|_{H^m}, \|f\|_{H^m}, \| \tilde{f}\|_{H^m}, t$ and universal quantities. \end{itemize}
\end{proposition}
\begin{proof}
We begin with the basic $L^2$ estimates for \eqref{eq:bprime:1}--\eqref{eq:bprime:2}.  Multiplying the first equation by $\Vort$, the second equation by $\Th$
and integrating over the domain yields
\begin{align*}
  &\frac{1}{2}  \frac{d}{dt} \| \Vort \|^2 + \nu \| \nabla \Vort\|^2 = \langle  g \pd_x \Th,  \Vort \rangle \leq   \frac{g^2}{2\nu} \| \Th \|^2  + \frac{\nu}{2} \| \nabla \Vort\|^2,
  \end{align*}
  and
  \begin{align*}
  &\frac{1}{2}  \frac{d}{dt} \| \Th \|^2 + \kappa \| \nabla \Th\|^2 = \langle  f,  \Th \rangle \leq \| f\|\| \Th\|.
\end{align*}
The fact that the velocity $u$ is divergence free justifies dropping the non-linear contributions in the above.  Suitably weighting and then adding these two inequalities we
find
\begin{align}
 \frac{d}{dt} \left( \| \Vort \|^2 + \frac{g^2}{2\nu \kappa} \| \Th \|^2  \right) + \nu \| \nabla \Vort\|^2 +  \frac{g^2}{\nu} \| \nabla \Th\|^2 \leq \frac{g^2}{\nu \kappa}\| f\|\| \Th\|.
 \label{eq:L2:en:bd}
\end{align}
The first item, \eqref{eq:L2:bnd:1}, follows immediately.    Moreover 
\begin{align*}
 \int_0^t (\| \nabla \Vort\|^2 +  \| \nabla \Th\|^2) dr 
 	\leq& C\left( \|\Vort_0\|^2 + \| \Th_0\|^2 + \sup_{r \in [0,t]} \| \Th(r)\| \cdot  t  \| f \| \right)\\
 	\leq& C \left( \|\Vort_0\|^2 + \| \Th_0\|^2 +   (\| \Vort_0\| + \| \Th_0 \| +  t\| f\|)  t\| f\| \right),
\end{align*}
implying \eqref{eq:L2:bnd:2}.

Given any multi-index $\alpha$ and taking the associated spatial derivatives of \eqref{eq:b:1}--\eqref{eq:b:2} we obtain
\begin{align*}
  \pd_t \pd^\alpha\Vort + \pd^\alpha(u \cdot \nabla \Vort) - \nu \Delta \pd^\alpha \Vort = g \pd_x \pd^\alpha\Th,  \quad 
  \pd_t \pd^\alpha\Th + \pd^\alpha(u \cdot \nabla \Th) - \kappa \Delta \pd^\alpha\Th =  \pd^\alpha f.
\end{align*}
Multiplying, integrating and summing over $|\alpha| \leq m$ yields
\begin{align}
  &\frac{1}{2} \frac{d}{dt}\|\Vort\|^2_{H^m} + \nu \|\Vort\|^2_{H^{m+1}} 
  	=\sum_{|\alpha| \leq m}    \langle g \pd_x \pd^\alpha\Th -\pd^\alpha(u \cdot \nabla \Vort), \pd^\alpha\Vort  \rangle 
	\label{eq:en:Hs:1}\\
  &\frac{1}{2} \frac{d}{dt} \| \Th \|^2_{H^m} + \kappa \| \Th \|^2_{H^{m+1}}
  	= \sum_{|\alpha| \leq m} \langle \partial^\alpha f -\partial^\alpha(u \cdot \nabla \Th),  \partial^\alpha \Th\rangle 	\label{eq:en:Hs:2}
\end{align}
Taking advantage of the fact that $u$ is divergence free and applying standard interpolation/commutator estimates produces 
for any $m \geq 1$
\begin{align}
 \sum_{|\alpha| \leq m} |\langle \pd^\alpha(u \cdot \nabla \Vort) , \pd^\alpha \Vort  \rangle|
   =& \sum_{|\alpha| \leq m}  |\langle \pd^\alpha(u \cdot \nabla \Vort)  - u \cdot \nabla \pd^\alpha\Vort, \pd^\alpha \Vort  \rangle| \notag\\
   \leq&C \sum_{|\alpha| \leq m}  (\| \pd^\alpha u \|_{L^\infty} \| \nabla \Vort \| 
   	+ \|\nabla \bfU \|_{L^4} \| \pd^\alpha \Vort\|_{L^4}  )\| \Vort \|_{H^m} \notag\\
   \leq& C  \| \nabla \Vort \|  \| \Vort \|_{H^m}^{3/2}  \| \Vort \|_{H^{m+1}}^{1/2} \leq \frac{\nu}{6} \| \Vort \|_{H^{m+1}}^{2}  + C\| \nabla \Vort \|^{4/3}  \| \Vort \|_{H^m}^{2}. 
   \label{eq:cm:be:bnd:1}
\end{align}
Similarly
\begin{align}
  \sum_{|\alpha| \leq m} | \langle \pd^\alpha(u \cdot \nabla \Th) , \pd^\alpha \Th  \rangle|
   =&     \sum_{|\alpha| \leq m} |\langle \pd^\alpha(u \cdot \nabla \Th)  - u \cdot \nabla \pd^\alpha\Th, \pd^\alpha \Th  \rangle| \notag\\
   \leq&C \sum_{|\alpha| \leq m}  (\| \pd^\alpha u \|_{L^\infty} \| \nabla \Th \| 
   	+ \|\nabla \bfU \|_{L^4} \| \pd^\alpha \Th\|_{L^4}  )\| \Th \|_{H^m} \notag\\
   \leq& C( \| \Vort \|_{H^m}^{1/2} \| \Vort\|_{H^{m+1}}^{1/2}\| \nabla \Th \| 	\| \Th \|_{H^m}  + \| \Vort\|^{1/2}  \|\nabla \Vort \|^{1/2} \|\Th\|^{3/2}_{H^m} \|\Th\|^{1/2}_{H^{m+1}})
   \notag\\
   \leq& \frac{\nu}{6} \| \Vort\|_{H^{m+1}}^2 + \frac{\kappa}{4} \| \Th\|_{H^{m+1}}^2 +  C(\| \Vort \|_{H^m}^2 + (\| \nabla \Th \|^2 + \| \nabla \Vort \|^{4/3}) \| \Th \|_{H^m}^2).
      \label{eq:cm:be:bnd:2}
\end{align}
Finally
\begin{align}
	\sum_{|\alpha| \leq m} |\langle g \pd_x \pd^\alpha\Th , \pd^\alpha\Vort  \rangle| 
	\leq \frac{\sqrt{3} g^2}{2\nu} \| \Th \|^2_{H^m}  + \frac{\nu}{6} \| \Vort\|^2_{H^{m+1}}.
	 \label{eq:cm:be:bnd:3}
\end{align}
Combining \eqref{eq:en:Hs:1}, \eqref{eq:en:Hs:2} with the estimates \eqref{eq:cm:be:bnd:1}--\eqref{eq:cm:be:bnd:3} we now obtain
\begin{align}
	&\frac{d}{dt} \left( 1+ \| \Vort \|^2_{H^m} + \frac{\sqrt{3}g^2}{2\nu \kappa} \| \Th \|^2_{H^m}  \right) + \nu \| \Vort\|^2_{H^{m+1}} +  \frac{\sqrt{3}g^2}{\nu} \| \Th\|^2_{H^{m+1}} 
	\notag\\
		&\leq C \| h\|_{H^m}\| \Th\|_{H^m}
			+ C(1+ \| \nabla \Th \|^2 + \| \nabla \Vort \|^{4/3})( \| \Th \|_{H^m}^2 + \| \Vort \|_{H^m}^2 ).
 \label{eq:L2:en:bd}
\end{align}
Thus, taking $X := (1+ \| \Vort \|^2_{H^m} + \frac{\sqrt{3}g^2}{2\nu \kappa} \| \Th \|^2_{H^m})^{1/2}$, 
$Y := (\nu \| \Vort\|^2_{H^{m+1}} +  \frac{\sqrt{3}g^2}{\nu} \| \Th\|^2_{H^{m+1}})^{1/2}$ we have,
\begin{align*}
  \frac{d}{dt} X + C Y \leq C \| h\|_{H^m} + C(1+ \| \nabla \Th \|^2 + \| \nabla \Vort \|^{2}) X.
\end{align*}
With this bound and \eqref{eq:L2:bnd:2} we now infer infer \eqref{eq:Hs:bnd}.

We turn next to establish the continuous dependence estimates Let $\xi = \Vort - \tilde{\Vort}$, $\zeta = \Th- \tilde{\Th}$, $\phi = f - \tilde{f}$.  Then $(\xi, \zeta)$ satisfy
\begin{align}
	&\pd_t \xi  + \tilde{u} \cdot \nabla \xi + (K \ast \xi) \cdot \nabla \Vort- \nu \Delta \xi = g \pd_x \zeta, \quad
	\pd_t \zeta  + \tilde{u} \cdot \nabla \zeta + (K \ast \xi)   \cdot \nabla \Th- \kappa \Delta \zeta =  \phi.
	\label{eq:diff:sol:1}
\end{align}
Start with the $L^2$ based estimates
\begin{align}
  \frac{d}{dt} \| \xi \|^2  +  \nu \| \nabla \xi \|^2 
  	=&   \langle  g \pd_x \zeta - (K \ast \xi) \cdot \nabla \Vort, \xi \rangle \notag\\
	\leq& \frac{g^2}{\nu} \| \zeta \|^2  + \frac{\nu}{4} \| \nabla \xi\|^2 + \| \nabla \Vort\| \| K \ast \xi\|_{L^\infty} \| \xi \| \notag\\
	\leq& \frac{g^2}{\nu} \| \zeta \|^2  + \frac{\nu}{4} \| \nabla \xi\|^2 + C \| \nabla \Vort\| \| \nabla \xi\|^{1/2} \| \xi\|^{3/2} \notag\\
	\leq& \frac{g^2}{\nu} \| \zeta \|^2  + \frac{\nu}{2} \| \nabla \xi\|^2 + C \| \nabla \Vort\|^{4/3} \| \xi\|^{2}
	\label{eq:d:est:1}
\end{align}
where we used Agmond's inequality for the penultimate estimate.  Similarly
\begin{align}	
 \frac{d}{dt} \| \zeta \|^2  +  \kappa \| \nabla \zeta \|^2
	=&  	\langle \phi- (K \ast \xi)   \cdot \nabla \Th  ,  \zeta \rangle \notag\\
	\leq& C \|\phi\|\| \zeta \| +  \| \nabla \Th\| \| \xi\|^{1/2} \| \nabla \xi \|^{1/2} \| \zeta \| \notag\\
	\leq& C \|\phi \|\| \zeta \| +  \frac{\nu}{2}\| \nabla \xi \|^2  + C\| \nabla \Th\|^2 \| \zeta \|^2.
	\label{eq:d:est:2}
\end{align}
Combining the estimates \eqref{eq:d:est:1}, \eqref{eq:d:est:2} we obtain the bound
\begin{align}
  \frac{d}{dt} (\| \xi \|^2 + \| \zeta \|^2) \leq C \| \phi\|^2 + C(1 +\| \nabla \Th\|^2  + \| \nabla \Vort\|^2 )\| \zeta \|^2.
  \label{eq:d:est:L2:f}
\end{align}

We turn to make the continuous dependence estimates in higher Sobolev norms.
Applying $\pd^\alpha$ for any multi-index $\alpha$ and summing over all $|\alpha| \leq m$ for any $m \geq 1$, we find that
\begin{align}
  &\frac{d}{dt} \| \xi \|^2_{H^m}  +  \nu \| \xi \|^2_{H^{m+1}} 
  	= \sum_{|\alpha| \leq m}  \langle \pd^\alpha(  g \pd_x \zeta - \tilde{u} \cdot \nabla \xi - (K \ast \xi) \cdot \nabla \Vort), \pd^\alpha \xi \rangle := I_1,
	\label{eq:d:bal:hs:1}\\
 &\frac{d}{dt} \| \zeta \|^2_{H^m}  +  \kappa \| \zeta \|^2_{H^{m+1}} 
	= \sum_{|\alpha| \leq m}  \langle \pd^\alpha (\phi - \tilde{u} \cdot \nabla \zeta - (K \ast \xi)   \cdot \nabla \Th ) ,  \pd^\alpha \zeta \rangle := I_2.
		\label{eq:d:bal:hs:2}
\end{align}
Regarding $I_1$ we have
\begin{align}
   |I_1| \leq&  C \| \zeta\|_{H^m}^2 + \frac{\nu}{2} \| \xi \|^2_{H^{m+1}} + C \|  \xi \|^2_{H^m} \! \sum_{|\alpha|\leq m}  \| \pd^\alpha \tilde{u} \|_{L^\infty} \notag\\
   		&+ C \|  \xi \|_{H^m} ( \| \nabla \Vort\|_{L^4}\! \sum_{|\alpha|\leq m} \| \pd^\alpha (K \ast \xi) \|_{L^4} +\|  \Vort\|_{H^{m+1}} \|K \ast \xi \|_{L^\infty}  )
		\notag\\
	\leq& C \| \zeta\|_{H^m}^2  + \frac{\nu}{2} \| \xi \|^2_{H^{m+1}} + C(\| \tilde{\Vort} \|_{H^{m+1}} +\| \Vort \|_{H^{m+1}}) \| \xi \|^2_{H^m}
	\label{eq:d:est:3}
\end{align}
For $I_2$
\begin{align}
  |I_2| \leq& \| \phi\|_{H^m} \|\zeta\|_{H^m} + C \| \zeta \|^2_{H^m} \! \sum_{|\alpha|\leq m}  \| \pd^\alpha \tilde{u} \|_{L^\infty} \notag\\
  &+ C \|  \zeta \|_{H^m} ( \| \nabla \Th\|_{L^4}\! \sum_{|\alpha|\leq m} \| \pd^\alpha (K \ast \xi) \|_{L^4} +\|  \Th\|_{H^{m+1}} \|K \ast \xi \|_{L^\infty}  )
  \notag\\
  \leq&\| \phi\|_{H^m} \|\zeta\|_{H^m} + C\|\tilde{\Vort}\|_{H^{m+1}} \| \zeta \|^2_{H^m}  + C \|\Th\|_{H^{m+1}} \|  \zeta \|_{H^m} \|\xi\|_{H^m}.  
  \label{eq:d:est:4}
\end{align}
Combining these estimates we conclude that
\begin{align*}
& \frac{d}{dt} (\| \xi \|^2_{H^m} +  \| \zeta \|^2_{H^m}) 
 	\\
	&\leq  \| \phi\|_{H^m} \|\zeta\|_{H^m}+  C(1 + \| \tilde{\Vort} \|_{H^{m+1}} +\| \Vort \|_{H^{m+1}} +  \|\Th\|_{H^{m+1}})(\| \xi \|^2_{H^m} +  \| \zeta \|^2_{H^m}).   
\end{align*}
\end{proof}

\subsection{Euler Equations}
\label{sec:NSE:apriori}

Proposition~\ref{prop:NSEsemi} follows immediately once we establish the following result.

\begin{proposition}
\label{prop:apriori}
Fix any $\bfg \in \mathcal{X}$ and any finite-dimensional subspace $X_0 \subset \mathcal{X}$.    
\begin{itemize}
\item[(i)]  For any $\bfU_0 \in \mathcal{X}$ and any $\bfh\in X_0$, there exists a unique $0<T_{\bfU_0, \bfh} \leq \infty$ and $\bfU(\cdot)=\bfU(\cdot, \bfU_0, \bfh) \in C([0, T_{\bfU_0, \bfh}); \mathcal{X})$ solving~\eqref{eq:NSE} such that if $T_{\bfU_0, \bfh}< \infty$ then 
\begin{align*}
\limsup_{t\uparrow T_{\bfU_0, \bfh}} \| \nabla \bfU(t) \|_{L^\infty} = \infty. 
\end{align*}
\item[(ii)]  For any $\bfU_0 \in \mathcal{X}$ and any $\bfh\in X_0$, let $$\tau_{\bfU_0,\bfh}^n= \inf\{ t>0 \, : \, \|\bfU(t)\|_{H^3} \geq n\} \,\,\, \text{ and } \,\,\, \tau_{\bfU_0, \bfh}= \sup_{n\in \N} \tau_{\bfU_0, \bfh}^n.$$ 
Then $\tau_{\bfU_0, \bfh}>0$ and $\tau_{\bfU_0, \bfh}\leq T_{\bfU_0, \bfh}$.  Moreover for all $m\geq 3$, $t < \tau_{\bfU_0, \bfh}^n$ and $n\in \N$ we have the estimate
\begin{align*}
\|\bfU(t) \|_{H^m}^2 \leq \| \bfU_0 \|_{H^m} e^{C (n+1) t} + \int_0^t C e^{C (n+1)(t-s) }  \| \bfg + \bfh  \|_{H^m}\, ds  
\end{align*}  
for some constant $C$ depending only on $m$.   
\item[(iii)]  Let $\bfU_1(0), \bfU_2(0)\in \mathcal{X}$, $\bfh_1, \bfh_2 \in X_0$ and $\bfU_1(t, \bfU_1(0), \bfh_1)$ and $\bfU_2(t)= \bfU(t, \bfU_2(0), \bfh_2)$.  Let $n, T>0$.  Then for all $t<\tau_{\bfU_1(0), \bfh_1}^n \wedge \tau^n_{\bfU_2(0), \bfh_2}$ there exists a constant $C$ depending only on $m$ and a constant $D>0$ depending only on $m, T, \| \bfU_2(0)\|_{H^m}, \| \bfU_1(0)\|_{H^{m+1}}, \| \bfg + \bfh_2\|_{H^m} , \| \bfg + \bfh_1 \|_{H^{m+1}}$ such that 
\begin{align*}
\| \bfU_1(t)-\bfU_2(t) \|_{H^m}^2 \leq \|\bfU_1(0)-\bfU_2(0)\|_{H^m}^2 e^{D t}  + C_m \int_0^t e^{D(t-s)} \|\bfh_1 - \bfh_2 \|_{H^m}^2 \, ds.  
\end{align*} 
    
\end{itemize}   
\end{proposition}
\begin{proof}[Proof of Proposition~\ref{prop:apriori}]
For the proof of (i), see~\cite{MajdaBertozzi2002, MarchioroPulvirenti2012}.  To see (ii), first note that for $\bfh\in X_0$ and $\bfU_0 \in \mathcal{X}$, the fact that $\tau_{\bfU_0, \bfh}>0$ and $\tau_{\bfU_0, \bfh} \leq T_{\bfU_0, \bfh}$ follow from (i) and the Gagliardo-Nirenberg inequality.  To obtain the claimed estimate, let $\bff= \bfg+ \bfh$ and observe that for all multi-indices $\beta$ with $|\beta|\leq m$, $m\geq 3$, we have the estimate 
\begin{align*}
\frac{1}{2}\frac{d}{dt}\| \partial^\beta \bfU(t)\|^2 &= \langle \partial^\beta \bfU(t), \partial^\beta \bff  \rangle  -  \langle \partial^\beta \bfU(t), \partial^\beta B(\bfU(t), \bfU(t))\rangle\\
& \leq  \| \bfU(t) \|_{H^m} \| \bff\|_{H^m} -  \langle \partial^\beta \bfU(t), \partial^\beta B(\bfU(t), \bfU(t)) - B(\bfU(t), \partial^\beta \bfU(t))\rangle
\end{align*}   
where in the inequality we used the fact that $\langle \partial^\beta \bfU(t), B(\bfU(t), \partial^\beta \bfU(t))\rangle=0$ as $\bfU(t)$ is divergence-free.  To estimate the contribution from the nonlinear term, we first observe that by interpolation and Agmon's inequality
\begin{align*}
\| \partial^\beta B(\bfU(t), \bfU(t)) - B(\bfU(t), \partial^\beta \bfU(t)) \| \leq c_{m} \| \bfU(t)\|_{W^{1, \infty}} \| \bfU(t) \|_{H^m} \leq c_{m}' \|\bfU(t) \|_{H^3} \|\bfU(t)\|_{H^m}
\end{align*}
as $m\geq 3$, where $c_{m}, c'_{m}$ are constants depending only on $m$.  Putting these estimates together, we find that   
\begin{align*}
\frac{1}{2}\frac{d}{dt}\| \partial^\beta \bfU(t)\|^2 &\leq  \| \bfU(t) \|_{H^m} \|\bff \|_{H^m} + c'_{m} \|\bfU(t)\|_{H^3} \| \bfU(t)\|_{H^m}^2.  
\end{align*}    
Summing over all multi-indices $\beta$ with $|\beta| \leq m$ and using Young's inequality produces 
\begin{align*}
\frac{1}{C_{m}} \frac{d}{dt}\| \bfU(t) \|_{H^m}^2 \leq \| \bff \|_{H^m}^2 + (1 + \| \bfU(t)\|_{H^3}) \| \bfU(t) \|_{H^m}^2
\end{align*}
for some constant $C_{m}$ depending only on $m$.  Supposing that $t< \tau^n_{\bfU_0, \bfh}$, Gronwall's inequality then implies the claimed estimate in (ii).  

To prove (ii), let $\bfw(t) = \bfU_1(t)- \bfU_2(t)$.  Then for $m\geq 3$ and any multi-index $\beta$ with $|\beta| \leq m$ we have the estimate
\begin{align*}
\frac{1}{2}\frac{d}{dt}\| \partial^\beta \bfw(t) \|^2&= \langle \partial^\beta \bfw(t), \partial^\beta( \bfh_1- \bfh_2)\rangle + \langle \partial^\beta \bfw(t), \partial^\beta B(\bfU_2(t), \bfU_2(t)) - \partial^\beta B(\bfU_1(t), \bfU_1(t))\rangle\\
&\leq \| \bfw(t) \|_{H^m} \| \bfh_1-\bfh_2 \|_{H^m}  + \langle \partial^\beta \bfw(t), \partial^\beta B(\bfU_2(t), \bfU_2(t)) - \partial^\beta B(\bfU_1(t), \bfU_1(t))\rangle\\
&= \| \bfw(t) \|_{H^m} \|\bfh_1-\bfh_2 \|_{H^m}  \\
&\qquad -  \langle \partial^\beta \bfw(t), \partial^\beta B(\bfw(t), \bfU_1(t)) + \partial^\beta B(\bfU_2(t), \bfw(t))- B(\bfU_2(t), \partial^\beta \bfw(t))\rangle
\end{align*} 
where again we used the fact that $\bfU_2$ is divergence-free as $\langle \partial^\beta \bfw(t), B(\bfU_2(t), \partial^\beta \bfw(t))\rangle=0.$  Note by interpolation
\begin{align*}
&|    \langle \partial^\beta \bfw(t), \partial^\beta B(\bfw(t), \bfU_1(t)) + \partial^\beta B(\bfU_2(t), \bfw(t))- B(\bfU_2(t), \partial^\beta \bfw(t))\rangle
|\\
&\leq c_{m} (\| \bfw(t)\|_{H^m}^2 \| \bfU_1(t) \|_{H^{m+1}} +  \|\bfw(t)\|_{H^m}^2 \|\bfU_2(t)\|_{H^m}).  
\end{align*}
for some constant $c_{m}$ depending only on $m$.  Thus combining this inequality with the previous, summing over all multi-indices $\beta$ with $|\beta| \leq m$ and applying Young's inequality produces the following bound \begin{align*}
\frac{1}{C}\frac{d}{dt}\|\bfw(t) \|_{H^m}^2 \leq \|\bfh_1 - \bfh_2\|_{H^m}^2 + \|\bfw(t)\|_{H^m}^2 ( \|\bfU_1(t)\|_{H^{m+1}} + \|\bfU_2(t) \|_{H^m}+1)  
\end{align*}
for some constant $C>0$ depending only on $m$.  Now for any $T>0$ if $t< \tau_{\bfU_1(0), \bfh_1}^n\wedge \tau_{\bfU_2(0), \bfh_2}^n\wedge T$, by the estimate in (ii) and Gronwall's inequality there exists a constant $D>0$ depending only on $m, T, \| \bfU_2(0)\|_{H^m}, \| \bfU_1(0)\|_{H^{m+1}}, \| \bfg + \bfh_2\|_{H^m} , \| \bfg + \bfh_1 \|_{H^{m+1}}$ such that 
\begin{align*}
\| \bfw(t) \|_{H^m}^2 \leq \|\bfw(0)\|_{H^m}^2 e^{D t}  + C \int_0^t e^{D(t-s)} \|\bfh_1 - \bfh_2 \|_{H^m}^2 \, ds.  
\end{align*} 
This finishes the proof of the estimate in (iii).  
\end{proof}

\section{Comparison Theorem}

For the estimates in Section~\ref{sec:examples}, we make repeated use of the following comparison principal.  
\begin{proposition}
\label{prop:odecomp}
Let $f:\R\rightarrow \R$ be locally Lipschitz continuous.  Fix $0 < T \leq \infty$ 
and suppose that $\phi:[0, T) \rightarrow [0, \infty)$ is continuous and satisfies  
\begin{align*}
\phi(t) = \phi(s) + \int_s^t f(\phi(u)) \, du 
\end{align*}
for all $0\leq s\leq t < T$.  On the other hand suppose that for some $0 < S \leq \infty$, $\psi:[0, S)\rightarrow [0, \infty)$ is continuous with $\psi(0) = \phi(0)$,
\begin{align*}
  \limsup_{t \to S} \psi(t) = \infty
\end{align*}
and 
\begin{align*}
\psi(t) \leq \psi(s) + \int_s^t f(\psi(u)) \, du
\end{align*}  
for all $0 \leq s\leq t < T \wedge S$.  Then $S \geq T$ and $\psi(t) \leq \phi(t)$ for all $0\leq t \leq T$.
\end{proposition}

In particular, we will leverage this proposition for the estimates above in the form of the following corollary.  
\begin{corollary}\label{lem:comp:lem:app}
Let $T>0$.  Suppose that for every $\lambda > 0$, there exists a $T_\lambda \in (0,\infty]$ and a $C^1$-function
$x_\lambda : [0,T_\lambda) \to [0,\infty)$ satisfying 
\begin{align}
  \frac{d x_\lambda}{dt} \leq \frac{c_0}{\lambda}(x_\lambda^p + \kappa_0)
  \text{ on } [0, T\wedge T_\lambda) \qquad \,\text{ and } \, \qquad  \limsup_{t \to T_\lambda} x_\lambda(t) = \infty,
  \label{eq:loc:ineq}
\end{align}
where $c_0, \kappa_0> 0$ and $p > 1$ are constants independent of $\lambda >0$.  For $\gamma, \lambda >0$ and $t\geq 0$, define 
\begin{align}
  T^*_{\lambda}(\gamma)  = \frac{\lambda}{2 c_0 (p-1)\gamma^{p-1}}  
\quad \text{ and } \quad
 R_\lambda(t, \gamma) = \left(1 - \frac{2 c_0 (p-1) \gamma^{p-1} }{\lambda} \,  t \right)^{-\frac{1}{p-1}}.  
  \label{eq:blow:cntrl}
\end{align}
Then for all $0 \leq t \leq T_\lambda^*(x_\lambda(0) + \kappa_0) \wedge T$ we have  
\begin{align}
  x_\lambda(t) \leq x_\lambda(0) R_\lambda(t, x_\lambda(0) + \kappa_0) 
   + \kappa_0( R_\lambda(t,x_\lambda(0) + \kappa_0 ) -1)
  \label{eq:loc:bnd}
\end{align}
\end{corollary}

\begin{remark}
\label{rem:comp2}
Observe that if $x_\lambda(0)=x_0 \geq 0$ is independent of $\lambda>0$, then the comparison~\eqref{eq:loc:bnd} holds for all $t\in [0, T]$ and all $\lambda \geq 2 c_0 T(p-1) (x_0 + \kappa_0)^{p-1}$.  
\end{remark}

Let us first prove Corollary~\ref{lem:comp:lem:app} using Proposition~\ref{prop:odecomp} 
and then establish the Proposition thereafter.

\begin{proof}[Proof of Corollary~\ref{lem:comp:lem:app}]
Under the given conditions on $x_\lambda$ notice that 
\begin{align*}
  \frac{d (x_\lambda + \kappa_0)}{dt} \leq \frac{2c_0}{\lambda}(x_\lambda + \kappa_0)^p 
\end{align*}
Now consider $y$ solving
\begin{align*}
 \frac{d y}{dt} = \frac{2c_0}{\lambda}y^p \quad y(0) = y_0.
\end{align*}
When $y_0 \geq 0$, this equation has the unique solution 
\begin{align*}
  y(t, y_0) := y_0 \left(1 - t \, \frac{2c_0 (p-1)}{\lambda } \,  y_0^{p-1} \right)^{-\frac{1}{p-1}}.
\end{align*}
defined on the interval $[0,\frac{\lambda}{2c_0 (p-1) y_0^{p-1}})$.  Thus, by comparing 
$y(\cdot, x_\lambda(0) + \kappa_0)$ to $x_\lambda + \kappa_0$, we obtain the desired result by invoking
Proposition~\ref{prop:odecomp}.
\end{proof}

\begin{proof}[Proof of Proposition~\ref{prop:odecomp}]
We first show that $\psi$ remains below $\phi$ on their common interval of definition.  
Let $R < T \wedge S$ and define
\begin{align}
  T_0 := \inf_{ t \in [0, R)} \{ \psi(t) > \phi(t) \} \wedge R.  
\end{align}
Let us show that $T_0 = R$.  If not, then there exist times $T_0 \leq T_1 < T_2 < R$
such that
\begin{align*}
  \psi(T_1) = \phi(T_1) \text{ and } \psi(t) > \phi(t) \text{ for every } T_1 < t \leq T_2.
\end{align*}
Take
\begin{align*}
  K= \{\phi(t) \, : \, t \in [T_1, T_2]\}\cup \{ \psi(t) \, : \, t \in [T_1, T_2] \}.
\end{align*}
By the continuity of $\phi$ and $\psi$, $K$ is compact and since $f$ is locally Lipshitz, there exists a constant $C_K>0$ such that 
\begin{align*}
|f(u) - f(v) | \leq C_K|u-v| \,\,\text{ for all } \,\, u,v \in K. 
\end{align*}
Now, for $T_1 < t \leq T_2$,
\begin{align*}
  0 < \psi(t)-\phi(t) \leq \int_{T_1}^t f(\psi(r)) - f(\phi(r)) \, dr \leq C_K \int_{T_1}^t \psi(r) - \phi(r) \, du.
\end{align*} 
Invoking Gr\"onwall's inequality, we have that $\psi(t) = \phi(t) = 0$ for $t \in [T_1, T_2]$, a contridiction.

To show that $T \geq S$ we again argue by contridiction and suppose on the contrary that $S < T$. Take 
\begin{align*}
  S_n = \inf_{t \in [0,S)} \{ \psi(t) > n \}.
\end{align*}
Then, by what we have already established,  $\phi(S_n) \geq \psi(S_n) = n$.  This in turn would imply that
$\sup_{t \in [0,S]} \phi(t) = \infty$, violating the continuity of $\phi$ and yielding the desired contridiction.  The proof is 
complete.
\end{proof}

\vspace{.3in}
\begin{multicols}{2}
\noindent
Nathan E. Glatt-Holtz\\ {\footnotesize
Department of Mathematics\\
Tulane University\\
Web: \url{http://www.math.tulane.edu/~negh/}\\
 Email: \url{negh@tulane.edu}} \\[.35cm]
David P. Herzog\\
{\footnotesize Department of Mathematics \\
Iowa State University\\
Web: \url{http://orion.math.iastate.edu/dherzog/}\\
 Email: \url{dherzog@iastate.edu}} \\[.35cm]
\columnbreak

 \noindent Jonathan C. Mattingly\\
{\footnotesize
Department of Mathematics \\
Duke University\\
Web: \url{https://services.math.duke.edu/~jonm/}\\
 Email: \url{jonm@math.duke.edu}}\\[.35cm]
\columnbreak
 \end{multicols}
\end{document}